\documentclass[11pt, reqno]{amsart}
\usepackage{amsmath, amsthm, amscd, amsfonts, amssymb, graphicx, xcolor, dsfont}
\usepackage{float}
\usepackage{tikz}
\usepackage{cases}
\usepackage{marginnote}
\usetikzlibrary{backgrounds}
\usetikzlibrary{patterns,fadings}
\usetikzlibrary{arrows,decorations.pathmorphing}
\usetikzlibrary{calc}
\definecolor{light-gray}{gray}{0.95}
\usepackage{float}
\usepackage[colorlinks=true,linkcolor=blue,citecolor=magenta]{hyperref}
\usepackage{comment}
\usepackage{charter}
\let\oldtocsection=\tocsection
\let\oldtocsubsection=\tocsubsection
\let\oldtocsubsubsection=\tocsubsubsection
\renewcommand{\tocsection}[2]{\hspace{0em}\oldtocsection{#1}{#2}}
\renewcommand{\tocsubsection}[2]{\hspace{1em}\oldtocsubsection{#1}{#2}}
\renewcommand{\tocsubsubsection}[2]{\hspace{2em}\oldtocsubsubsection{#1}{#2}}
\DeclareRobustCommand{\SkipTocEntry}[5]{}

\def\centerarc[#1](#2)(#3:#4:#5){\draw[#1] ($(#2)+({#5*cos(#3)},{#5*sin(#3)})$) arc (#3:#4:#5);}

\newcommand{\mc}[1]{{\mathcal #1}}
\newcommand{\mf}[1]{{\mathfrak #1}}

\newcommand{\bb}[1]{{\mathbb #1}}

\newcommand{\<}{\langle}
\renewcommand{\>}{\rangle}
\newcommand{\p}{\partial}
\newcommand{\eps}{\varepsilon}
\newcommand{\pfrac}[2]{\genfrac{}{}{}{1}{#1}{#2}}

\newcommand{\msf}[1]{{\mathsf #1}}

\newcommand{\Sobolev}{L^2([0,1];\mc H^1(0,1))}
\usepackage{graphicx}

\makeatletter
\newcommand*\bigcdot{\mathpalette\bigcdot@{.5}}
\newcommand*\bigcdot@[2]{\mathbin{\vcenter{\hbox{\scalebox{#2}{$\m@th#1\bullet$}}}}}
\makeatother

\newtheorem{theorem}{Theorem}[section]
\newtheorem{lemma}[theorem]{Lemma}
\newtheorem{proposition}[theorem]{Proposition}
\newtheorem{corollary}[theorem]{Corollary}
\theoremstyle{definition}
\newtheorem{definition}[theorem]{Definition}

\theoremstyle{remark}
\newtheorem{remark}[theorem]{Remark}
\numberwithin{equation}{section}

\newif\ifwip
\wiptrue 

\ifwip

\usepackage{marginnote}
\newcommand{\usernote}[3]{
	\marginnote{\raggedright\footnotesize \textcolor{#2}{#1: #3}}
}
\newcommand{\mfnote}[1]{\usernote{Matheus}{cyan}{#1}}
\newcommand{\pgnote}[1]{\usernote{Patrícia}{red}{#1}}
\newcommand{\tfnote}[1]{\usernote{Tertu}{orange}{#1}}

\newcommand{\reviewdel}[1]{}

\newcommand{\freviewdel}[1]{\textcolor{red}{\st{#1}}}

\else

\newcommand{\mfnote}[1]{}
\newcommand{\pgnote}[1]{}
\newcommand{\tfnote}[1]{}

\newcommand{\reviewdel}[1]{}

\newcommand{\freviewdel}[1]{}

\fi

\begin{document}
\setcounter{page}{1}

\centerline{}

\centerline{}

\title[Finite Reservoirs lead to Wentzell Boundary Conditions]{Finite reservoirs lead to Wentzell boundary conditions\\ for independent random walks and exclusion processes}

\author[M. Franco, T. Franco, P. Gonçalves]{Matheus Franco$^1$, Tertuliano Franco$^2$ and Patrícia Gonçalves$^3$}

\address{$^{1}$ Center for Mathematical Analysis,  Geometry and Dynamical Systems,
	Instituto Superior T\'ecnico, Universidade de Lisboa,
	Av. Rovisco Pais, 1049-001 Lisboa, Portugal}
\email{\textcolor[rgb]{0.00,0.00,0.84}{matheusguilherme99@tecnico.ulisboa.pt}}

\address{$^{2}$ UFBA, Instituto de Matem\'atica e Estat\'istica, Campus de Ondina, Av. Milton Santos, S/N. CEP 40170-110. Salvador, Brazil}
\email{\textcolor[rgb]{0.00,0.00,0.84}{tertu@ufba.br}}

\address{$^{3}$ Center for Mathematical Analysis,  Geometry and Dynamical Systems,
	Instituto Superior T\'ecnico, Universidade de Lisboa,
	Av. Rovisco Pais, 1049-001 Lisboa, Portugal}
\email{\textcolor[rgb]{0.00,0.00,0.84}{pgoncalves@tecnico.ulisboa.pt}}


\subjclass[2020]{Primary 60K35; Secondary 35K20}

\keywords{Exclusion process, independent random walks, hydrodynamic limit, propagation of local equilibrium, Wentzell boundary conditions, non-local Dirichlet boundary conditions, non-linear Dirichlet boundary conditions}


\begin{abstract}
	We analyze the scaling limits (hydrodynamics/propagation of local equilibrium) of two particle systems, namely independent random walks (IRW) and symmetric exclusion processes (SEP), both in the discrete one-dimensional segment where the left boundary is in contact with a reservoir, which may hold any number of particles. 
	At rate one a particle  jumps from site $1$ to the reservoir, and at rate $\alpha \eta(0) N^{-\theta}$ a particle jumps from the reservoir to site $1$ (if site $1$ is empty in SEP), where $\eta(0)$ is the total number of particles in the reservoir and $\theta\geq 0$ is a parameter  whose tuning leads to a dynamical phase transition. 
	For every $\theta\geq 0$, the hydrodynamic equation is the heat equation with Neumann boundary conditions at the right boundary for both processes, while the boundary  conditions at the left boundary depend on the value of $\theta$.  For $0\leq \theta<1$, it is given by Neumann boundary conditions, meaning the deposit is asymptotically empty, acting as a barrier. For $\theta>1$, and for IRW  (resp. SEP) it is given by a non-homogeneous (resp. homogeneous) Dirichlet boundary condition, meaning  the reservoir becomes asymptotically infinite, acting as a heat bath (resp. the reservoir behaves as a sink). For $\theta=1$, we obtain a  non-local Dirichlet boundary condition, which is additionally non-linear for SEP. As a by-product, we obtain an  equivalence between solutions to the heat equation with Wentzell boundary conditions and with non-local Dirichlet boundary conditions related to the total mass of the system. 
\end{abstract}

\maketitle

\tableofcontents

\allowdisplaybreaks

\section{Introduction}
Wentzell boundary conditions have been  introduced independently by\break Wentzell in \cite{Wentzell1956Semigroups,Venttsel1959} and by Feller in \cite{Feller1952,Feller1954}, both in the context of 
second order generalized differential operators, where they also established connections with Markov processes. 
A Wentzell boundary condition, also named a Feller-Wentzell boundary condition, is a boundary condition for a partial differential equation (PDE) that relates the second derivative, and possibly the first derivative and/or the function itself. In recent years, this kind of boundary condition has attracted a great deal of attention in the PDE community; see, for instance, \cite{Boccardo1996,Daners2008, Favini2002,Favini2004, Liang2004, Luo_Trudinger_1991, Vasquez, Warma2005Parabolic, Warma2005Trace}. For a review on the subject and  its physical interpretations, we refer the reader to~\cite{Goldstein2006}.

Since the 1980s, rigorous methods for deriving  scaling limits of interacting particle systems (IPS) have been extensively developed, enabling a mathematical description of how microscopic interactions give rise to macroscopic behavior, see e.g. \cite{kl} and \cite{DeMasi_Presutti}. 
Over the last two decades, many papers in this area have dealt with  IPS whose hydrodynamics  is described by a PDE supplemented with various boundary conditions. As examples, we cite \cite{BaldassoMenezesNeumann2017,bertini_landim_mourragui,BouleyLandim2024,efgnt,FrancoGoncalvesLandimNeumann2022, fgn1, fgn2, fgn3,fgschutz2015,FN,FrancoTavares2019,DeMasiPresuttiTsagkarogiannisVares2011,Mourragui_mixed}. The boundary conditions of the parabolic equations obtained in those papers vary among Dirichlet (involving the function itself), Neumann (involving the first derivative), and Robin (relating the function and its first derivative). However, to the best of our knowledge, no one has ever obtained a parabolic equation with Wentzell boundary conditions as the hydrodynamic equation of an  IPS. In this work, we fill this gap in the literature by deriving the  heat equation with Wentzell boundary conditions as the scaling limit of two different discrete systems of IPS. As a by-product of this scaling procedure, we obtain a correspondence between parabolic PDEs with Wentzell boundary conditions and parabolic PDEs with non-local Dirichlet boundary conditions, the latter relating the value of the function at the boundary to its total mass.

\subsection{Independent Random Walks}

Fix a scaling parameter $N$. We consider two IPS evolving under diffusive time scaling and in the finite box $\{0,1,\ldots, N\}$. The first is a system of IRW,  where the jump rate to a neighboring site is always equal to one, except when a jump occurs from  site $0$ to  site $1$, whose associated jump rate  is given by $\alpha N^{-\theta}$, where $\alpha>0$ and 
$\theta\geq 0$. 

For this model, we prove two results. The first is the propagation of local equilibrium. Roughly speaking, this means that, starting the system from a proper measure associated to a continuous profile $\gamma$, the system converges locally to a product measure parametrized by a parameter $\rho(t,u)$ which solves a PDE -- namely, the heat equation supplemented with boundary conditions that  depend on the  value of $\theta$, with initial condition given by $\gamma$.

For $\theta<1$, we obtain  Neumann boundary conditions on the left boundary, while for  $\theta>1$, these are replaced by homogeneous Wentzell boundary conditions given by $\rho''(0)=0$.
Finally, for the critical value $\theta=1$, we obtain  Wentzell boundary conditions of the form $\rho''(0)=\alpha \rho'(0)$. Since the box is isolated at the rightmost site, the right boundary condition is always of homogeneous Neumann type, so we do not mention the right boundary condition from here on.

In the sequel, we prove the hydrodynamic limit for that model, which undergoes a dynamical phase transition. For $\theta<1$, the strength of site zero is not enough to have a macroscopic effect, and the boundary conditions are of Neumann type. For $\theta>1$, the reservoir has a small exit rate of order $N^{-\theta}$, but simultaneously it accommodates a large number of particles of order $N^{\theta}$. The system's density stabilizes close to the reservoir, and in this case the hydrodynamic equation has (local) Dirichlet boundary conditions.
Finally, for the critical value $\theta=1$,  we obtain  non-local Dirichlet boundary conditions involving the total mass of the system. Since the propagation of local equilibrium implies the hydrodynamic limit, as a by-product of these two previous results, we obtain an interesting correspondence between the PDEs above, which can  also be corroborated  directly through a simple heuristic argument at the macroscopic level, see Remark~\ref{remark:2.8}.

\subsection{Symmetric Exclusion Process}
The second IPS is the SEP, where the site $1$ is in contact with a finite reservoir at site zero. The sites from $1$ to $N$ can hold  at most one particle,  but the reservoir can store any finite number of particles. At rate one a particle jumps from site $1$ to site $0$, and at rate $\alpha N^{-\theta}\eta(0)$ a particle jumps from the reservoir to site $1$, but only if the latter is empty, otherwise nothing happens. Here, $\eta(0)$ is the number of particles in the reservoir before the jump. Analogously to IRW, we obtain a hydrodynamic phase transition in the sense that the hydrodynamic equation, the heat equation, is supplemented with different boundary conditions that depend on the value of $\theta$. If $\theta<1$, the reservoir is not strong enough to retain mass, so the hydrodynamic equation is given by the heat equation with  Neumann boundary conditions. If $\theta>1$, the small rate of order $N^{-\theta}$ of the incoming particles  from the reservoir plus the exclusion rule at  site $x=1$ leads to a homogeneous Dirichlet boundary condition $\rho(0)=0$, which means that the density vanishes near the reservoir. Or, equivalently, that the reservoir behaves as a sink. At the critical value $\theta=1$, the hydrodynamic equation is given by the heat equation with a non-linear non-local Dirichlet boundary condition. Furthermore, via the correspondence deduced in the previous case about IRW, we present an equivalent PDE with a non-linear Wentzell boundary for this hydrodynamic equation. \\

As the main features in the proofs, we highlight the understanding of the propagation of local equilibrium in terms of a central limit theorem for the underlying random walk and its reversible measure. In terms of the hydrodynamic limit in the exclusion setting,  we cite: a characterization of the reservoir  in terms of the total mass of the system;  attractiveness, which is  used throughout the text;  and two apparently contradictory facts:  site zero is a trap which retains many particles, yet it does not contribute to the limit of the empirical measure in any range of $\theta$. On the other hand, its presence in the empirical measure helps  in many steps of the proofs. Characterization of reversible measures is provided in all cases, as well as the uniqueness of weak solutions of the corresponding PDE, whose proofs are based on the knowledge of the eigenfunctions of the corresponding Sturm-Liouville problem. We highlight that each regime  requires sharp arguments that only hold for the corresponding  range of~$\theta$. For instance, the set of test functions in each case had to be carefully chosen: if that set of test functions is too large, the uniqueness of weak solutions of the corresponding PDE becomes easier, but the convergence of Dynkin's martingale becomes intractable, and the opposite occurs if the set of test functions is too small.

We point out that the problem addressed is a type of Bouchaud trap model, where a single trap (the finite reservoir) is strong enough to modify  the entire  macroscopic behavior of the system. It is worth mentioning two closely related articles. 
We first mention the paper \cite{jaralandimteixeira} by Jara, Landim and Teixeira, which studies the hydrodynamic limit of the IRW, where the waiting time parameter at a site $x$ in the discrete torus $\bb T_N$ is determined by a measure $W$. There, in dimension one, the hydrodynamic equation is given by a parabolic equation involving the Krein--Feller operator $\frac{d}{dW}\frac{d}{dx}$ associated with the degenerate diffusion obtained in \cite{Fontes_Isopi_Newman} as the scaling limit of the random walk in the trap environment. Although \cite{jaralandimteixeira} is set on the torus, whereas our model evolves on the segment, the result we obtain here for the critical parameter $\theta=1$ for IRW is essentially connected to \cite[Theorem~2.2]{jaralandimteixeira}. Taking, as in \cite{jaralandimteixeira}, the measure $W$ to be Lebesgue measure plus a Dirac delta at the origin, and $\gamma=1$ in the hypotheses of that theorem, our result yields a classical description of the hydrodynamic equation of \cite{jaralandimteixeira} in terms of a parabolic equation with Wentzell boundary conditions.

Second, the paper \cite{fgschutz2015} by Franco, Gonçalves and Sch\"utz, which studied the hydrodynamic limit of the SEP with a slow site in the discrete torus. In that model, there is a single site acting as a trap, whose exit rate is $\alpha N^{-\theta}$, with $\theta>0$. That model corresponds to the exclusion process in contact with a finite reservoir here studied under the additional rule that requires the reservoir to hold at most one particle. However, this condition is far from minor and our results neither cover nor are covered by that model. Thus, the hydrodynamic limit  of that model in the critical parameter $\theta=1$ remains a difficult  open problem.   

The paper is organized as follows. In Section~\ref{s2} we state definitions and results.
In Section~\ref{s3}, we study the propagation of local equilibrium for IRW in contact with the finite reservoir. In Section~\ref{s4}, we prove the hydrodynamic limit of IRW in contact with a finite reservoir. In Section~\ref{s5}, we prove the hydrodynamic limit of the SEP in contact with the finite reservoir. Finally, in Section~\ref{sec:sec7}, we provide some computational simulations.

\section{Statement of results}\label{s2}
In this paper, we adopt the convention  that $\bb N =\{0,1,2,\ldots\}$, and we fix once and for all a time-horizon $T>0$. 
 Given two real-valued functions $f,g$ defined on the same space $X$, we hereafter write  $f(u) \lesssim g(u)$ if there exists a constant $C$ independent of $u$ such that $f(u) \le C g(u)$ for every $u\in X$. 
Moreover, we will write $f(u) = {O} (g(u) )$ if the condition $|f (u) | \lesssim |g(u) |$ is satisfied for all $u\in X$.

\subsection{IRW in contact with a finite reservoir}

This corresponds to the Markov process $\{\eta_t:t\geq 0\}$ with state-space $\Omega_N=\bb N^{N+1}$, whose generator $\msf{L}_N$ acts on functions $f:\Omega_N \to \bb{R}$ as
\begin{equation}\label{generator_rw}
	\msf{L}_N f(\eta) \;=\; \sum_{\substack{x,y\,\in\, \{0,\ldots,N\} \\ |x - y| \,=\, 1}} \xi_{x,y}^N \big[f(\eta^{x,y}) - f(\eta)\big]\,,
\end{equation}
where the jump rates $\xi_{x,y}^N$ are given by
\begin{equation*}
	\xi_{x,y}^N \;=\; N^2 \times 
	\begin{cases}
		\dfrac{\alpha \eta(0)}{N^\theta}, & \text{if } x = 0, y = 1, \vspace{5pt}\\
		\eta(x), & \text{if } 1\leq x \leq N, |x - y| =1, \vspace{2pt}\\
		0, & \text{ otherwise, }
	\end{cases}
\end{equation*}
with $\theta\geq 0$, and 
\begin{equation}\label{eq:eta_xy_rw}
	\eta^{x,y}(z) \;=\;
	\begin{cases}
        \eta(x)-1, & \text{if } z = x \text{ and }\eta(x) >0, \\
		\eta(y)+1, & \text{if } z = y \text{ and }\eta(x) >0, \\
        \eta(x), & \text{if } z = x \text{ and }\eta(x) =0, \\
        \eta(y), & \text{if } z = y \text{ and }\eta(x) =0, \\
		\eta(z), & \text{otherwise.}
	\end{cases}
\end{equation}
Note that $\eta^{x,y}= \eta$ if the site $x$ is  empty. An illustration of the jump rates is provided in Figure~\ref{fig:fig2}. Although the Markov process $\{\eta_t:t\in[0,T]\}$ depends on $N$, $\alpha$ and $\theta$, we do not index it on these parameters to not overload notation.
	\begin{figure}[!htb]
		\centering
		\begin{tikzpicture}[scale=1.3]
			\centerarc[thick,<-](3.5,2.3)(0:170:0.45);
			\centerarc[thick,->](3.5,-0.3)(-10:-170:0.45);
			\centerarc[thick,<-](4.5,-0.3)(-10:-170:0.45);
			\centerarc[thick,<-](8.5,1.8)(10:170:0.45);
			\centerarc[thick,->](7.5,1.8)(10:170:0.45);
			\centerarc[thick,->](11.5,0.8)(10:170:0.45);
			\draw (3,0) -- (12,0);
			\shade[ball color=black](3,0) circle (0.25);
			\shade[ball color=black](3,0.5) circle (0.25);
			\shade[ball color=black](3,1) circle (0.25);
			\shade[ball color=black](3,1.5) circle (0.25);
			\shade[ball color=black](3,2) circle (0.25);
			\shade[ball color=black](4,0) circle (0.25);
			\shade[ball color=black](4,0.5) circle (0.25);
			\shade[ball color=black](4,1.0) circle (0.25);
			\shade[ball color=black](5,0) circle (0.25);
			\shade[ball color=black](5,0.5) circle (0.25);
			\shade[ball color=black](6,0) circle (0.25);
			\shade[ball color=black](7,0) circle (0.25);
			\shade[ball color=black](7,0.5) circle (0.25);
			\shade[ball color=black](8,0) circle (0.25);
			\shade[ball color=black](8,0.5) circle (0.25);
			\shade[ball color=black](8,1) circle (0.25);
			\shade[ball color=black](8,1.5) circle (0.25);
			\shade[ball color=black](9,0) circle (0.25);
			\shade[ball color=black](9,0.5) circle (0.25);
			\shade[ball color=black](9,1) circle (0.25);
			\shade[ball color=black](10,0) circle (0.25);
			\shade[ball color=black](10,0.5) circle (0.25);
			\shade[ball color=black](11,0) circle (0.25);
			\shade[ball color=black](12,0) circle (0.25);
			\shade[ball color=black](12,0.5) circle (0.25);
			\filldraw[fill=white, draw=black]
			(6,0) circle (.25);
			\draw (2.7,-0.1) node[anchor=north] {\small $0$};
			\draw (3.7,-0.1) node[anchor=north] {\small $1$};
			\draw (4.7,-0.1) node[anchor=north] {\small $2$};
			\draw (7.7,-0.15) node[anchor=north] {\small $x$};
			\draw (8.55,-0.1) node[anchor=north] {\small $x\!+\!1$};
			\draw (11.65,-0.1) node[anchor=north] {\small $N$};
			\draw (7.5,2.3) node[anchor=south]{$ \eta(x)$};
			\draw (8.5,2.3) node[anchor=south]{$\eta(x)$};
			\draw (3.5,-0.8) node[anchor=north]{$\eta(1)$};
			\draw (4.5,-0.8) node[anchor=north]{$\eta(1)$};
			\draw (3.5,3.7) node[anchor=north]{$\displaystyle\frac{\alpha \eta(0)}{N^\theta}$};
			\draw (11.5,1.3) node[anchor=south]{$\eta(N)$};
		\end{tikzpicture}
		\bigskip
		\caption{Illustration of jump rates (without the diffusive scaling parameter $N^2$) for IRW in contact with a finite reservoir.}\label{fig:fig2}
	\end{figure}

Our first result provides a family of  reversible measures for this process, consisting of Poisson product measures whose parameters at the sites $1,\ldots, N$ are all equal, and the parameter at zero reflects the strength of the reservoir, which acts as a trap. We properly define it below.
\begin{proposition}\label{prop:reversible_RW}
	For any $\lambda >0$,  the product measure
	\begin{equation}\label{nu_lambda}
		\nu_\lambda \;=\; \mathrm{Poisson}\Big(\frac{N^\theta}{\alpha}\cdot \lambda\Big)\otimes \bigotimes_{x=1}^N \mathrm{Poisson}(\lambda)
	\end{equation}
	in the space $\Omega_N$  is reversible for the Markov process $\{\eta_t: t \geq 0\}$.
\end{proposition}
\subsubsection{Propagation of local equilibrium}
Next, we present the propagation of local equilibrium for this model, which is a slightly stronger result  than the hydrodynamic limit (see \cite[Proposition 0.4, page 44]{kl}, for instance). We first recall the meaning of propagation of local equilibrium: starting from a product measure associated to a profile, and properly rescaling the system (in our case, in the diffusive scaling $t N^2$), it  will locally converge at a future time to a product measure whose parameter is determined by the evolution of the initial profile through a certain PDE. Below, since we are going to take the limit in $N$ with  $k\in \bb N$ fixed and $0<u<1$, we assume without loss of generality that  $ 1 \leq \lfloor u N\rfloor - k \leq \lfloor u N\rfloor + k\leq N$. 
\begin{theorem}[Propagation of Local Equilibrium]\label{thm:propagation_local}
	Let $\gamma : [0,1] \to [0, \infty)$ be a continuous profile, and define the slowly varying measure
	\begin{equation}\label{mu_N}
		\mu_N \;=\; \mathrm{Poisson}\Big(\frac{N^\theta}{\alpha} \cdot \gamma(\pfrac{0}{N})\Big)\otimes \bigotimes_{x=1}^N\mathrm{Poisson}\big(\gamma(\pfrac{x}{N})\big)\,.
	\end{equation}

	Consider the system of IRW $\{\eta_t: t \geq 0\}$ on $\Omega_N$ defined by~\eqref{generator_rw}, starting from $\mu_N$. Then, for any fixed $0<u<1$, for any positive integer $k$ and any $t>0$, the random vector $\big(\eta_{t}(\lfloor uN\rfloor - k), \ldots, \eta_{t}(\lfloor uN\rfloor + k)\big)$ converges in distribution to the product of the $2k+1$ Poisson distributions with the same parameter $\rho(t,u)$, one for each site $\lfloor uN \rfloor + i$ with $|i|\leq k$, namely 
	of constant parameter
	\begin{equation*}
		\bigotimes_{i=-k}^k \mathrm{Poisson}\big(\rho(t,u)\big)\,,
	\end{equation*}
	where $\rho: [0,T] \times [0,1]\to \bb R$ is as follows. In all three cases, the right boundary condition is of homogeneous Neumann type and, as this holds for every PDE in this paper, we do not mention the right boundary in the text to avoid repetitions.
	\begin{itemize}
		\item if $\theta \in [0,1)$, the unique strong solution to the heat equation with Neumann boundary condition at the left  boundary, given by
		\begin{equation}\label{eq_1}
			\begin{cases}
				\partial_t \rho(t,u) = \p_{uu}^2 \rho(t,u),  & \text{for } (t,u) \in (0,T]\times (0,1),    \vspace{3pt}\\
				\p_u\rho(t,0)  = \p_u\rho(t,1) = 0, & \text{for }  t \in (0,T],  \vspace{3pt}\\
				\rho(0,u) = \gamma(u), &\text{for }  u\in [0,1]. \vspace{3pt}
			\end{cases}
		\end{equation}
		\item if $\theta =1 $, the unique strong solution to the heat equation with Wentzell boundary condition at the left boundary,  given by
		\begin{equation}\label{eq_2}
			\begin{cases}
				\partial_t \rho(t,u) = \p_{uu}^2 \rho(t,u),  &\text{for }    (t,u) \in (0,T]\times (0,1),     \vspace{3pt}\\
				\p_{uu}^2\rho(t,0)= \alpha\p_{u}\rho(t,0),  &\text{for }    t \in (0,T],  \vspace{3pt}\\
				\p_u\rho(t,1) = 0, &\text{for }     t \in (0,T],  \vspace{3pt}\\
				\rho(0,u) = \gamma(u), &\text{for }    u\in [0,1]. \vspace{3pt}
			\end{cases} 
		\end{equation}
		
		\item if $\theta \in (1,\infty)$, the unique strong solution to the heat equation with homogeneous Wentzell boundary condition at the left boundary,   given by
		\begin{equation}\label{eq_3}
			\begin{cases}
				\partial_t \rho(t,u) = \p_{uu}^2 \rho(t,u),  & \text{for }  (t,u) \in (0,T]\times (0,1),     \vspace{3pt}\\
				\p_{uu}^2\rho(t,0) = 0, &  \text{for }  t \in (0,T], \vspace{3pt}\\
				\p_u\rho(t,1) = 0, &  \text{for }  t \in (0,T],  \vspace{3pt}\\
				\rho(0,u) = \gamma(u), & \text{for }  u\in [0,1]. \vspace{3pt}
			\end{cases} 
		\end{equation}
	\end{itemize}
\end{theorem}
 Note that $\mu_N$ is the slowly varying version of the reversible measures $\nu_\lambda$ in \eqref{nu_lambda}, with the constant parameter $\lambda$ replaced by $\gamma(\pfrac{x}{N})$.  The continuity of $\gamma$ is required to identify the limit density (Proposition~\ref{prop:RW}), as required by Portmanteau's Theorem.
\subsubsection{Hydrodynamic limit}
The next theorem we present is about the hydrodynamic limit of IRW. 
Let $D([0,T]; \Omega_N)$ be the path space of
c\`adl\`ag trajectories taking values on the space $\Omega_N$. For a
measure $\mu_N$ on $\Omega_N$, denote by $\bb P_{\mu_N}^{\theta, N}$ the
probability measure on $D([0,T]; \Omega_N)$ induced by the
initial state $\mu_N$ and the Markov process $\{\eta_t : t\ge 0\}$, and we denote by  $\bb E_{\mu_N}^{\theta, N}$  the expectation with respect to $\bb P_{\mu_N}^{\theta, N}$.

For $\<\cdot, \cdot \>$, we denote both the inner product associated with the $L^2[0,1]$ space, having the Lebesgue measure as the reference measure and the duality bracket (an integral of a test function against a measure ). For the sake of simplicity of notation, in the sequel we will write $\rho_t(u)$ for $\rho(t,u)$. Denote by $C^k[0,1]$ the space of functions $f:[0,1]\to \bb R$ with continuous derivatives up to order $k$, being $C[0,1]$ the set of continuous functions. A prime denotes the derivative of a function of the space variable alone, so that $H'$ and $H''$ are the first and second derivatives of $H\in C^2[0,1]$, while $\p_t$, $\p_u$ and $p^2_{uu}$ are used for functions of both time and space.

\begin{definition}\label{def:2.4}
Let $\gamma:[0,1]\to [0,\infty)$, with $\gamma\in L^2[0,1]$, and fix a constant $M\in \bb R$.
	We say that $\rho: [0,1]\times [0,T]\to \bb R$ is a weak solution to the heat equation with a non-local Dirichlet boundary condition at the left boundary and a Neumann boundary condition at the right boundary
	\begin{equation}\label{weak_solution_non_local}
		\begin{cases}
			\p_t \rho(t,u) = \p_{uu}^2 \rho(t,u),  & \text{for } (t,u) \in (0,T]\times (0,1),   \vspace{3pt}\\
			\rho(t,0)= \alpha\Big(M - \int_0^1 \rho(t,u)du\Big),  & \text{for } t \in (0,T],  \vspace{3pt}\\
			\p_u\rho(t,1) = 0, & \text{for }  t \in (0,T],  \vspace{3pt}\\
			\rho(0,u) = \gamma(u), &  \text{for } u\in [0,1] \vspace{3pt}
		\end{cases}
	\end{equation}
	if   $\rho\in L^2([0,T]\times [0,1])$ and, for any  $H\in C^{2}[0,1]$ such that $H(0)=H'(1)=0$, 
	\begin{equation}\label{integral_weak_solution_non_local}
		\big\<\rho_t, H\big\> - \big\<\gamma, H\big\>  =  \int_0^t \big\< \rho_s,    H'' \big\>\, ds  + \int_0^t \alpha \Big(M- \int_0^1\rho_s(u)du\Big)  H'(0) \, ds
	\end{equation}
	for any $t\in[0,T]$.
\end{definition}

\begin{proposition}\label{prop:uniqueness_2.9}
There exists at most one weak solution to \eqref{weak_solution_non_local}.
\end{proposition}

\begin{definition}\label{def:2.5} Let $\gamma:[0,1]\to [0,\infty)$, with $\gamma\in L^2[0,1]$, where the value $\gamma(0)$ is fixed. 
	We say that $\rho: [0,1]\times [0,T]\to \bb R$ is a weak solution to the heat equation with a Dirichlet boundary condition at the left boundary, given by
	\begin{equation}\label{weak_solution_Dirichlet}
		\begin{cases}
			\p_t \rho(t,u) = \p_{uu}^2 \rho(t,u),  & \text{for } (t,u) \in (0,T]\times (0,1),   \vspace{3pt}\\
			\rho(t,0)= \gamma(0),  & \text{for } t \in (0,T],  \vspace{3pt}\\
			\p_u\rho(t,1) = 0, &  \text{for } t \in (0,T],  \vspace{3pt}\\
			\rho(0,u) = \gamma(u), & \text{for } u\in [0,1] \vspace{3pt}
		\end{cases}
	\end{equation}
	if $\rho\in L^2([0,T]\times [0,1])$ and with the same class of test functions,  
	\begin{equation}\label{integral_weak_solution_Dirichlet_homogeneous}
		\big\<\rho_t, H\big\> - \big\<\gamma, H\big\>  = \int_0^t \big\< \rho_s,   H'' \big\>\, ds  + \int_0^t  \gamma(0)  H'(0)\,  ds
	\end{equation}
	for any $t\in[0,T]$.
\end{definition}

Note that Definitions~\ref{def:2.4} and \ref{def:2.5} are two instances of a general one if we write $\rho(t,0)= c - \kappa\int_0^1\rho(t,u)\,du$, with $(c,\kappa)=(\alpha M,\alpha)$ in \eqref{weak_solution_non_local} and $(c,\kappa)=(\gamma(0),0)$ in \eqref{weak_solution_Dirichlet}.

\begin{proposition}\label{prop:unique_4.3}
There exists at most one weak solution to \eqref{weak_solution_Dirichlet}.
\end{proposition}

The proofs of Propositions~\ref{prop:uniqueness_2.9} and \ref{prop:unique_4.3} are shown in Subsection~\ref{sub:uniqueness}, in which we expand the difference of two solutions in the eigenfunctions $\eqref{Psi_k}$ of the Sturm-Liouville problem \eqref{Liouville} and show that every coefficient vanishes.

The motivation for integral formulations appearing in the Definitions \ref{def:2.4} and \ref{def:2.5} is the usual one: from two integrations by parts, a strong solution is indeed a weak solution, as one can check. Existence of weak solutions of \eqref{weak_solution_non_local} and \eqref{weak_solution_Dirichlet} follows from the proof of the next result.
 \begin{theorem}[Hydrodynamic limit of IRW]\label{thm:hydrolimit_rws}    
	Let $\mu_N$ be\break the probability measure on $\Omega_N$ defined in \eqref{mu_N}, where $\gamma$ is a continuous profile. Consider IRW $\{\eta_t:t\in [0,T]\}$ with generator given by \eqref{generator_rw} starting from the measure $\mu_N$. Then, for any $t\in [0,T]$, for every $\delta>0$ and every $H\in C[0,1]$, it holds 
	\begin{equation*}
		\lim_{N\to\infty}
		\bb P_{\mu_N}^{\theta, N} \Big\{\eta_{\,\bigcdot} : \, \Big\vert \frac{1}{N} \sum_{x=1}^N
		H(\pfrac{x}{N})\, \eta_t(x) - \int_0^1 H(u)\, \rho(t,u) du \Big\vert
		> \delta \Big\} \;=\; 0\,,
	\end{equation*}
	where $\rho(t,u)$ is:
	\begin{itemize}
		\item if $\theta\in [0,1)$, the unique strong (thus also weak) solution to the heat equation  \eqref{eq_1}.\smallskip
		
		\item if $\theta =  1$, the unique weak solution to the heat equation   \eqref{weak_solution_non_local}.\smallskip
		
		\item if $\theta\in (1,\infty)$, the unique weak solution to the heat equation  \eqref{weak_solution_Dirichlet}.
	\end{itemize}
\end{theorem}
We note that the hypothesis on the continuity of $\gamma$ is necessary in the theorem above because its proof partially relies on Theorem~\ref{thm:propagation_local}.
\begin{remark}\label{remark:2.8}
	The propagation of local equilibrium implies the hydrodynamic limit, see \cite[Chapter 3, Proposition~0.4]{kl}. Therefore,  the case $\theta\in[0,1)$ above is simply a corollary of Theorem~\ref{thm:propagation_local}. More relevant  are the cases $\theta=1$ and $\theta\in (1,\infty)$, which leads us to deduce two surprising facts about PDEs. 
	First,  the  solution to the heat equation with Wentzell boundary condition  \eqref{eq_2}  coincides with the solution to the heat equation with non-local Dirichlet boundary condition~\eqref{weak_solution_non_local}. Second, 
	the solution to the heat equation  with homogeneous Wentzell boundary condition \eqref{eq_3}  coincides with the solution to the  heat equation with Dirichlet boundary condition \eqref{weak_solution_Dirichlet}. This correspondence can be clarified by the following argument.
    
    The Dirichlet  boundary condition~\eqref{weak_solution_non_local} is given by
	\[\rho(t,0) \;=\; \alpha\Big(M - \int_0^1 \rho(t,u) du\Big).\]
Formally differentiating the above equation in time gives us
	\[\p_t\rho(t,0) \;=\; -\alpha \int_0^1 \p_t\rho(t,u) du\,.\]
	Since $\p_t \rho(t,u) = \p_{uu}^2 \rho(t,u)$ and $\p_u \rho(t,1)=0$, this leads to
	\begin{align*}
		\p_{uu}^2\rho(t,0) & = -\alpha \int_0^1 \p_{uu}^2\rho(t,u) du =-\alpha \Big(\p_u \rho(t,1)-\p_u \rho(t,0)\Big)= \alpha \p_u \rho(t,0)\,,
	\end{align*}
	which is the Wentzell boundary condition given in \eqref{eq_2}.   
	An analogous heuristic argument can be made to deduce that the homogeneous Wentzell boundary condition appearing in \eqref{eq_3}  agrees with the non-homogeneous Dirichlet boundary condition given in \eqref{weak_solution_Dirichlet}.  
\end{remark}

\begin{remark}\label{remark:2.9}
	In view of the previous remark, which establishes a correspondence between the solution of the heat equation with the Wentzell boundary condition  \eqref{eq_2} and the solution of the heat equation with the non-local Dirichlet boundary condition  \eqref{weak_solution_non_local}, 
	the reader may wonder why the  constant $M$ appears in the PDE \eqref{weak_solution_non_local}, but it is not present in the PDE \eqref{eq_2}.  Actually, the constant $M$  is present in \eqref{eq_2} through its initial condition $\gamma$. 
	Observing that 
	\begin{equation*}
		\frac{\rho_t(0)}{\alpha} + \int_0^1 \rho_t(u)du
	\end{equation*}
	is a time conserved quantity  for the parabolic equation \eqref{eq_2}, we define $M$ by
	\begin{equation*}
		M \;=\; \frac{\gamma(0)}{\alpha} + \int_0^1 \gamma(u)du\,,
	\end{equation*}
	which can be interpreted as the total mass of the system (reservoir plus bulk), 
	playing the role of the integration constant which appears when integrating in time the Wentzell boundary condition   \eqref{eq_2} in order to arrive at the non-local Dirichlet boundary condition \eqref{weak_solution_non_local}.
\end{remark}

\begin{remark}\label{rmk:model}
	The previous remarks allow us to understand the PDE \eqref{weak_solution_non_local} in terms of a simple physical  model of diffusion (and therefore the equivalent PDE \eqref{eq_2} as well).
	Consider a diffusion system in the finite  one-dimensional interval $[0,1]$, where $\rho(t,u)$ represents its density. Hence, it satisfies the heat equation $\p_t \rho = \p_{uu}^2 \rho$. Assume that the right boundary is isolated, so that it satisfies the  Neumann boundary condition $\p_u\rho(t,1)=0$ for all $t\in (0,T]$. 
    \begin{figure}[!htb]
		\begin{tikzpicture}[scale=1]
			\fill[fill = lightgray]
			(-1,0) rectangle (0,2);
			\draw[thick] (-1,0) -- (-1,2.5);
			\fill[fill = lightgray] (0,0) to  (0,2) to[out=20,in=180] (2,3) to[out=0,in=180] (4,1.5)
			to (4,0) to (0,0);
			\draw[thick,->] (-1.5,0) -- (5,0) node[below] {$u$};
			\draw[thick,->] (0,-0.75) -- (0,3);
			\draw (2.2, 3.3) node[right]{$\rho(t,u)$};
			\draw[thick,<->] (-0.95,-0.2) -- (-0.05,-0.2)   node[midway, below] {\small $\dfrac{1}{\alpha}$};
			\draw (4,0) node[below]{$1$};
			\draw (0,0) node[anchor = north west]{$0$};
			\draw[very thick] (0,2) to[out=20,in=180] (2,3) to[out=0,in=180] (4,1.5);            
			\fill (0,2) circle (0.07);
			\node[right] at (0,1.9) {$\rho(t,0)$};
			\draw [thick, white] (4,0) -- (4,2);
			\draw [thick] (4,0) -- (4,2);
		\end{tikzpicture}
		\caption{Model of diffusion for the parabolic PDE \eqref{weak_solution_non_local}. The segment of length $1/\alpha$ represents the finite reservoir in contact with the left boundary of the interval $[0,1]$.}
		\label{fig:fig3}
	\end{figure}
	
	Assume that the left boundary of the interval $[0,1]$ is in contact with a finite reservoir of length $1/\alpha$, see Figure~\ref{fig:fig3},
	where the diffusion coefficient inside the finite reservoir is infinite, so  the reservoir immediately achieves  the equilibrium state, causing the diffusion quantity to be constant inside it. 
	Assume also that the finite reservoir is in contact with the left boundary of the interval $[0,1]$, implying that  the reservoir's height is always equal to $\rho(t,0)$, and that the finite reservoir is isolated from the exterior world. Denote by $M$ the total mass (diffusion quantity) of the system, which is  time-conserved and equal to the sum of the quantity in the reservoir and in the interval. Hence
	\begin{equation*}
		M\;=\; \frac{1}{\alpha} \rho(t,0) + \int_0^1\rho(t,u)du\,,
	\end{equation*}
	which is the non-local Dirichlet boundary condition in \eqref{weak_solution_non_local}.
\end{remark}

\begin{remark}
	The homogeneous Wentzell boundary condition $\p^2_{uu}\rho(t,0)=0$ appearing in \eqref{eq_3} is also the condition that defines the domain of the absorbed Brownian motion, where the process is stopped when it reaches zero. 
       This agrees with the idea of a system in contact with a reservoir that behaves as a sink: once the diffusion quantity enters the reservoir, it stays there forever. On the other hand, the homogeneous Dirichlet boundary condition $\rho(t,0)=0$  is also the condition that defines the domain of the killed Brownian motion (in \eqref{weak_solution_Dirichlet} the Dirichlet boundary condition is not homogeneous, but it is intimately related to it). In the killed Brownian motion, once the particle hits zero, the process ceases to exist (to be more precise, zero is not part of the state-space of the killed BM, that is why the process must cease to exist once it goes to zero). This agrees with the interpretation of a boundary PDE  that kills mass through it. In the same way that killed and absorbed Brownian motions are essentially the same process, equations \eqref{eq_3} and \eqref{weak_solution_Dirichlet} are also essentially the same PDE.
\end{remark}

\subsection{SEP in contact with a  finite reservoir}
As we shall see, the SEP leads to a rather more complicated non-linear boundary condition. In what follows, we often employ notation already used for IRW, but it should yield no ambiguity.

Let  $\theta\geq 0$ and $\alpha>0$.
The SEP in contact with a slow finite reservoir is a Markov process with state-space $\Omega_N=\bb N\times \{0,1\}^N$, whose generator $\msf{L}_N$ acts on functions $f:\Omega_N \to \bb{R}$ as
\begin{equation}\label{generator}
	\msf{L}_N f(\eta) \;=\; \sum_{\substack{x,y\,\in\, \{0,\ldots,N\} \\ |x - y| \,=\, 1}} \xi_{x,y}^N \big[f(\eta^{x,y}) - f(\eta)\big]\,,
\end{equation}
where the jump rates $\xi_{x,y}^N$ are given by
\begin{equation}\label{eq:rates}
	\xi_{x,y}^N \;=\; N^2\times
	\begin{cases}
		\dfrac{\alpha \eta(0)}{N^\theta}\big(1-\eta(1)\big), & \text{if } x = 0,y = 1, \vspace{2pt}\\
		\eta(1), & \text{if } x = 1, y = 0, \vspace{2pt}\\
		\eta(x)\big(1-\eta(y)\big), & \text{if } 1\leq x,y \leq N, |x - y| =1, \vspace{2pt}\\
		0, & \text{ otherwise},
	\end{cases}
\end{equation}
 and the configuration $\eta^{x,y}$ is given by moving a particle from the site $x$ to the site $y$ if possible. That is, denoting
\begin{equation*}
	\Xi(\eta)(z) \;=\;
	\begin{cases}
        \eta(x)-1, & \text{if } z = x, \\
		\eta(y)+1, & \text{if } z = y, \\
		\eta(z), & \text{otherwise,}
	\end{cases}
\end{equation*}
we define
\begin{equation}\label{eq:eta_xy_ex}
	\eta^{x,y} \;=\;
	\begin{cases}
        \Xi(\eta), & \text{if } \Xi(\eta) \in \Omega_N, \\
		\eta, & \text{otherwise.}
	\end{cases}
\end{equation}
Note that a jump from $x$ to $y\in\{1, \ldots, N\}$ may occur only if the site $x$ is occupied and the site $y$ is empty. Observe also the presence of the diffusive scaling factor $N^2$ in \eqref{eq:rates}.
In Figure~\ref{fig:fig1} we illustrate the jump rates.
\begin{figure}[!htb]
	\centering
	\begin{tikzpicture}[scale=1.3]
		\centerarc[thick,<-](3.5,1.3)(0:170:0.45);
		\centerarc[thick,->](3.5,-0.3)(-10:-170:0.45);
		\centerarc[thick,->](7.5,-0.3)(-10:-170:0.45);
		\centerarc[thick,<-](7.5,0.3)(10:170:0.45);
		\centerarc[thick,->](11.5,-0.3)(-10:-170:0.45);
		\centerarc[thick,<-](11.5,0.3)(10:170:0.45);
		\draw (3,0) -- (12,0);
		\shade[ball color=black](3,0) circle (0.25);
		\shade[ball color=black](3,0.5) circle (0.25);
		\shade[ball color=black](3,1) circle (0.25);
		\shade[ball color=black](5,0) circle (0.25);
		\shade[ball color=black](8,0) circle (0.25);
		\shade[ball color=black](10,0) circle (0.25);
		\shade[ball color=black](6,0) circle (0.25);
		\shade[ball color=black](12,0) circle (0.25);
		
		\filldraw[fill=white, draw=black]
		(4,0) circle (.25)
		(7,0) circle (.25)
		(9,0) circle (.25)
		(11,0) circle (.25);
		
		\draw (2.7,-0.1) node[anchor=north] {\small $0$}
		(3.7,-0.1) node[anchor=north] {\small $1$}
		(4.7,-0.1) node[anchor=north] {\small $2$}
		(5.7,-0.1) node[anchor=north] {}
		(6.7,-0.15) node[anchor=north] {\small $x$}
		(7.6,-0.1) node[anchor=north] {\small $x\!+\!1$}
		(11.7,-0.1) node[anchor=north] {\small $N$};
		\draw (7.5,0.8) node[anchor=south]{$\eta(x)\big(1-\eta(x+1)\big)$};
		\draw (7.5,-0.8) node[anchor=north]{$\eta(x+1)\big(1-\eta(x)\big)$};
		\draw (3.5,-0.8) node[anchor=north]{$\eta(1)$};
		\draw (3.5,2.6) node[anchor=north]{$\displaystyle\frac{\alpha \eta(0)}{N^\theta}\big(1-\eta(1)\big)$};
		\draw (11.5,0.8) node[anchor=south]{$\eta(N-1)\big(1-\eta(N)\big)$};
		\draw (11.5,-0.8) node[anchor=north]{$\eta(N)\big(1-\eta(N-1)\big)$};
	\end{tikzpicture}
	\bigskip
	\caption{Illustration of jump rates (without the diffusive scaling parameter $N^2$) for the SEP in contact with a finite reservoir.}\label{fig:fig1}
\end{figure}

Our first result is concerned with the invariant (actually reversible) measure for this process, which consists of a product measure of Bernoulli measures with constant parameter at sites $1,\ldots, N$ and a Poisson measure at site zero, whose parameters are properly related.
\begin{proposition}\label{prop:reversible}
Let $p \in (0,1)$.
Then, the product measure $\nu_p$ on the state-space $\Omega_N$ defined by
\begin{equation}\label{nu_p}
	\nu_p \;=\; \mathrm{Poisson}\Big(\frac{N^\theta}{\alpha}\cdot \frac{p}{1-p}\Big)\otimes \bigotimes_{i=1}^N \mathrm{Bernoulli}(p)
\end{equation}
is reversible for the Markov process $\{\eta_t: t \geq 0\}$.
\end{proposition}

\subsubsection{Hydrodynamic limit} We start by  stating the three hydrodynamic equations obtained in the hydrodynamic limit of the SEP in contact with a finite reservoir, for different  ranges of the parameter $\theta$.
\begin{definition}\label{def:weak_Neumann}
Assume that $\gamma:[0,1]\to [0,1]$ is a measurable profile, not necessarily continuous. We say that a measurable function $\rho:[0,T]\times [0,1] \to [0,1]$ is a weak solution to the heat equation with Neumann boundary conditions
\begin{equation}\label{heat_Neumann}
	\begin{cases}   \partial_t \rho(t,u) = \p_{uu}^2 \rho(t,u),  & \text{for } (t,u) \in (0,T]\times (0,1),   \vspace{3pt}\\
		\p_u\rho(t,0) = \p_u\rho(t,1) = 0, &  \text{for } t \in (0,T],  \vspace{3pt}\\
		\rho(0,u) = \gamma(u), &  \text{for } u\in [0,1] \vspace{3pt}
	\end{cases}
\end{equation}
if, for any function $H\in C^{2}[0,1]$ such that $H'(0)= H'(1)=0$, the equation
\begin{equation*}
	\big\<\rho_t, H\big\> - \big\<\gamma, H\big\>  \;=\;  \int_0^t \big\< \rho_s,  H'' \big\>\, ds
\end{equation*}
holds for any time $t\in[0,T]$.
\end{definition}
\begin{proposition}\label{prop:uniqueness_5.7} 
There exists at most one weak solution to \eqref{heat_Neumann}.
\end{proposition}

The proof of this proposition, as well as those of the two uniqueness propositions stated below, is postponed to Subsection~\ref{sub:uniqueness_SEP}.

Denote by $C^{0,1}_c([0,T]\times [0,1])$ the space of functions $G:[0,T]\times [0,1]\to \bb R$ that are continuous in the first variable and $C^1$ in the second variable, with compact support in $[0,T]\times (0,1)$.
Following the standard notation of \cite{E}, let $L^2([0,T]; \mc H^1(0,1))$ denote the Sobolev space of functions $\psi\in L^2([0,T]\times [0,1])$ for which there exists a function in $L^2([0,T]\times [0,1])$, denoted by $\p_u \psi$, such that
\begin{equation*}
\int_0^T \int_0^1 (\p_u G)(s,u) \psi(s,u) du ds \;=\; - \int_0^T \int_0^1  G(s,u) (\p_u\psi)(s,u) du ds
\end{equation*}
for any $G\in C^{0,1}_c([0,T]\times [0,1])$. 

\begin{definition}\label{def:weak}
Fix a constant $M\in \bb R$ and let  $\gamma:[0,1]\to [0,1]$ be a measurable profile. We  say that a measurable function $\rho:[0,T]\times [0,1] \to [0,1]$ is a weak solution to
\begin{equation}\label{non_linear_Dirichlet}
	\begin{cases}   \partial_t \rho(t,u) = \p_{uu}^2 \rho(t,u),  & \text{for } (t,u) \in (0,T]\times (0,1),   \vspace{3pt}\\
		\dfrac{\rho(t,0)}{1-\rho(t,0)} = \displaystyle \alpha\bigg(M - \int_0^1 \rho(t,u)du\bigg),  &  \text{for } t \in (0,T],  \vspace{3pt}\\
		\p_u\rho(t,1) = 0, &  \text{for } t \in (0,T],  \vspace{3pt}\\
		\rho(0,u) = \gamma(u), &  \text{for } u\in [0,1] \vspace{3pt}
	\end{cases}
\end{equation}
 if  $\rho\in \Sobolev$
and, for any  $H\in C^{2}[0,1]$ such that $H(0)= H'(1)=0$, it holds
\begin{equation}\label{eq:cond_3}
	\begin{split}
		\big\<\rho_t, H\big\> - \big\<\gamma, H\big\>  = & \int_0^t \big\< \rho_s,  H'' \big\>\, ds\\
		& + \int_0^t \alpha \big(1-\rho_s(0)\big)\bigg(M- \int_0^1\rho_s(u)du\bigg)   H'(0)  \,ds
	\end{split}
\end{equation}
for any time $t\in[0,T]$.
\end{definition}
Note that  \eqref{non_linear_Dirichlet}, in contrast to \eqref{weak_solution_non_local}, has  a non-linear boundary condition, which can  be rewritten as $\rho_t(0) = \frac{\alpha(M - \int_0^1 \rho_t(u)du)}{1+\alpha(M - \int_0^1 \rho_t(u)du)}$  for any $t \in (0,T]$.
Moreover, in the integral equation \eqref{eq:cond_3} above,  the term $\rho_s(0)$ should be understood in the sense of the trace of a function in the Sobolev space $\Sobolev$, see \cite[Chapter~5]{E} for instance. 

\begin{proposition}\label{prop:unique_non_linear_Dirichlet}
There exists at most one weak solution to \eqref{non_linear_Dirichlet}.
\end{proposition}

\begin{definition}\label{def:weak_homogeneous_Dirich}
Let $\gamma:[0,1]\to [0,1]$ be a measurable profile. We say that a measurable function $\rho:[0,T]\times [0,1] \to \bb R$ is a weak solution to the heat equation with homogeneous Dirichlet boundary condition 
\begin{equation}\label{heat_dirichlet}
	\begin{cases}   \partial_t \rho(t,u) = \p_{uu}^2 \rho(t,u),  & \text{for } (t,u) \in (0,T]\times (0,1),   \vspace{3pt}\\
		\rho(t,0) = 0, &  \text{for } t \in (0,T],  \vspace{3pt}\\
		\p_u\rho(t,1) = 0, &  \text{for } t \in (0,T],  \vspace{3pt}\\
		\rho(0,u) = \gamma(u), &  \text{for } u\in [0,1] \vspace{3pt}
	\end{cases}
\end{equation}
if, for any function $H\in C^{2}[0,1]$ such that $H(0)= H'(1)=0$, it holds 
\begin{equation*}
	\begin{split}
		\big\<\rho_t, H\big\> - \big\<\gamma, H\big\>  \;=\; & \int_0^t \big\< \rho_s,  H'' \big\>\, ds
	\end{split}
\end{equation*}
for any time $t\in[0,T]$.
\end{definition}
\begin{proposition}\label{prop:uniqueness_5.10}
There exists at most one weak solution to \eqref{heat_dirichlet}.
\end{proposition}
Existence of weak solutions to \eqref{heat_Neumann}, \eqref{non_linear_Dirichlet} and \eqref{heat_dirichlet} is a consequence of the next theorem.
Let $D([0,T]; \Omega_N)$ be the path space of
c\`adl\`ag trajectories taking values on $\Omega_N$. For a
measure $\mu_N$ on $\Omega_N$, denote by $\bb P_{\mu_N}^{\theta, N}$ the
probability measure on $D([0,T]; \Omega_N)$ induced by the
initial state $\mu_N$ and the Markov process $\{\eta_t : t\ge 0\}$ and denote by  $\bb E_{\mu_N}^{\theta, N}$  the expectation with respect to~$\bb P_{\mu_N}^{\theta, N}$.
\begin{theorem}[Hydrodynamic Limit]\label{thm:hydro_ex}
Let $\gamma:[0,1]\to [0,1]$ be a continuous profile. Assume that in the case $\theta=1$ it also holds $0<\gamma(0)<1$.  
Consider the slowly varying  measure
\begin{equation}\label{mu_N_ex}
	\mu_N \;=\; \mathrm{Poisson}\Big(\frac{N^\theta}{\alpha} \cdot \gamma(\pfrac{0}{N})\Big)\otimes \bigotimes_{x=1}^N\mathrm{Bernoulli}\big(\gamma(\pfrac{x}{N})\big)\,.
\end{equation}
Then, for any $t\in [0,T]$, for every $\delta>0$ and every $H\in C[0,1]$, it holds 
\begin{equation*}
	\lim_{N\to\infty}
	\bb P_{\mu_N}^{\theta, N} \Big\{\eta_{\,\bigcdot} : \, \Big\vert \frac{1}{N} \sum_{x=1}^N
	H(\pfrac{x}{N})\, \eta_t(x) - \int_0^1 H(u)\, \rho(t,u) du \Big\vert
	\;>\; \delta \Big\} \;=\; 0\,,
\end{equation*}
where $\rho(t,u)$ is:
\begin{itemize}
	\item if $\theta\in [0,1)$, the unique weak solution to the  heat equation  \eqref{heat_Neumann}.
	\item if $\theta =  1$, the unique weak solution to the heat equation  \eqref{non_linear_Dirichlet}.
	\item if $\theta\in (1,\infty)$, the unique weak solution to the heat equation  \eqref{heat_dirichlet}.
\end{itemize}
\end{theorem}

\begin{remark}
For $\theta\in[0,1)$, in the random walk scenario presented in the previous subsection, the limit was given by a \textit{strong solution} to the heat equation with Neumann boundary conditions. Here the limit is stated in terms of a \textit{weak solution} because Varadhan's Entropy Method naturally leads to weak solutions via the limit of Dynkin's martingale. Moreover, the assumption $0<\gamma(0)<1$ is necessary only for $\theta=1$, in order to ensure an entropy estimate, which  is an ingredient in the proof of a local replacement lemma, which is needed only in that regime of $\theta$. The continuity of $\gamma$ is used to be able to identify the initial condition. 
\end{remark}

\begin{remark}
By the reasoning of Remark~\ref{remark:2.8},  differentiating in time the left boundary condition of \eqref{non_linear_Dirichlet}, we can say that  the  PDE  \eqref{non_linear_Dirichlet}  is formally the heat equation with a non-linear Wentzell boundary condition given by
\begin{equation*}
	\begin{cases}   \partial_t \rho(t,u) = \p_{uu}^2 \rho(t,u),  & \text{for } (t,u) \in (0,T]\times (0,1),   \vspace{3pt}\\
		\p_{uu}^2\rho(t,0)  = \alpha \big(1-\rho(t,0)\big)^2\p_{u}\rho(t, 0), &  \text{for } t \in (0,T],\vspace{3pt}\\
		\p_u\rho(t,1) = 0, &  \text{for } t \in (0,T],  \vspace{3pt}\\
		\rho(0,u) = \gamma(u), &  \text{for } u\in [0,1], \vspace{3pt}
	\end{cases}
\end{equation*}
where the constant $M$ in \eqref{non_linear_Dirichlet} can be interpreted in the same way as in the Remark~\ref{remark:2.9}, i.e. as the total mass of the system (reservoir plus bulk) and $\p_{uu}^2\rho(t,0)$ is the time-derivative of the mass in the reservoir.
\end{remark}

\begin{figure}[!htb]
\begin{tikzpicture}[scale=1]
	\fill[fill = lightgray]
	(-1,0) rectangle (0,2.5);
	\draw[thick] (-1,0) -- (-1,3);
	\fill[fill = lightgray] (0,0) to  (0,1) to[out=20,in=180] (2,2) to[out=0,in=180] (4,0.5)
	to (4,0) to (0,0);
	\draw[thick,->] (-1.5,0) -- (5,0) node[below] {$u$};
	\draw[thick,->] (0,-0.75) -- (0,3);
	\draw (2.2, 2.3) node[right]{$\rho(t,u)$};
	\draw (0, 2.5) node[right]{$\frac{\rho(t,0)}{1-\rho(t,0)}$};
	\draw[thick,<->] (-0.95,-0.2) -- (-0.05,-0.2)   node[midway, below] {\small $\dfrac{1}{\alpha}$};
	\draw (4,0) node[below]{$1$};
	\draw (0,0) node[anchor = north west]{$0$};
	\draw[very thick] (0,1) to[out=20,in=180] (2,2) to[out=0,in=180] (4,0.5);
	\fill (0,1) circle (0.07);
	\fill (0,2.5) circle (0.07);
	\node[right] at (0,0.9) {$\rho(t,0)$};
	\draw [thick, white] (4,0) -- (4,2);
	\draw [ thick] (4,0) -- (4,2);
\end{tikzpicture}
\caption{Model of diffusion corresponding to the parabolic PDE \eqref{non_linear_Dirichlet}. The segment of length $1/\alpha$ represents the finite reservoir in contact with the left boundary of the box $[0,1]$.}
\label{fig:fig4}
\end{figure}

\begin{remark}
Similarly to what we did in Remark~\ref{rmk:model}, we can interpret  the PDE \eqref{non_linear_Dirichlet} as a  diffusion model of an interval $[0,1]$ isolated at the right boundary and in contact with a finite reservoir at the left boundary, where the constant $M$ in the non-linear non-local Dirichlet boundary condition satisfies
\begin{equation*}
	\dfrac{\rho_t(0)}{1-\rho_t(0)} \;=\; \alpha\Big(M - \int_0^1 \rho_t(u)du\Big)
\end{equation*}
is interpreted as the total mass of the system.  We  illustrate it in Figure~\ref{fig:fig4}.
Note that the boundary condition  above plays the role of a thermostat, i.e. since $0\leq\rho<1$,  the height $\frac{\rho_t(0)}{1-\rho_t(0)}$ of the reservoir always stays above the density
$\rho_t(0)$ of the bulk at the left boundary. Additionally, the profile $\rho$ should remain below one, which is reasonable  since the hydrodynamic equation \eqref{non_linear_Dirichlet} is the limit of an exclusion process, where at most one particle is allowed  per site.  Moreover, the case $\rho\equiv 1$  heuristically leads the total mass to be $M=\infty$, which intuitively agrees with the invariant measure~\eqref{nu_p}. We  point out that  Wentzell boundary conditions have been also interpreted as thermostats, see \cite{Goldstein2006}.
\end{remark}

\begin{figure}[!htb]
\begin{tikzpicture}[scale=1]
	\fill[fill = lightgray]
	(-2.5,0) rectangle (0,1);
	\draw[thick] (-2.5,0) -- (-2.5,1.5);
	\fill[fill = lightgray] (0,0) to  (0,1) to[out=20,in=180] (2,2) to[out=0,in=180] (4,0.5)
	to (4,0) to (0,0);
	\draw[thick,->] (-3,0) -- (5,0) node[below] {$u$};
	\draw[thick,->] (0,-0.75) -- (0,3);
	\draw (2.2, 2.3) node[right]{$\rho(t,u)$};
	\draw[thick,<->] (-2.45,-0.2) -- (-0.05,-0.2)   node[midway, below] {$\frac{1}{\alpha(1-\rho(t,0))}$};
	\draw (4,0) node[below]{$1$};
	\draw (0,0) node[anchor = north west]{$0$};
	\draw[very thick] (0,1) to[out=20,in=180] (2,2) to[out=0,in=180] (4,0.5);
	\fill (0,1) circle (0.07);
	\node[right] at (0,0.9) {$\rho(t,0)$};
	\draw [thick, white] (4,0) -- (4,2);
	\draw [thick] (4,0) -- (4,2);
\end{tikzpicture}
\caption{A second model of diffusion corresponding to the parabolic PDE \eqref{non_linear_Dirichlet}. The segment of length $1/\big(\alpha(1-\rho(t,0))\big)$ represents the finite reservoir whose capacity is a function of its internal height.}
\label{fig:fig5}
\end{figure}

\begin{remark}
We can relate the  hydrodynamic equation \eqref{non_linear_Dirichlet} to another physical model of diffusion avoiding the thermostat interpretation.
To do so, we assume  that the capacity of the reservoir  is a function of its internal height, given by $1/\big(\alpha(1-\rho(t,0))\big)$, which is represented by its width, see Figure~\ref{fig:fig5}. In this setting,  the reservoir and the left boundary of the bulk will be at equilibrium.
\end{remark}

\begin{remark}
In the regime $\theta\in(1,\infty)$, the Dirichlet boundary condition obtained in Theorem~\ref{thm:hydro_ex} is homogeneous, see equation \eqref{heat_dirichlet}, so the density close to the reservoir is null. On the other hand, the Dirichlet boundary condition obtained in Theorem~\ref{thm:hydrolimit_rws} for IRW is not homogeneous, which means that the density close to the reservoir is fixed and equal to~$\gamma(0)$, see  equation \eqref{weak_solution_Dirichlet}. This can be understood taking into account the fact that it is harder to jump from the reservoir to the site $x=1$ in the exclusion scenario, which leads the reservoir to behave as a sink.
\end{remark}

\section{Proof of propagation of local equilibrium for IRW}\label{s3}
The argument is split into two parts: the convergence in distribution of a single random walk, and an explicit formula for the joint Laplace transform of the system. We start with the former. 

\subsection{Convergence in distribution of the underlying RW}
In the proof of propagation of local equilibrium, we will have to understand the convergence in distribution of a single  random walk (see also Figure~\ref{fig:fig2}) for each regime of $\theta$. This is the content of this subsection. Let $\{Y_t^{N, \theta}:t\geq 0\}$ be the continuous-time random walk on the state-space $\{0, \pfrac{1}{N},\ldots, \pfrac{N-1}{N}, 1\}$ whose generator acts on functions $f:\{0, \pfrac{1}{N},\ldots, \pfrac{N-1}{N}, 1\}\to \bb R$ as
\begin{equation}\label{gene_RW}
\bb L_{N, \theta} f\big(\pfrac{x}{N}\big) \;=\;
N^2\times \begin{cases}
	f(\pfrac{x+1}{N})+f(\pfrac{x-1}{N}) -2 f(\pfrac{x}{N}), & \text{ if } 0<x<N, \vspace{3pt}\\
	f(\pfrac{N-1}{N})- f(1), & \text{ if } x=N, \vspace{3pt}\\
	\dfrac{\alpha}{N^\theta}\big[f(\pfrac{1}{N})- f(0)\big], & \text{ if } x=0. \vspace{3pt}\\
\end{cases}
\end{equation}
Note that the random walk $Y^{N,\theta}_t$ corresponds to the position normalized by $N$ of a single particle in the particle system $\eta_t$ defined by \eqref{generator_rw}.

\begin{proposition}[Convergence of the underlying random walk]\label{prop:RW}
Fix $0<u<1$, fix $k$ a positive integer, and let  $Y_t^{N\, \theta}$ be the random walk above  starting from any site in the window $\lfloor uN\rfloor/N-k, \ldots, \lfloor uN\rfloor/N+k$. 
Then $\{Y^{N,\theta}_t:t\in [0,T]\}$ converges in 
distribution with respect to the Skorokhod topology of $D([0,T];\bb R)$  to $\{Y_t^\theta:t\in [0,T]\}$ as $N\to\infty$, where $Y_t^\theta$ is the Feller process on the state-space $[0,1]$ starting from $u$, such that:\smallskip
\begin{itemize}
	\item if $\theta\in [0,1)$, it  is a  Brownian motion reflected at $0$ and at $1$, that is, its generator is
	$\bb L_\theta f(u) \;=\; f''(u)$ with domain 
    \[\mf D(\bb L_\theta)=\big\{f\in C^2[0,1]: f'(0)= f'(1) = 0\big\}\,.\]
	
	\item if $\theta=1$, it is a sticky Brownian motion at $0$ and reflected  at $1$, that is, its generator is $\bb L_\theta f(u) = f''(u)$  with domain 
    \[\mf D(\bb L_\theta)\;=\;\big\{f\in C^2[0,1]: f''(0) = \alpha f'(0) \text{ and } f'(1) = 0\big\}\,.\]
	
	\item if $\theta\in (1,\infty)$, it is a Brownian motion  absorbed at $0$ and reflected  at $1$, that is, its generator is $\bb L_\theta f(u) = f''(u)$  with domain
    \[\mf D(\bb L_\theta)\;=\;\big\{f\in C^2[0,1]: f''(0) = 0 \text{ and } f'(1) = 0\big\}\,.\]
\end{itemize}
\end{proposition}

Note that in all  cases above $\bb L_\theta f(u) = f''(u)$, whilst a usual Brownian motion has generator $ \frac{1}{2}f''$, so a factor of $1/2$ is missing. That is, to be more precise, we should say  that all the limiting processes in Proposition~\ref{prop:RW} are Brownian motions at time $2t$.

\begin{proof}[Proof of Proposition~\ref{prop:RW}]
Since the state-space here is the set $\{0, \pfrac{1}{N}, \ldots, \pfrac{N-1}{N}, 1\}$, since $k$ is fixed, and  $k/N\to 0$ as $N\to \infty$,  we can assume without loss of generality that the starting point of $\{Y^{N, \theta}_t:t\in[0,T]\}$ is $x=\lfloor uN\rfloor/N$.
Denote by $\Vert \cdot\Vert_\infty$ the supremum norm.
By \cite[Theorem~6.1, page~28 and Theorem~2.11, page~172]{EK}, to assure convergence in distribution of $\{Y^{N,\theta}_t:t\in[0,T]\}$, it is enough to find a core   $\mc C_\theta$ for the generator $\bb L_\theta$ such that for any $f\in \mc C_\theta$ there exists a sequence $f_N\in \mf D(\bb L_\theta)$ satisfying  
\begin{align}
	& \Vert f_N - \pi_N f\Vert_\infty\to 0 \quad \text{ and } \label{SovaKurtz_1}\\
	& \Vert \bb L_{N,\theta}f_N - \pi_N \bb L_\theta f\Vert_\infty\to 0\,,\label{SovaKurtz_2}
\end{align}
where $\pi_N f$ is the restriction of  $f:[0,1]\to \bb R$ to the lattice $\{0, \frac{1}{N},\ldots,\frac{N-1}{N},1\}$.  In any of the three regimes of $\theta$ given above,  we choose the core as the domain itself, that is, $\mc C_\theta=\mf D(\bb L_\theta)$. Given $f\in \mf D(\bb L_\theta)$, the approximating sequence $f_N$ will be  defined as
\begin{equation*}
	f_N \;:=\; \pi_N f + \widetilde{f}_N + \widehat{f}_{N,\theta} \,,
\end{equation*}
where
\begin{equation}\label{eq:f_Ntilde}
	\widetilde{f}_N\big(\pfrac{x}{N}\big) \;:=\; \frac{f''(1)}{2} \cdot \frac{h \big(\pfrac{x}{N}\big)}{N}\,, 
\end{equation}
being $h:[0,1]\to \bb R$ a smooth function with compact support in $(0,1]$ such that $h'(1)=-1$,
and 
\begin{equation}\label{eq:f_Nhat}
	\widehat{f}_{N, \theta}\big(\pfrac{x}{N}\big) \;:=\;  
	\begin{cases}
		\dfrac{1}{\alpha N}\big(1-\pfrac{\alpha}{2}\big)f''(0) \cdot h \big(1-\pfrac{x}{N}\big) ,  & \text{ if }\theta =0, \vspace{7pt}\\
		\dfrac{N^{\theta-1}}{\alpha} f''(0) \cdot h \big(1-\pfrac{x}{N}\big) ,  & \text{ if }\theta \in (0,1),\vspace{5pt}\\
		0, & \text{ if }\theta \in [1,\infty),\\
	\end{cases}
\end{equation}
for $x\in \{0,1,\ldots, N\}$. As we shall see, the functions $\widetilde{f}_N$ and  $ \widehat{f}_{N, \theta}$   play the role of  ``correctors'' at the sites $N$ and $0$, respectively. Since $\Vert \widetilde{f}_N\Vert_\infty$ and $\Vert \widehat{f}_{N, \theta}\Vert_\infty$ go to zero as $N\to\infty$, then  \eqref{SovaKurtz_1} holds. It remains to check \eqref{SovaKurtz_2}.

We start by observing that, according to \eqref{gene_RW}, the generator $\bb L_{N,\theta} f(\frac{x}{N})$ is the discrete Laplacian for any  $x\in \{1, \ldots, N-1\}$, which approximates the continuous Laplacian since $f\in C^2[0,1]$. Therefore,
\begin{equation*}
	\sup_{x\in \{1,\ldots, N-1\}} \big\vert \bb L_{N,\theta}\pi_N f(\pfrac{x}{N}) -  \bb L_\theta f(\pfrac{x}{N})\big\vert \longrightarrow 0\,, \quad \text{ as }N\to\infty\,. 
\end{equation*}
Recall that $f\in C^2[0,1]$. At $x=N$, by a Taylor expansion of $f$   around the point $u=1$ with remainder in Lagrange form, and using the fact that $f'(1)=0$, we conclude that 
\begin{equation*}
	\bb L_{N,\theta} \pi_N f(1) \;=\; N^2 \big[f(\pfrac{N-1}{N}) - f(1) \big] \;=\; \frac{1}{2} f''(\zeta)
\end{equation*}
for some $\zeta \in (1-1/N,1)$. Thus  
\begin{equation*}
	\bb L_{N,\theta} \pi_N f(1) \longrightarrow \frac{1}{2} f''(1)\quad \text{ as }N\to \infty\,.
\end{equation*}
Note that the generator of the limiting process at $u=1$ is $f''(1)$ in all cases, so the $1/2$ factor above is not the desired one (which illustrates the need of the ``correctors''). From  \eqref{eq:f_Ntilde}, 
\begin{align*}
	\bb L_{N,\theta} \widetilde{f}_N(\pfrac{N}{N})  & = N^2 \big[\widetilde{f}_N(\pfrac{N-1}{N}) - \widetilde{f}_N(1) \big]
	\\
	&= \frac{f''(1)}{2} \cdot N\Big[h\big(\pfrac{N-1}{N}\big) - h(1) \Big]\;\longrightarrow\; -h'(1)\frac{f''(1)}{2} = \frac{f''(1)}{2}\,.
\end{align*}
Recall  \eqref{eq:f_Nhat}. Since $h$ has compact support in $(0,1]$, then $h(1-u)$ has compact support in $[0,1)$, hence $\bb L_{N,\theta}\widehat{f}_{N, \theta}(1)=0$ for $N$ large enough.
Therefore, 
\[\bb L_{N,\theta} f_N(1)\;\longrightarrow\; f''(1) = \bb L_\theta f(1)\,.\] 
It remains to study the convergence of $\bb L_{N,\theta} f_N(0)$. To do so, we divide the analysis according to the regime of $\theta$.\medskip

 $\bullet$ \textbf{Case  $\theta\in [0,1)$.}
By a Taylor expansion of $f$ around the point $u=0$,
\begin{equation*}
	\bb L_{N,\theta} \pi_N f(0) \;=\; \alpha N^{2-\theta} \big[f(\pfrac{1}{N}) - f(0) \big] \;=\; \alpha N^{2-\theta} \Big[\frac{f'(0)}{N} + \frac{f''(\zeta)}{2N^2}\Big] \,,
\end{equation*}
for some $\zeta\in [0,1/N]$. Since $f'(0)=0$ and $f\in C^2[0,1]$, we conclude that
\begin{equation*}
	\bb L_{N,\theta} \pi_N f(0) \;\longrightarrow\; 
	\begin{cases}
		\frac{\alpha}{2}f''(0), & \text{ if } \theta=0,\\ 
		0, & \text{ if } \theta\in (0,1),
	\end{cases}
\end{equation*}
as $N\to\infty$. Since $h$ has compact support in $(0,1]$, then $\bb L_{N,\theta} \widetilde{f}_N(0)  = 0$. 
From \eqref{eq:f_Nhat}, we have that
\begin{align*}
	\bb L_{N,\theta} \widehat{f}_{N, \theta}(0) \; \longrightarrow\; 
	\begin{cases}
		\big(1-\frac{\alpha}{2}\big)f''(0), & \text{ if } \theta=0,\\ 
		f''(0), & \text{ if } \theta\in (0,1),
	\end{cases}
\end{align*}
thus $\bb L_{N,\theta} f_N(0)\to f''(0) = \bb L_\theta f(0)$.
\medskip

 $\bullet$ \textbf{Case $\theta = 1$.} Since the boundary condition is $\alpha f'(0) = f''(0)$, 
\begin{equation*}
	\bb L_{N,\theta} \pi_N f(0) =\alpha N \big[f(\pfrac{1}{N}) - f(0) \big]\;\longrightarrow\; \alpha f'(0)  =
	\bb L_{\theta} f(0)\,,
\end{equation*}
 hence $\bb L_{N,\theta} f_N(0)\to f''(0) = \bb L_\theta f(0)$.\medskip

 $\bullet$ \textbf{Case $\theta\in (1,\infty)$.}  Since $f$ is Lipschitz, 
\begin{equation*}
	\bb L_{N,\theta} \pi_N f(0) = \alpha N^{2-\theta} \big[f(\pfrac{1}{N}) - f(0) \big] \;\longrightarrow\; 0
	= f''(0)= \bb L_{\theta} f(0)\,.
\end{equation*}
so $\bb L_{N,\theta} f_N(0)\to f''(0) = \bb L_\theta f(0)$.
Putting together all the cases above, we establish \eqref{SovaKurtz_2}, finishing the proof.
\end{proof}

\subsection{Joint Laplace transform of IRW}
The proof of the propagation of local equilibrium is based on the analysis of the Laplace transform  of the vector
$\eta_t=\break (\eta_t(0), \eta_t(1),\ldots, \eta_t(N))$, where we recall that $\eta_t(x)$ denotes the number of particles at the site $x$ at time $t$.  Let  $(\zeta(x))_{x=0}^N$ be a family of parameters and consider the joint Laplace transform
\begin{equation*}
\Phi(\zeta)\;=\; \Phi_{\eta_t}(\zeta)
\;:=\;
\mathbb{E}_{\mu_N}\bigg[\exp\Big\{-\sum_{x=0}^N \zeta(x)\eta_t(x)\Big\}\bigg]\,.
\end{equation*}
Let $X_t^{y,k}$ denote the position at time $t$ of the $k$-th random walk that started at site $y$. 
Then
\[
\eta_t(x) \;=\; \sum_{y=0}^N \sum_{k=1}^{\eta_0(y)} \mathbf{1}\{X_t^{y,k}=x\}\,,
\]
and then
\begin{align*}
\sum_{x=0}^N \zeta(x)\eta_t(x)
& = \sum_{x=0}^N \zeta(x)\sum_{y=0}^N \sum_{k=1}^{\eta_0(y)}\mathbf{1}\{X_t^{y,k}=x\}= \sum_{y=0}^N \sum_{k=1}^{\eta_0(y)}\zeta(X_t^{y,k}) \sum_{x=0}^N \mathbf{1}\{X_t^{y,k}=x\}\\
&=\sum_{y=0}^N \sum_{k=1}^{\eta_0(y)}\zeta(X_t^{y,k})\,.
\end{align*}
Thus
\begin{align*}
\exp\Big\{-\sum_{x=0}^N\zeta(x)\eta_t(x)\Big\}
&\;=\; \prod_{y=0}^N \prod_{k=1}^{\eta_0(y)} \exp\Big\{- \zeta(X_t^{y,k})\Big\}\,.
\end{align*}
Since the random walks $X^{y,k}_t$ are independent, we get
\begin{align*}
\mathbb{E}_{\mu_N}\Big[\exp\Big\{-\sum_{x=0}^N\zeta(x)\eta_t(x)\Big\}\Big]
&\;=\; \prod_{y=0}^N \mathbb{E}_{\mu_N}\bigg[
\exp\Big\{-\sum_{k=1}^{\eta_0(y)}\zeta(X_t^{y,k})\Big\}
\bigg].
\end{align*}
Since the random walks $X_t^{y,k}$ are also  independent of the initial quantity of particles $\eta_0(y)$ for every $y \in \{0, \ldots, N\}$, 
we can apply the Substitution Principle (see \cite[Example~5.1.5]{Durrett}) to rewrite the last display as
\begin{align}
\prod_{y=0}^N \mathbb{E}_{\mu_N}\Bigg[ \mathbb{E}_{\mu_N}\bigg[
\exp\Big\{-\sum_{k=1}^{\eta_0(y)}\zeta(X_t^{y,k}) \Big\}\Big\vert &  \eta_0\bigg]
\Bigg] \;=\; \prod_{y=0}^N\mathbb{E}_{\mu_N}\bigg[
\prod_{k=1}^{\eta_0(y)}\mathbb{E}_{\mu_N}\Big[e^{- \zeta(X_t^{y,k})}\Big]
\bigg] \notag \\ 
&\;=\; \prod_{y=0}^N\mathbb{E}_{\mu_N}\bigg[
\mathbb{E}_{\mu_N}\Big[e^{- \zeta(X_t^{y,1})}\Big]^{\eta_0(y)}
\bigg]\,,\label{eq_algo}
\end{align}
where in the last equality we used the fact that random walks starting from $y$ are identically distributed.
Recall the  definition of the initial measure $\mu_N$ in~\eqref{mu_N}
and  the standard fact that a random variable $Z\sim \mathrm{Poisson}(\lambda)$ has a probability generating function given by $\bb E[s^{Z}] = \exp\{\lambda(s-1)\}$. We can thus rewrite \eqref{eq_algo} as follows:
\begin{align}\label{eq_algo_2}
\prod_{y=0}^N\exp\Big(
\lambda(y)\big(\mathbb{E}_{\mu_N}[e^{-\zeta(X_t^{y,1})}]-1\big)
\Big)\;=\; 
\exp\bigg[\sum_{y=0}^N
\lambda(y)\Big(\mathbb{E}_{\mu_N}[e^{-\zeta(X_t^{y,1})}]-1\Big)
\bigg],
\end{align}
where 
\begin{equation}\label{Poisson}
\lambda(y) \;=\; \begin{cases}
	\frac{N^{\theta}}{\alpha}\gamma(\pfrac{0}{N}),& \text{ if } y=0,\vspace{3pt}\\
	\gamma(\pfrac{y}{N}),& \text{ if } y=1, \ldots, N\\
\end{cases}
\end{equation}
is the parameter of the Poisson at the site $y$, see  \eqref{mu_N}. We now rewrite the right-hand side of \eqref{eq_algo_2} as
\begin{align}\label{eq_algo_3}
\exp\bigg[\sum_{y=0}^N
\lambda(y)\sum_{z=0}^N \big\{e^{-\zeta(z)}\!-\!1\big\}\, p_t(y,z) 
\bigg] = \exp\bigg[\sum_{z=0}^N
\big\{e^{-\zeta(z)}\!-\!1\big\}\sum_{y=0}^N \lambda(y)\, p_t(y,z)
\bigg] ,
\end{align}
where $p_t(y,z)= \mathbb{P}_{\mu_N} ( X^{y,1}_t = z)$  is  the transition probability of the random walk.  Now we want to write the rightmost sum on  \eqref{eq_algo_3} as an expectation with respect to $p_t(y,z)$. To do so, we note that the probability measure $\pi$ on $\{0,\ldots, N\}$ given by
\begin{equation}\label{invariant_measure}
\pi(0)\;=\; \frac{\frac{N^\theta}{\alpha}}{\frac{N^\theta}{\alpha}+N}\quad \text { and }\quad
\pi(x) \;=\; \frac{1}{\frac{N^\theta}{\alpha}+N}\quad \text{ for } x=1,\ldots, N
\end{equation}
is reversible for $p_t(y,z)$, that is, $\pi(y)p_t(y,z) = \pi(z) p_t(z,y)$. Applying this fact, we can rewrite  \eqref{eq_algo_3} as
\begin{align*}
& \exp\bigg[\sum_{z=0}^N
\big\{e^{-\zeta(z)}-1\big\}\sum_{y=0}^N \lambda(y)\,\frac{\pi(z)}{\pi(y)} p_t(z,y) 
\bigg] \\
&\;=\;  \exp\bigg[\sum_{z=0}^N
\big\{e^{-\zeta(z)}-1\big\}\pi(z)\sum_{y=0}^N  p_t(z,y) \,\frac{\lambda(y)}{\pi(y)}
\bigg]. 
\end{align*}
Renaming the variable $z$ to $x$, we arrive at the identity
\begin{equation}
\label{Laplace_transform}
\begin{split}
	\Phi_{\eta_t}(\zeta)
	&\;=\; \exp\bigg[\sum_{x=0}^N
	\big\{e^{-\zeta(x)}-1\big\}\pi(x)\sum_{y=0}^N  p_t(x,y) \,\frac{\lambda(y)}{\pi(y)}
	\bigg].
\end{split}
\end{equation}
\begin{remark}
Up to this point, we have not yet used the expression \eqref{Poisson} for the Poisson parameter $\lambda(x)$. 
Assuming that
\begin{equation*}
	\lambda(y) \;=\; \begin{cases}
		\frac{N^{\theta}}{\alpha}\cdot \lambda,& \text{ if } y=0,\vspace{3pt}\\
		\lambda,& \text{ if } y=1, \ldots, N\\
	\end{cases}
\end{equation*}
for a constant $\lambda>0$ (here we are abusing notation), one can infer from \eqref{Laplace_transform} that the measure $\nu_\lambda$ defined in \eqref{nu_lambda} is invariant. In the next section we will prove  that $\nu_\lambda$ is  actually reversible, as stated in Proposition~\ref{prop:reversible_RW}. 
\end{remark}
Note that the identity \eqref{Laplace_transform} already tells us that, for any time $t>0$, the distribution of $\eta_t$ is a Poisson product measure whose parameter at the site $x$ is given by
\begin{equation}\label{psi_def}
      \psi^N_t(x) \;:=\; \pi(x)\sum_{y=0}^N  p_t(x,y) \,\frac{\lambda(y)}{\pi(y)}\,.
\end{equation}
In the next proof we estimate this parameter.
  
\begin{proof}[Proof of Theorem~\ref{thm:propagation_local}]
Fix $0<u<1$ and  any positive integer $k$. Since we are going to take the limit as $N\to \infty$, assume without loss of generality that $\lfloor uN\rfloor >k$ so the random vector 
\[\eta_{t}^{k, u} := \big(\eta_{t}(\lfloor uN\rfloor - k), \ldots, \eta_{t}(\lfloor uN\rfloor + k)\big)\]
does not contain the site zero.  By the Laplace transform formula \eqref{Laplace_transform}, choosing $\zeta(x)=0$ for $x\notin \{\lfloor uN\rfloor - k, \ldots, \lfloor uN\rfloor + k\}$ we infer that
\begin{align*}
	\Phi_{\eta_t^{k, u}}(\zeta)
	& \;:=\;
	\mathbb{E}_{\mu_N}\bigg[\exp\Big\{-\sum_{x=\lfloor uN\rfloor - k}^{\lfloor uN\rfloor + k} \zeta(x)\eta_t(x)\Big\}\bigg]\\
	& \;=\;  \exp\bigg[\sum_{x=\lfloor uN\rfloor - k}^{\lfloor uN\rfloor + k}
	\big\{e^{-\zeta(x)}-1\big\}\,\psi^N_t(x)
	\bigg].
\end{align*}
Since the convergence of the Laplace transform characterizes the weak convergence of probability measures concentrated on  Cartesian products of non-negative half-lines, our proof reduces to characterize the limit of $\psi^N_t(x)$ defined in \eqref{psi_def}, for $ x \in \{\lfloor uN\rfloor - k, \ldots, \lfloor uN\rfloor + k\}$ as $N\to \infty$. 
Recalling  the Poisson parameters~\eqref{Poisson} and the invariant measure~\eqref{invariant_measure} and having in mind that $x\neq 0$, one can check that 
\begin{equation}\label{eq:expectation}
	\psi^N_t(x) \;=\;  \sum_{y=0}^N  p_t(x,y)\gamma\big(\pfrac{y}{N}\big) \;=\; E \Big[ \gamma\Big(\frac{X_t^{x,1}}{N}\Big)\Big]\,,
\end{equation}
where $E$ is the expectation with respect to the random walk $X_t^{x,1}$ that starts at $x\neq 0$. This is the only step of the proof where the continuity of $\gamma$, assumed in Theorem~\ref{thm:propagation_local}, is used. Since $\gamma$ is a bounded continuous function, by Proposition~\ref{prop:RW} and Portmanteau's Theorem,  for any $x\in \{\lfloor uN\rfloor - k, \ldots, \lfloor uN\rfloor + k\}$,
\begin{equation*}
	E \Big[ \gamma\Big(\frac{X_t^{x}}{N}\Big)\Big] \longrightarrow P_t^\theta \gamma(u)\quad \text{ as }N\to\infty\,,
\end{equation*}
where $P^\theta_t$ is, as explained in Proposition~\ref{prop:RW}:\smallskip
\begin{itemize}
	\item if $\theta\in [0,1)$, the semigroup of  a  Brownian motion on $[0,1]$ reflected both at $0$ and $1$.  \smallskip
	\item if $\theta = 1$, the semigroup of a Brownian motion on $[0,1]$ sticky at $0$ and reflected at $1$.\smallskip
	\item if $\theta \in (1,\infty)$, the semigroup of a Brownian motion on $[0,1]$ absorbed at $0$ and reflected at $1$.\smallskip
\end{itemize}
For the three semigroups above, it is true that $P_t^\theta f\in \mf D(\bb L_\theta)$ for any time $t>0$ and for any $f\in C[0,1]$. 
This implies that $P_t^\theta \gamma(u)$ satisfies the boundary conditions of \eqref{eq_1}, 
\eqref{eq_2} or \eqref{eq_3} if, respectively, $\theta\in [0,1)$, $\theta = 1$ or $\theta \in (1,\infty)$. Moreover, since in any case the process is a Brownian motion in $[0,1]$, the Feller semigroup $P_t^\theta \gamma(u)$ satisfies the heat equation with initial condition $\gamma$. This concludes the proof. 
\end{proof}

\begin{remark}\label{rmk:rmk3.3}
Although  the  Poisson parameter  at zero was used in the proof above in \eqref{eq:expectation}, we did not explicitly study the behavior at  the site zero. For the sake of completeness, see that 
\begin{equation*}
	\psi^N_t(0) \;=\; \pi(0)\sum_{y=0}^N  p_t(0,y) \,\frac{\lambda(y)}{\pi(y)} \;=\; \frac{N^\theta}{\alpha}   E \Big[ \gamma\Big(\frac{X_t^{0,1}}{N}\Big)\Big]\,,
\end{equation*}
which explodes in the same order as  $\frac{N^\theta}{\alpha} P^\theta_t\gamma(0)$.
\end{remark}

\section{Proof of hydrodynamics for IRW}\label{s4}
\subsection{Reversible measure}
We start by proving Proposition~\ref{prop:reversible_RW} which says that $\nu_\lambda$ is reversible for IRW. 
\begin{proof}[Proof of Proposition~\ref{prop:reversible_RW}]
The statement is equivalent to the identity
\begin{equation}\label{eq:reversible_rw_L}
	\int g(\eta) \msf{L}_N f(\eta) \,d\nu_\lambda(\eta) \;=\; \int f(\eta) \msf{L}_N g(\eta) \,d\nu_\lambda(\eta)  
\end{equation}
for every $f, g :\Omega_N\to \bb R$, see \cite{kl}.
It suffices to verify \eqref{eq:reversible_rw_L} for the bond $\{0,1\}$ since the dynamics over the remaining bonds is the same as of symmetric IRW, for which a product Poisson measure of constant parameter is known to be reversible, see \cite{kl}.
  
Recall \eqref{nu_lambda}. For ease of notation, denote $\psi_0= \frac{N^\theta}{\alpha}\lambda$ and $\psi_1=\lambda$, so
\[
\nu_\lambda \;=\;\mathrm{Poisson}(\psi_0) \otimes\bigotimes_{x=1}^N \mathrm{Poisson}(\psi_1)\,.
\]
Let $\xi$ be the configuration obtained from $\eta$ by moving one particle from the site $0$ to the site $1$, if possible. That is, $\xi=\eta^{0,1}$ as in \eqref{eq:eta_xy_rw}.
Thus,  assuming that $k\geq 1$, we have $\eta(0) = k$ and $\eta(1) = \ell $ if, and only if, $\xi(0) = k-1$ and $\xi(1) = \ell +1$. 
On the other hand,  $\eta(0) = 0$ if, and only if, $\eta = \xi$. 
The Radon-Nikodym derivative of this transformation on the set $A=\{\eta:\eta(0)\geq 1\}$ is given by
\begin{equation*}
	\frac{\nu_\lambda(\eta)}{\nu_\lambda(\xi)} = \frac{e^{-\psi_0} \psi_0^{k} }{k!} \cdot \frac{e^{-\psi_1}\psi_1^{\ell } }{\ell!} \cdot \frac{(k-1)!}{e^{-\psi_0}\psi_0^{k-1} } \cdot \frac{(\ell +1)!}{ e^{-\psi_1}\psi_1^{\ell +1}} = \frac{\psi_0}{\psi_1} \cdot \frac{\ell +1}{k} = \frac{\psi_0}{\psi_1} \cdot \frac{\xi(1)}{\xi(0)+1}\,.
\end{equation*}

We compute the contribution of the jump $0\to 1$ by performing the change of variables $\xi = \eta^{0,1}$ as follows. Note that the integrand in the first integral below vanishes in the set $\{\eta:\eta(0)=0\}$. Then the contribution of that jump to the left-hand side of \eqref{eq:reversible_rw_L} is equal to
\begin{align*}
	& \int g(\eta) \frac{\alpha}{N^\theta}\eta(0) \Big[f(\eta^{0,1}) - f(\eta)\Big] d\nu_\lambda(\eta) \\
	& = \sum_{k\geq 1} \sum_{\substack{\eta \in \Omega_N:\\\eta(0) = k}} g(\eta) \frac{\alpha}{N^\theta}\eta(0)\Big[f(\eta^{0,1}) - f(\eta)\Big] \nu_\lambda(\eta)\\
	& = \sum_{k\geq 1} \sum_{\substack{\xi \in \Omega_N:\\ \xi(0) = k-1}} g(\xi^{1,0}) \frac{\alpha}{N^\theta} (\xi(0)+1) \Big[f(\xi) - f(\xi^{1,0})\Big] \frac{\nu_\lambda(\eta)}{\nu_\lambda(\xi)} \nu_\lambda(\xi)\\
	& = \sum_{k\geq 1} \sum_{\substack{\xi \in \Omega_N:\\ \xi(0) = k-1}} g(\xi^{1,0}) \frac{\alpha}{N^\theta}(\xi(0)+1) \Big[f(\xi) - f(\xi^{1,0})\Big] \frac{\psi_0}{\psi_1}\frac{\xi(1)}{(\xi(0)+1)} \nu_\lambda(\xi) \\
	& = \sum_{k\geq 0} \sum_{\substack{\xi \in \Omega_N:\\ \xi(0) = k}} \frac{\alpha}{N^\theta}\,\frac{\psi_0}{\psi_1} g(\xi^{1,0}) \xi(1) \Big[f(\xi) - f(\xi^{1,0})\Big] \nu_\lambda(\xi) \\
	& = \int \frac{\alpha}{N^\theta}\,\frac{\psi_0}{\psi_1} g(\xi^{1,0}) \xi(1) \Big[f(\xi) - f(\xi^{1,0})\Big] d\nu_\lambda(\xi)\,.
\end{align*}
Now we can check the detailed balance condition related to the bond $\{0,1\}$, i.e. :
\begin{align*}
	& \int g(\eta) \frac{\alpha}{N^\theta}\eta(0) \big[f(\eta^{0,1}) - f(\eta)\big] d\nu_\lambda + \int g(\eta) \eta(1) \big[f(\eta^{1,0}) - f(\eta)\big] d\nu_\lambda \\
	& - \bigg{[} \int f(\eta) \frac{\alpha}{N^\theta}\eta(0) [g(\eta^{0,1}) - g(\eta)] d\nu_\lambda + \int f(\eta) \eta(1) \big[g(\eta^{1,0}) - g(\eta)\big] d\nu_\lambda \bigg{]}\\
	& = \int \frac{\alpha}{N^\theta}\,\frac{\psi_0}{\psi_1} g(\eta^{1,0}) \eta(1) \big[f(\eta) - f(\eta^{1,0})\big] d\nu_\lambda + \int g(\eta) \eta(1) \big[f(\eta^{1,0}) - f(\eta)\big] d\nu_\lambda \\
	& -  \int \frac{\alpha}{N^\theta}\,\frac{\psi_0}{\psi_1} f(\eta^{1,0}) \eta(1) \big[g(\eta) - g(\eta^{1,0})\big] d\nu_\lambda - \int f(\eta) \eta(1) \big[g(\eta^{1,0}) - g(\eta)\big] d\nu_\lambda\,, 
\end{align*}
which vanishes since  $\psi_0 = \frac{N^\theta}{\alpha}\,\psi_1$, concluding the proof.
\end{proof}

\subsection{Scaling limit}\label{sec:scaling_limit}
Let  $\{\pi^N_t:t\in [0,T]\}$ be the \textit{empirical measure} defined by  
\begin{equation}\label{empirical}
\pi^N_t \;=\; \frac{1}{N}\sum_{x=1}^N \eta_t(x) \delta_{\frac{x}{N}}\,,
\end{equation}
which characterizes the spatial density of the particles of the process as embedded in the interval $[0,1]$. 
Note that  the empirical measure \textit{does not} include the zero site, and it is a random element in the Skorokhod space $D([0,T]; \mc M)$ of càdlàg trajectories, where $\mc M$ is the space of finite non-negative measures on $[0,1]$. The base space is kept as $[0,1]$, even though $\pi^N_t$ does not give mass to the zero site, since the limiting measure is $\rho(t,u)du$ on $[0,1]$ and the test functions are taken in $C[0,1]$.

Let $\mu_N$ be the Poisson product measure introduced in \eqref{mu_N}, and let  $\bb Q_{\mu_N}^{\theta,N}$ be the measure in the path space $D([0,T]; \mc M)$ induced by the measure  and 
the empirical measure $\pi^N_t$ introduced in \eqref{empirical} above. Recall that $\mu_N$ is defined in terms of a fixed continuous profile $\gamma:[0,1]\to [0,\infty)$.
\begin{definition}\label{def:Qtheta}
Let $\bb Q^{\theta}$ be
the probability measure in $D([0,T]; \mc M)$ concentrated on the
deterministic path $\pi(t,du) = \rho (t,u)du$, where $\rho (t,u)$ is:
\begin{itemize}
	\item if $\theta\in [0,1)$, the unique strong solution to the  heat equation   \eqref{eq_1}.\smallskip
	
	\item if $\theta =  1$, the unique  strong solution to the heat equation  \eqref{eq_2}.\smallskip
	
	\item if $\theta\in (1,\infty)$, the unique strong solution to the heat equation   \eqref{eq_3}.
\end{itemize}
\end{definition}
Since the propagation of local equilibrium implies the hydrodynamic limit, as a consequence of Theorem~\ref{thm:propagation_local} we have the following:
\begin{corollary}\label{cor:local}
As $N\uparrow\infty$, the sequence of probability measures $\{\bb
Q_{\mu_N}^{\theta,N}:N\geq{1}\}$ converges weakly to $\bb Q^{\theta}$.
\end{corollary}

\subsection{A second characterization of \texorpdfstring{$\bb Q^\theta$}{Q theta} for \texorpdfstring{$\theta=1$}{theta=1} and \texorpdfstring{$\theta\in (1,\infty)$}{theta>1}}\label{sub:limit}

By Corollary~\ref{cor:local}, we know that, as $N\to\infty$, the sequence of measures $\{\bb Q^{\theta,N}_{\mu_N}: N\geq{1}\}$ converges weakly to  $\bb Q^\theta$, which is a delta of Dirac on the trajectory $\pi(t,du) = \rho(t,u) du$, whose density
$\rho(t,u)$ is a weak solution to \eqref{eq_1}, \eqref{eq_2} or \eqref{eq_3} according to $\theta\in[0,1)$, $\theta=1$ or $\theta\in(1,\infty)$, respectively. 

Our task now is to give a second characterization of $\bb Q^\theta$ for $\theta=1$ and $\theta\in(1,\infty)$ via the convergence of Dynkin's martingale.
Namely, we will show that, for $\theta =1$, the density $\rho(t,u)$ is (also) the unique  weak solution to \eqref{weak_solution_non_local} and, for $\theta >1$,  it  is (also)  the unique  weak solution to \eqref{weak_solution_Dirichlet}, providing a rigorous proof of what was already discussed in the heuristic arguments given in Remarks~\ref{remark:2.8} and~\ref{remark:2.9}. 

A second characterization of $\bb Q^\theta$ for $\theta\in[0,1)$  is  omitted because our approach would lead to the same hydrodynamic equation already found in the Theorem~\ref{thm:propagation_local}  (the heat equation with Neumann boundary conditions at both $0$ and $1$), thus  providing no extra information.

Assume for the moment the uniqueness of weak solutions of the PDE \eqref{weak_solution_non_local} and \eqref{weak_solution_Dirichlet}, which is the topic of the next subsection.
Recall \eqref{generator_rw}. 
For any function $H$, the process
\begin{equation}\label{Martingal}
M^{N}_{t}(H)\;:=\;\<\pi^{N}_{t}, H\>- \<\pi^{N}_{0}, H\>-\int_{0}^{t}\msf L_{N}\<\pi^{N}_{s},H\>\,ds
\end{equation}
is a martingale with respect to the natural filtration $\mathcal{F}_t:=\sigma(\eta_s: s\leq{t})$, the so-called Dynkin's martingale, see \cite{kl}. 
Assume that $H\in C^2[0,1]$ and $H(0)=H'(1)=0$.
By the carré-du-champ formula, the quadratic variation of the martingale $M^{N}_{t}(H)$  is given by
\begin{equation}\label{eq:quadratic}
\begin{split}
	\<M^{N}(H)\>_t \;=\; &  \int_{0}^{t} \Big(\frac{\alpha}{N^{\theta}} \eta_s(0) + \eta_s(1)\Big)\Big[H(\pfrac{1}{N})-H(\pfrac{0}{N})\Big]^2ds \\
	&+\int_{0}^{t}\sum_{x=1}^{N-1}  \big(\eta_{s}(x)+\eta_{s}(x+1)\big)
	\Big[H(\pfrac{x+1}{N})-H(\pfrac{x}{N})\Big]^2 ds\,.
\end{split}
\end{equation}
Consider the partial order $\lesssim$ in $\Omega_N$ defined by
\begin{equation}\label{partial}
\eta_1 \lesssim \eta_2 \quad\Leftrightarrow\quad \eta_1(x)\leq \eta_2(x),\, \forall \,x=0,1,\ldots, N.
\end{equation}
The initial measure $\mu_N$ defined in \eqref{mu_N} is stochastically dominated by the product measure
\begin{equation}\label{mu_tilde}
\widetilde{\mu}_N \;:=\; \mathrm{Poisson}\Big(\frac{N^\theta}{\alpha} \cdot \Vert \gamma\Vert_\infty\Big)\otimes \bigotimes_{x=1}^N\mathrm{Poisson}\big(\Vert \gamma\Vert_\infty\big)\,,
\end{equation}
which, by Proposition~\ref{prop:reversible_RW}, is an invariant measure. IRW is an attractive system with respect to \eqref{partial}, so this partial order is preserved in time, see \cite[Chapter 2]{Liggett}. Since the quadratic variation \eqref{eq:quadratic} is an increasing function with respect to the partial order \eqref{partial}, we have that
\begin{align*}
\bb E_{\mu_N}^{\theta, N}\Big[\<M^{N}(H)\>_t \Big]& \;\leq\;  \bb E_{\widetilde{\mu}_N}^{\theta, N}\Big[\<M^{N}(H)\>_t \Big]\\
&\;=\;  \bb E_{\widetilde{\mu}_N}^{\theta, N}\bigg[ \int_{0}^{t} \Big(\frac{\alpha}{N^{\theta}} \eta_s(0) + \eta_s(1)\Big)\Big[H(\pfrac{1}{N})-H(\pfrac{0}{N})\Big]^2ds \\
&\qquad\qquad+\int_{0}^{t}\sum_{x=1}^{N-1}  \big(\eta_{s}(x)+\eta_{s}(x+1)\big)
\Big[H(\pfrac{x+1}{N})-H(\pfrac{x}{N})\Big]^2 ds\bigg]\\
&\;\leq\;  \frac{2t}{N}\|\partial_u H\|_{\infty}^2 \cdot \Vert \gamma\Vert_\infty \,,
\end{align*}
which goes to zero as $N\to\infty$ in view of $H\in C^2[0,1]$. Since $\big(M^{N}_{t}(H)\big)^2 - \<M^{N}(H)\>_t$ is a zero mean martingale (see for instance \cite[Appendix 1]{kl}), by Doob's inequality, for every $\delta>0$,
\begin{equation*}
\lim_{N\rightarrow\infty}\bb P_{\mu_N}^{\theta, N}\bigg[\sup_{0\leq t\leq T} |M^{N}_{t}(H)|>\delta\bigg]\;=\;0\,,
\end{equation*}
hence the sequence of martingales $\{M^{N}_{t}(H):t\in [0,T]\}_{N\geq 1}$ converges in distribution to zero in the Skorokhod topology. Our goal now is to understand the limit of each term in the martingale \eqref{Martingal} aiming to achieve in the limit  the integral equation \eqref{integral_weak_solution_non_local} for $\theta=1$, or the integral equation \eqref{integral_weak_solution_Dirichlet_homogeneous} for $\theta\in(1,\infty)$.

The definition \eqref{mu_N} of $\mu_N$ implies that the sequence of time constant processes $\{\<\pi^{N}_{0}, H\>:t\in[0,T]\}_{N\geq 1}$  converges in distribution to the time-constant process\break $\{\<\gamma, H\>:t\in [0,T]\}$. Moreover, Corollary~\ref{cor:local} implies that the sequence of process 
$\{\<\pi^{N}_{t}, H\>:t\in[0,T]\}_{N\geq 1}$ converges in distribution to  $\{\<\rho_{t}, H\>:t\in[0,T]\}$.
Thus, it only remains to study the integral term in~\eqref{Martingal}. Performing elementary calculations, 
\begin{align*}
\msf{L}_N \<{\pi}_s^N, H\> \;=\; & N \cdot H\big(\pfrac{1}{N}\big)\Big[\frac{\alpha}{N^\theta}\eta(0) - 2\eta(1) + \eta(2)\Big]\\
&+ N\cdot \sum_{x=2}^{N-1} {H}\big(\pfrac{x}{N}\big)\big[\eta_s(x-1) - 2\eta_s(x) + \eta_s(x+1)\big]\\
&+ N \cdot H\big(\pfrac{N}{N}\big)\big[-\eta_s(N) + \eta_s(N-1) \big]\,. 
\end{align*}
Since $H(0)=0$, this can be rewritten as
\begin{align}
&\frac{\alpha \eta_s(0)}{N^\theta}\cdot N\Big[ H\big(\pfrac{1}{N}\big) - H\big(\pfrac{0}{N}\big)\Big] \label{eq:int_1}\\ 
&+ \eta_s(1)\cdot N\Big[H\big(\pfrac{0}{N}\big) - 2\,H\big(\pfrac{1}{N}\big) + H\big(\pfrac{2}{N}\big) \Big] \label{eq:int_2}\\[0.25em]
&+ \frac{1}{N}\sum_{x=2}^{N-1} \eta_s(x)\cdot N^2\Big[H\big(\pfrac{x+1}{N}\big)+H\big(\pfrac{x-1}{N}\big) - 2H\big(\pfrac{x}{N}\big)  \Big] \label{eq:int_3}\\
& + \eta_s(N)\cdot N\Big[H\big(\pfrac{N-1}{N}\big) - H\big(\pfrac{N}{N}\big) \Big]\,. \label{eq:int_4}  
\end{align}
 Keep in mind that the attractiveness property discussed at \eqref{partial} and \eqref{mu_tilde} assures domination. Since $H'(1)=0$, the Dominated Convergence Theorem guarantees that the integral in time from $0$ to $t$ of \eqref{eq:int_4} converges to zero as $N\to\infty$.
By a similar argument, the sequence of processes $\{\frac{1}{N}\eta_t(N):t\in [0,T]\}_{N\geq 1}$ converges to zero in probability. Thus, since $H\in C^2[0,1]$, applying the Corollary~\ref{cor:local}, the integral in time of \eqref{eq:int_2} plus \eqref{eq:int_3} converges in distribution (thus in probability) to the constant
\begin{equation*}
\int_0^t \<\rho_s, H''\> \,ds\,.
\end{equation*}
Since the Markov process $\{\eta_t: t\in [0,T]\}$ on $\Omega_N$ is conservative, we can rewrite the expression \eqref{eq:int_1} as
\begin{align*}
&  \alpha N\Big[ H\big(\pfrac{1}{N}\big)  - H\big(\pfrac{0}{N}\big)\Big]\Bigg[\frac{1}{N^\theta} \sum_{x=0}^N \eta_s(x) - \frac{1}{N^\theta} \sum_{x=1}^N \eta_s(x) \Bigg] \\
& \;=\;  \alpha N\Big[ H\big(\pfrac{1}{N}\big)  - H\big(\pfrac{0}{N}\big)\Big]\Bigg[\frac{1}{N^\theta} \sum_{x=0}^N \eta_0(x) - \frac{1}{N^\theta} \sum_{x=1}^N \eta_s(x) \Bigg].
\end{align*}
From the definition \eqref{mu_N} of $\mu_N$ and an application of the law of large numbers for triangular arrays (see \cite{Durrett} for instance), we are led to 
\begin{equation}\label{eq:massa_total}
\frac{1}{N^\theta} \sum_{x=0}^N \eta_0(x) \;\longrightarrow\; 
\begin{cases}
	\displaystyle\frac{\gamma(0)}{\alpha}+ \int_0^1 \gamma(u)du \;=:\; M, & \text{ if } \theta=1,\\
	\displaystyle\frac{\gamma(0)}{\alpha}, & \text{ if } \theta\in(1,\infty)\\
\end{cases}
\end{equation}
in probability as $N\to\infty$. From the Corollary~\ref{cor:local},
we have the following convergence (of processes)
\begin{equation}\label{eq:massa_bulk}
\frac{1}{N^\theta} \sum_{x=1}^N \eta_s(x) = N^{1-\theta}\<\pi^N_s, 1\> \longrightarrow
\begin{cases}
	\<\rho_s, 1\> =    \int_0^1 \rho_s(u)du, & \text{if } \theta=1,\\
	0, & \text{if } \theta\in(1,\infty)
\end{cases}
\end{equation}
in probability as $N\to\infty$, for every $s\in[0,T]$. By \eqref{eq:massa_total} and \eqref{eq:massa_bulk}, we deduce that
\eqref{eq:int_1} converges in probability to
\begin{equation*}
\alpha H'(0) \times  \begin{cases}
	\displaystyle M - \int_0^1 \rho_s(u)du, & \text{ if } \theta=1,\\
	\displaystyle\frac{\gamma(0)}{\alpha}, & \text{ if } \theta\in(1,\infty).
\end{cases}
\end{equation*}
Putting together the convergences of \eqref{eq:int_1}, \eqref{eq:int_2}, \eqref{eq:int_3} and \eqref{eq:int_4},
we get 
\begin{equation*}
\int_0^t\! \msf{L}_N \<{\pi}_s^N, H\> ds \to 
\int_0^t\! \<\rho_s, H''\> ds + \alpha H'(0)\! \! \int_0^t\!\! ds \begin{cases}
	M - \int_0^1 \rho_s(u)du, & \text{if } \theta=1\\
	\frac{\gamma(0)}{\alpha}, & \text{if } \theta\in(1,\infty)\\
\end{cases}
\end{equation*}
in probability. We have therefore obtained that, for any $H\in C^2[0,1]$ such that $H(0)=H'(1)=0$ and for $\theta=1$, it holds
\begin{equation*}
\big\<\rho_t, H\big\> - \big\<\gamma, H\big\>  \;=\;  \int_0^t \big\< \rho_s,  \p^2_{uu}  H \big\>\, ds  + \int_0^t \alpha \bigg(M- \int_0^1\rho_s(u)du\bigg)  H'(0) \, ds\,;
\end{equation*}
while, for $\theta\in (1,\infty)$, it holds
\begin{equation*}
\big\<\rho_t, H\big\> - \big\<\gamma, H\big\>  \;=\;  \int_0^t \big\< \rho_s,  \p^2_{uu}  H \big\>\, ds  + \int_0^t  \gamma(0)  H'(0) \, ds\,.
\end{equation*}
These are, respectively, the integral equations \eqref{integral_weak_solution_non_local} and \eqref{integral_weak_solution_Dirichlet_homogeneous}. This result, together with the uniqueness of weak solutions (to be proved in the next subsection),  characterizes  $\rho(t,u)$ for $\theta=1$ or  $\theta\in(1,\infty)$ as the unique weak solution to \eqref{weak_solution_non_local} or to \eqref{weak_solution_Dirichlet}, respectively, and henceforth concludes  the proof of Theorem~\ref{thm:hydrolimit_rws}.

\subsection{Uniqueness of weak solutions}\label{sub:uniqueness}

In this subsection we prove  uniqueness of weak solutions to    \eqref{weak_solution_non_local} and \eqref{weak_solution_Dirichlet}.
Let $\Psi_k:[0,1]\to \bb R$ be given by
\begin{equation}\label{Psi_k}
\Psi_k(u) \;:=\; \sqrt{2}\sin\Big(\pi\big(k+\pfrac{1}{2}\big)  u\Big),\qquad k=0,1,2,\ldots
\end{equation}
These functions are the solutions of the following Sturm-Liouville problem associated to  the Laplacian operator with Dirichlet and Neumann boundary conditions
\begin{equation}\label{Liouville}
\begin{cases}
	-f''(u) = \lambda f(u), & \text{ for } u\in (0,1),\\
	f(0) = f'(1) = 0, 
\end{cases}
\end{equation}
having $\lambda_k= \pi^2\big(k+\pfrac{1}{2}\big)^2$ as  the $k^{\text{th}}$-eigenvalue. Moreover, the set $\{ \Psi_k\}_{k\geq 0}$ is an orthonormal complete basis of $L^2[0,1]$, see \cite{Titchmarsh}, for instance, on the subject.
We start below showing uniqueness of weak solutions to \eqref{weak_solution_non_local}.

\begin{proof}[Proof of Proposition~\ref{prop:uniqueness_2.9}]
Let $\rho^1$ and $\rho^2$ be weak solutions of \eqref{weak_solution_non_local}. For  $\xi=\rho^1-\rho^2$, it holds
\begin{equation*}
\<\xi_t, H\>  \;=\; \int_{0}^{t} \<\xi_s, H''\>\, ds - \alpha H'(0) \int_0^t  \<\xi_s, 1 \> \, ds \,,
\end{equation*}
for any $H\in C^2[0,1]$ such that $H(0)=H'(1)=0$. In particular, 
\begin{equation}\label{eq:derivative}
\<\xi_t, \Psi_k\>  \;=\;- \lambda_k \int_{0}^{t} \<\xi_s, \Psi_k\>\, ds - \alpha \Psi'_k(0) \int_0^t  \<\xi_s, 1 \> \, ds\,. 
\end{equation}
Consider the energy functional 
\begin{equation}
\mc E(t)\;=\;\sum_{k\geq 0}\frac{\< \xi_t , \Psi_k\>^2}{2\lambda_k}\,,
\end{equation}
which by \eqref{eq:derivative} satisfies
\begin{equation}\label{eq:energy_functional}
\mc E'(t) \;=\; - \sum_{k\geq 0} \Big[\<\xi_t, \Psi_k\>^2 +\frac{\alpha \Psi_k'(0)}{\lambda_k}\<\xi_t, \Psi_k\>\<\xi_t, 1\>\Big]. 
\end{equation}
The fact that $\{ \Psi_\ell\}_{\ell\geq 0}$ is  an orthonormal complete basis of $L^2[0,1]$ implies that $1 = \sum_{\ell\geq 0} \<1, \Psi_\ell\>\Psi_\ell$, and a simple calculation gives us that 
\[\<1, \Psi_\ell\> \;=\; \frac{\sqrt{2}}{\pi(\ell+\frac{1}{2})} \;=\; \frac{2}{\Psi_\ell'(0)}\,,\]
hence
\begin{equation}\label{eq:one}
1 \;=\; \sum_{\ell\geq 0} \frac{2}{\Psi_\ell'(0)} \Psi_\ell\,.
\end{equation}
Applying \eqref{eq:one} into \eqref{eq:energy_functional} yields
\begin{align*}
\mc E'(t) & \;=\; - \sum_{k\geq 0} \Big[\<\xi_t, \Psi_k\>^2 +\frac{\alpha \Psi_k'(0)}{\lambda_k}\<\xi_t, \Psi_k\>\sum_{\ell\geq 0} \frac{2}{\Psi_\ell'(0)} \<\xi_t,\Psi_\ell\>\Big]\\
& \;=\; - \sum_{k\geq 0} \<\xi_t, \Psi_k\>^2  
- \alpha\bigg(\sum_{k\geq 0}\frac{\Psi_k'(0)}{\lambda_k}\<\xi_t, \Psi_k\>\bigg)\bigg(\sum_{\ell\geq 0}\frac{2}{\Psi_\ell'(0)}\<\xi_t, \Psi_\ell\>\bigg).
\end{align*}
Since
\begin{equation}\label{eq:inverse_Psi}
\frac{2}{\Psi_\ell'(0)}\;=\;\frac{\Psi_\ell'(0)}{\lambda_\ell}  \,,
\end{equation}
we deduce that
\begin{align}
\mc E'(t)  & \;=\; - \Vert \xi_t\Vert_2^2  
- \alpha\bigg(\sum_{k\geq 0}\frac{\Psi_k'(0)}{\lambda_k}\<\xi_t, \Psi_k\>\bigg)^2 \label{eq:a1}\\
& \;=\; - \Vert \xi_t\Vert_2^2  
- \alpha\bigg\<\xi_t, \sum_{k\geq 0}\frac{\Psi_k'(0)}{\lambda_k}\Psi_k\bigg\>^2 \label{eq:a2}\\
& \;=\; - \Vert \xi_t\Vert_2^2  
- \alpha\<\xi_t, 1\>^2\leq 0\,.\label{eq:a3}
\end{align}
Thus, $\mc E(t) = 0$ for all $t\geq 0$, which in its turn implies that $\<\xi_t,\Psi_k\>=0$ for all $t\geq 0$ and $k$.
Since the set $\{ \Psi_k\}_{k\geq 0}$ is an orthonormal complete basis of $L^2[0,1]$, by Parseval's formula we have that $\Vert \xi_t\Vert_2 =0$ for all $t\geq 0$, ensuring the uniqueness of weak solutions 
 of \eqref{weak_solution_non_local}. 
\end{proof}

\begin{remark}
At first sight, it seems that \eqref{eq:a1} is already enough   to conclude the proof, whose right-hand side is non-positive. However, it was necessary to arrive at \eqref{eq:a3}  to assure that the series  \eqref{eq:energy_functional} is indeed convergent.
\end{remark}

\begin{proof}[Proof of Proposition~\ref{prop:unique_4.3}]
Let $\rho^1$ and $\rho^2$ be weak solutions of \eqref{weak_solution_Dirichlet}. For $\xi=\rho^1-\rho^2$, it holds
\begin{equation}\label{eq:unique}
\<\xi_t, H\>  \;=\; \int_{0}^{t} \<\xi_s, H''\>\, ds\,.
\end{equation}
As before, consider the energy functional  $\mc E(t) = \sum_{k\geq 0}\frac{\< \xi_t , \Psi_k\>^2}{2\lambda_k}$,
which satisfies  $\mc E(0)=0$ and $\mc E(t)\geq 0$ for all $t>0$. By \eqref{eq:unique} and the fact that $\Psi_k$ are the solutions of the Sturm-Liouville problem \eqref{Liouville},  we deduce that 
\begin{equation*}
\mc E'(t)\;=\; -\sum_{k\geq 0}\< \xi_t , \Psi_k\>^2 \;\leq\; 0\,.
\end{equation*}
leading to the  uniqueness of weak solutions to \eqref{weak_solution_Dirichlet}.
\end{proof}

\section{Proof of hydrodynamics for SEP}\label{s5}

\subsection{Scaling limit}
Let  $\{\pi^N_t:t\in [0,T]\}$ be the \textit{empirical measure} defined by  
\begin{equation}\label{empirical_ex}
\pi^N_t \;=\; \frac{1}{N}\sum_{x=0}^N \eta_t(x) \delta_{\frac{x}{N}}\,,
\end{equation}
which characterizes the spatial density of the particles of the process as embedded in the interval $[0,1]$. 
The empirical measure  is a random element in the Skorokhod space $D([0,T]; \mc M)$ of càdlàg trajectories, where $\mc M$ is the space of finite non-negative measures on $[0,1]$. 

Unlike what we did in Subsection~\ref{sec:scaling_limit}, the definition above of the empirical measure now \textit{includes} the site $x=0$.  Nevertheless,  the mass at $x=0$ (the reservoir) does not play any role in the limit of \eqref{empirical_ex}, as we shall see. For $\theta=1$ and $\theta\in(1,\infty)$, this is due to the fact that test functions are assumed to satisfy the boundary condition $H(0)=0$. And for $\theta\in [0,1)$, this is due to the fact that the number of particles in the reservoir is of order $N^\theta\ll N$. 
On the other hand, although the site zero does not contribute to the limit of \eqref{empirical_ex}, such a   definition is essential in our calculations, helping to prove the tightness and to characterize the limit.

Let $\mu_N$ be defined as in \eqref{mu_N_ex} in terms of a continuous profile $\gamma:[0,1]\to [0,1]$. Let  $\bb Q_{\mu_N}^{\theta,N}$ be the measure on the path space $D([0,T]; \mc M)$ induced by the initial measure $\mu_N$  and 
the empirical measure $\pi^N_t$ introduced in \eqref{empirical_ex} above. 
\begin{definition}\label{def:Qtheta_ex}
Let $\bb Q^{\theta}$ be
the probability measure on $D([0,T]; \mc M)$ concentrated on the
deterministic path $\pi(t,du) = \rho (t,u)du$, where $\rho(t,u)$ is:
\begin{itemize}
	\item if $\theta\in [0,1)$, the unique weak solution to the heat equation   \eqref{heat_Neumann}.
    \smallskip
	
	\item if $\theta =  1$, the unique weak solution to the heat equation   \eqref{non_linear_Dirichlet}.
    \smallskip
	
	\item if $\theta\in (1,\infty)$, the unique weak solution to the heat equation  \eqref{heat_dirichlet}.

\end{itemize}
\end{definition}

Uniqueness of weak solutions is postponed to Subsection~\ref{sub:uniqueness_SEP}. The Theorem~\ref{thm:hydro_ex} is an immediate consequence of the next proposition.
\begin{proposition}\label{prop:Q_weakly}
As $N\uparrow\infty$, the sequence of probability measures $\{\bb
Q_{\mu_N}^{\theta,N}:N\geq{1}\}$ converges weakly to $\bb Q^{\theta}$.
\end{proposition}
Let us prove the proposition above subject to forthcoming results.
\begin{proof}[Proof of Proposition~\ref{prop:Q_weakly}]
In   Subsection~\ref{tight},
we show that, for any $\theta\in [0,\infty)$,  the sequence $\{\bb Q_{\mu_N}^{\theta, N}\}_{N\ge
1}$ is tight, thus relatively compact by Prokhorov's Theorem. In Subsection~\ref{subsec:charac_SEP} we show that, if $\bb Q_{*}^{\theta}$ is the limit  along a subsequence of $\{\bb Q_{\mu_N}^{\theta, N}\}_{N\ge
1}$, then $\bb Q_{*}^{\theta}$  is concentrated on trajetories $\rho(t,u)du$ such that $\rho(t,u)$ is a weak solution  to the   corresponding PDE according to the range of $\theta$. Uniqueness of weak solutions, presented in Subsection~\ref{sub:uniqueness_SEP}, implies that $\bb Q_{*}^{\theta}= \bb Q^{\theta}$, concluding the proof.
\end{proof}

We start  with Subsections~\ref{subsec:entropy} and \ref{sec:replacement}, which contain some technical results needed in the sequel.

\subsection{Reversible measure and entropy}\label{subsec:entropy}
Let us prove Proposition~\ref{prop:reversible}, which says that the measure
\begin{equation*}
\nu_p \;=\; \mathrm{Poisson}\Big(\frac{N^\theta}{\alpha}\cdot\frac{p}{1-p}\Big)\otimes \bigotimes_{x=1}^N \mathrm{Bernoulli}(p)
\end{equation*}
for $0<p<1$ is reversible for the Markov process defined by the generator \eqref{generator}.

\begin{proof}[Proof of Proposition~\ref{prop:reversible}]
It must be shown that for every $g,  f  \in L_{\nu_p}^2(\Omega_N)$,
\begin{equation*}
	\int  g  \msf{L}_N  f  \,d\nu_p \;=\;  \int  f  \msf{L}_N  g  \, d\nu_p\,,
\end{equation*}
see the Appendix of \cite{kl} for instance.  It suffices to verify the condition for the bond $\{0,1\}$ since the dynamics over the remaining bonds follow the symmetric exclusion part of the dynamics, for which the product of Bernoulli measures of constant parameter is known to be reversible, see \cite{kl}.

For the  jump across the bond $\{0,1\}$, fix a configuration $\eta$ with $\eta(0) \geq 1$ and $\eta(1) = 0$, and denote by $\xi$ the configuration obtained from $\eta$ by moving a particle from $0$ to $1$, if possible. That is, 
$\xi:=\eta^{0,1}$ as in \eqref{eq:eta_xy_ex}. Let $\lambda := \frac{N^\theta}{\alpha}\cdot\frac{p}{1-p}$. Since $\xi(0) = \eta(0)-1$ and $\xi(1) = 1$ in the set $A=\{\eta: \eta(0)\geq 1\}$, the Radon-Nikodym derivative in the set $A$ is given by
\begin{equation*}
	\frac{\nu_p(\eta)}{\nu_p(\xi)} \;=\; \frac{\frac{e^{-\lambda}\lambda^{\xi(0)+1} }{(\xi(0)+1)!}\cdot(1-p)}{\frac{e^{-\lambda}\lambda^{\xi(0)}}{\xi(0)!}\cdot p} \;=\; \frac{\lambda}{\xi(0)+1} \cdot \frac{1- p}{p} \;=\; \frac{N^\theta}{\alpha}\cdot\frac{1}{\xi(0)+1}\,.
\end{equation*}
Keeping in mind that the integrand in the first integral below vanishes outside the set $A$, we  consequently have that
\begin{align*}
	& \int  \frac{\alpha}{N^\theta}\eta(0) f (\eta)  (1-\eta(1)) \Big[ g (\eta^{0,1}) -  g (\eta)\Big] d\nu_p(\eta) \\
	& \;=\; \sum_{k\geq 1} \sum_{\substack{\eta \in \Omega_N:\\ \eta(0) = k}}  f (\eta) \frac{\alpha}{N^\theta}\eta(0)(1-\eta(1))\Big[ g (\eta^{0,1}) -  g (\eta)\Big] \nu_p(\eta)\\
	& \;=\; \sum_{k\geq 1} \sum_{\substack{\xi \in \Omega_N:\\ \xi(0) = k-1}}  f (\xi^{1,0}) \frac{\alpha}{N^\theta} (\xi(0)+1) \xi(1) \Big[ g (\xi) -  g (\xi^{1,0})\Big] \frac{\nu_p(\eta)}{\nu_p(\xi)} \nu_p(\xi)\\
	& \;=\; \sum_{k\geq 1} \sum_{\substack{\xi \in \Omega_N:\\ \xi(0) = k-1}}  f (\xi^{1,0}) \frac{\alpha}{N^\theta}(\xi(0)+1) \xi(1) \Big[ g (\xi) -  g (\xi^{1,0})\Big] \frac{N^\theta}{\alpha(\xi(0)+1)} \nu_p(\xi) \\
	& \;=\; \sum_{k\geq 0} \sum_{\substack{\xi \in \Omega_N:\\ \xi(0) = k}}  f (\xi^{1,0}) \xi(1) \Big[ g (\xi) -  g (\xi^{1,0})\Big] \nu_p(\xi) \\
	& \;=\; \int  f (\xi^{1,0}) \xi(1) \Big[ g (\xi) -  g (\xi^{1,0})\Big] d\nu_p(\xi)\,.
\end{align*}
With the change of variables above at hand, we can check the detailed balance condition related to the  $\{0,1\}$ bond:
\begin{align*}
	& \int  f (\eta) \frac{\alpha}{N^\theta}\eta(0) (1-\eta(1)) \big[ g (\eta^{0,1}) -  g (\eta)\big] d\nu_p + \int  f (\eta) \eta(1) \big[ g (\eta^{1,0}) -  g (\eta)\big] d\nu_p \\
	& - \int  g (\eta) \frac{\alpha}{N^\theta}\eta(0) (1 \!-\!\eta(1)) \big[ f (\eta^{0,1}) \!-\!  f (\eta)\big] d\nu_p - \int  g (\eta) \eta(1) \big[ f (\eta^{1,0}) \!-\!  f (\eta)\big] d\nu_p \\
	& = \int  f (\eta^{1,0}) \eta(1) \big[ g (\eta) -  g (\eta^{1,0})\big] d\nu_p + \int  f (\eta) \eta(1) \big[ g (\eta^{1,0}) -  g (\eta)\big] d\nu_p \\
	& -  \int  g (\eta^{1,0}) \eta(1) \big[ f (\eta) -  f (\eta^{1,0})\big] d\nu_p - \int  g (\eta) \eta(1) \big[ f (\eta^{1,0}) -  f (\eta)\big] d\nu_p   \;=\; 0\,,
\end{align*}
finishing the proof.
\end{proof}
Let $\mathbf{H}(\mu | \nu_p) = \sum_{\eta\in\Omega_N} \mu(\eta) \log \pfrac{\mu(\eta)}{\nu_p(\eta)}$ be the relative entropy of a probability measure $\mu$ with respect to the invariant state $\nu_p$. 
\begin{proposition}\label{prop:linear_entropy}
Recall \eqref{mu_N_ex}  and let $0<p<1$ be such that $\frac{p}{1-p}=\gamma(\pfrac{0}{N})$.
Then, there exists a constant $K_0<\infty$  such that
\begin{equation}
	\mathbf{H}(\mu_N | \nu_p) \;\leq\; K_0 N\,.
\end{equation}
\end{proposition}

\begin{proof} Note that $\nu_p$ and $\mu_N$ are product measures, and their marginals at the site $x=0$ have Poisson
distribution with the same parameter. Thus,  
\begin{align*}
	\mathbf{H}(\mu_N | \nu_p) 	& \;=\; \sum_{\eta\in\Omega_N} \mu(\eta)\log\bigg[\frac{\prod_{x=1}^N \big[\gamma(\pfrac{x}{N})\eta(x)+(1-\gamma(\pfrac{x}{N}))\big(1-\eta(x)\big)\big]}{\prod_{x=1}^N \big[p\eta(x)+(1-p)\big(1-\eta(x)\big)\big]}\bigg]\\
	& \;\leq\;  \sum_{\eta\in\Omega_N} \mu(\eta)\log\bigg[\frac{1}{\big(p\wedge(1-p)\big)^N}\bigg] \;=\; N\log\frac{1}{ \big(p\wedge(1-p)\big)} \,.      
\end{align*}
\end{proof}

\subsection{Replacement lemma and energy estimate}\label{sec:replacement}
For the regime $\theta=1$, two technical ingredients are needed: first, a replacement lemma that allows to substitute in $L^1$ the time integral  of the occupation at the site $x=1$ by the time integral of the average occupation at the sites $x=1, \ldots, \eps N$. Here,  by $\eps N$ we  mean $\lfloor \eps N\rfloor$. Second, an energy estimate that allows to conclude that any limit of $\bb Q_{\mu_N}^{\theta,N}$ along some subsequence
is concentrated in trajectories whose density with respect to the Lebesgue measure belongs to the Sobolev space $\Sobolev$. These two technical facts are standard for the SSEP, and they can be promptly adapted from \cite{fgn1}, for example. In what follows, we precisely state these two results. Afterwards, we explain the required hypothesis and why they are satisfied for the present model. 
Define the empirical average by
\begin{equation*}
\eta^{\eps N}(x)\;=\;\pfrac{1}{\eps N}\sum_{y=x+1}^{x+\eps N}\eta(y)\,, \qquad x\in \{0,\ldots, (1-\eps)N\}\,.
\end{equation*}
\begin{lemma}[Local Replacement Lemma]\label{lem:replacement}
Let $x\in \{1,\ldots, (1-\eps)N\}$. Then, for any $\theta\geq 0$,
\begin{equation*}
	\limsup_{\eps\to 0}\limsup_{N\to\infty}
	\bb E_{\mu_N}^{\theta, N}\bigg[\,\Big|\int_0^t\{\eta_s(x)-\eta_s^{\eps N}(x)\}\, ds\Big|\,\bigg]\;=\;0\,.
\end{equation*}
\end{lemma}
\begin{proposition}[Energy Estimate]\label{s05}
Let  $\bb Q_*^{\theta}$  be a weak limit of $\bb Q_{\mu_N}^{\theta,N}$ along some subsequence. Then $\bb Q_*^{\theta}$ is concentrated on paths $\pi(t,u)=\rho(t,u) du$ such that there exists a function in $L^2([0,T]\times [0,1])$, denoted by $\partial_u \rho$, such that
\begin{equation*}
	\int_0^T \int_{[0,1]}  \, (\partial_u G) (s, u) \, \rho(s,u)\,du \,ds
	\;=\; -\; \int_0^T \int_{[0,1]} \,   G (s, u)\, (\partial_u \rho) (s, u)\,du \,ds\,, 
\end{equation*}
for all $G$ in $C_c^{0,1}([0,T]\times [0,1])$.
\end{proposition}
Below we explain the necessary hypothesis to prove the two results above (adapting the proofs from \cite{fgn1}, for example) and why they are satisfied here:
\begin{itemize}
\item The dynamics of the sites $x=1, \ldots, N$ must be of SEP type, which is the case.
\item The marginal of the reversible distribution at the sites $x=1, \ldots, N$ must be the product of Bernoulli measures of constant parameter, which is provided by Proposition~\ref{prop:reversible}.
\item The entropy between the distribution of the process at initial time and the reversible invariant state must grow at most linearly, that is, we must have $\mathbf{H}(\mu_N|\nu_p) \leq K_0 N$, which is the content of Proposition~\ref{prop:linear_entropy}. 
\end{itemize}
Let $\<\!\<\cdot, \cdot \>\!\>$ denote the inner product in $L^2([0,T]\times [0,1])$.
The  facts listed above allow to conclude that if  $\bb Q_*^{\theta}$  is a weak limit of $\bb Q_{\mu_N}^{\theta,N}$ along some subsequence, then $\bb Q_*^{\theta}$ is concentrated on trajectories $\{\pi_t = \rho(t,u)du: t\in [0,T]\}$ and 
\begin{equation*}
 \bb E_{\bb Q_*^{\theta}}\Big[\sup_{H}\Big\{\<\!\<\rho, \p_u H\>\!\> - 2\<\!\< H, H \>\!\>\Big\} \Big]\; \leq \; K_0\;<\;\infty \,,
\end{equation*}
where  the supremum is taken over functions $H\in C_c^{0,1}([0,T]\times [0,1])$, and the constant $K_0$ is the one appearing in Proposition~\ref{prop:linear_entropy}. Applying the Riesz Representation Theorem, we conclude the proof of Proposition~\ref{s05}.

\subsection{Tightness}\label{tight}
To ensure the tightness of $\{\pi^{N}_t : 0\le t \le T\}$ it is sufficient to
show the tightness of the real-valued processes $\{\<\pi^{N}_t ,H\> :
0\le t \le T\}$ for $H\in{C[0,1]}$, cf. \ \cite[Chapter~2, Prop.~1.7]{kl}. More than that,  it is enough to show tightness of $\{\<\pi^{N}_t ,H\> :
0\le t \le T\}$ for a dense set of functions in $C[0,1]$ with respect to the uniform topology, since $(C[0,1], \Vert \cdot \Vert_\infty )$  is a separable metric space. For any function $H$, the process
\begin{equation}\label{M_ex}
M^{N}_{t}(H)\;:=\;\<\pi^{N}_{t}, H\>- \<\pi^{N}_{0}, H\>-\int_{0}^{t}\msf L_{N}\<\pi^{N}_{s},H\>\,ds
\end{equation}
is a martingale with respect to the natural filtration $\mathcal{F}_t:=\sigma(\eta_s: s\leq{t})$. In order to prove the tightness of $\{\<\pi^{N}_t ,H\>:t\in[0,T]\}_{N\geq 1}$, we prove tightness
of the sequence of the martingales and the integral terms in the above decomposition. We start with the former.
By the carré-du-champ formula, the quadratic variation of the martingale $M^{N}_{t}(H)$  is given by
\begin{equation*}
\begin{split}
	\<M^{N}(H)\>_t  
	\;=\; &  \int_{0}^{t} \Big(\frac{\alpha}{N^{\theta}} \eta_s(0)\big(1-\eta_s(1)\big) + \eta_s(1)\Big)\Big[H(\pfrac{1}{N})-H(\pfrac{0}{N})\Big]^2ds \\
	&+\int_{0}^{t}\sum_{x=1}^{N-1}  \big(\eta_{s}(x)-\eta_{s}(x+1)\big)^2
	\Big[H(\pfrac{x+1}{N})-H(\pfrac{x}{N})\Big]^2 ds\,.
\end{split}
\end{equation*}
Assume  that $H\in C^2[0,1]$. Similarly to what we have done before in Subsection~\ref{sub:limit}, consider the partial order  in $\Omega_N$ defined by
\begin{equation*}
\eta_1 \lesssim \eta_2 \quad\Longleftrightarrow\quad \eta_1(x)\leq \eta_2(x),\, \forall \,x=0,1,\ldots, N.
\end{equation*}
The initial measure $\mu_N$ defined in \eqref{mu_N_ex} is stochastically dominated by the product measure
\begin{equation}\label{mu_tilde_ex}
\widetilde{\mu}_N \;:=\; \mathrm{Poisson}\Big(\frac{N^\theta}{\alpha} \cdot \Vert \gamma\Vert_\infty\Big)\otimes \bigotimes_{x=1}^N\mathrm{Bernoulli}\big(\Vert \gamma\Vert_\infty\big)\,,
\end{equation}
which is an invariant measure by Proposition~\ref{prop:reversible}. By the attractiveness property and the fact that $\eta(x) \in \{0,1\}$ for $x=1,\ldots, N$, we have
\begin{align*}
\bb E_{\mu_N}^{\theta, N}\big[\<M^{N}(H)\>_t \big] \;\leq\; &  \bb E_{\widetilde{\mu}_N}^{\theta, N}\big[\<M^{N}(H)\>_t\big] \\
\;\leq\; &   \bb E_{\widetilde{\mu}_N}^{\theta, N}\bigg[ \int_{0}^{t} \Big(\frac{\alpha}{N^{\theta}} \eta_s(0) + 1\Big)\Big[H(\pfrac{1}{N})-H(\pfrac{0}{N})\Big]^2ds \bigg]\\
&+\int_{0}^{t}\sum_{x=1}^{N-1}  
\Big[H(\pfrac{x+1}{N})-H(\pfrac{x}{N})\Big]^2 ds\\
\;\leq\; & \bigg[\frac{t}{N^2}\big( \Vert \gamma\Vert_\infty + 1\big)  +\frac{t}{N}  \bigg] \cdot\| H'\|_{\infty}^2\,, 
\end{align*}
which goes to zero as $N\to\infty$. Since $\big(M^{N}_{t}(H)\big)^2 - \<M^{N}(H)\>_t$ is a zero mean martingale (see  \cite[Appendix 1]{kl}), by Doob's inequality, for every $\delta>0$,
\begin{equation}\label{limprobmart}
\lim_{N\rightarrow\infty}\bb P_{\mu_N}^{\theta, N}\bigg[\sup_{0\leq t\leq T} |M^{N}_{t}(H)|>\delta\bigg]\;=\;0\,,
\end{equation}
implying tightness of the sequence $\{M^{N}_{t}(H): t\in[0,T]\}_{N\geq 1}$.
Let us examine the tightness of the integral term $\big\{\int_{0}^{t}\msf L_{N}\<\pi^{N}_{s},H\>\,ds:t\in [0,T]\big\}_{N\geq 1}$ appearing in \eqref{M_ex}. Our goal is to apply Aldous' Criterion:
\begin{theorem}[Aldous' Criterion \cite{Aldous}]
A sequence $\{\msf{X}_t^N: t\in [0,T]\}_{N\geq 1}$ of real-valued processes is tight with respect to the Skorokhod topology of $	D([0,T];\bb R)$ if:
\begin{itemize}
	\item[i)]
	$\displaystyle\lim_{A\rightarrow{+\infty}}\;\limsup_{N\rightarrow{+\infty}}\;\mathbb{P}\Big(\sup_{0\leq{t}\leq{T}}|\msf{X}_{t
	}^N |>A\Big)=0,$
	
	\item[ii)] for any $\varepsilon >0,$
	$\displaystyle\lim_{\delta \to 0} \;\limsup_{N \to {+\infty}} \;\sup_{\zeta \leq \delta} \;\sup_{\tau \in \mc T_T}\;
	\mathbb{P}\Big(|
	\msf{X}_{\tau+\zeta}^N- \msf{X}_{\tau}^N| >\varepsilon\Big)=0,$
\end{itemize}
where $\mc T_T$ is the set of stopping times bounded by $T$.
\end{theorem}
Note that
\begin{align*}
\msf{L}_N \<{\pi}_s^N, {H}\>  
\;=\;&   N \cdot H\big(\pfrac{0}{N}\big)\Big[\eta(1) - \frac{\alpha}{N^\theta}\eta(0)\big(1-\eta(1) \big)\Big]\\ 
& +N \cdot H\big(\pfrac{1}{N}\big)\Big[\frac{\alpha}{N^\theta}\eta(0)\big(1-\eta(1)\big) - 2\eta(1) + \eta(2)\Big]\\
&+ N\cdot \sum_{x=2}^{N-1} {H}\big(\pfrac{x}{N}\big)\big[\eta_s(x-1) - 2\eta_s(x) + \eta_s(x+1)\big]\\
&+ N \cdot H\big(\pfrac{N}{N}\big)\big[-\eta_s(N) + \eta_s(N-1) \big] \,,
\end{align*}
which can be rewritten as
\begin{align}
& \Big\{\frac{\alpha}{N^\theta}\eta_s(0)\big(1-\eta_s(1)\big)-\eta_s(1)\Big\}\cdot   N\Big[H\big(\pfrac{1}{N}\big)-H\big(\pfrac{0}{N}\big)\Big] \label{eq_LN_1}\\ 
&+ \eta_s(1)\cdot N\Big[  H\big(\pfrac{2}{N}\big) - H\big(\pfrac{1}{N}\big) \Big]  \label{eq_LN_2}\\
&+ \frac{1}{N}\sum_{x=2}^{N-1} \eta_s(x)\cdot N^2\Big[ H\big(\pfrac{x+1}{N}\big)+H\big(\pfrac{x-1}{N}\big) - 2H\big(\pfrac{x}{N}\big)  \Big] \label{eq_LN_3}\\
& + \eta_s(N)\cdot N\Big[ H\big(\pfrac{N-1}{N}\big) - H\big(\pfrac{N}{N}\big) \Big]\,. \label{eq_LN_4} 
\end{align}
Since $\eta(x)\in\{0,1\}$ for $x=1,\ldots, N$, we get
\begin{equation}\label{eq:int_reescrita_ex}
\big\vert \msf{L}_N \<{\pi}_s^N, {H}\>\big\vert  \;\leq\;  \Big\{3+\frac{\alpha}{N^\theta}\eta_s(0)\Big\}\cdot  \Vert H'\Vert_\infty +  \Vert H''\Vert_\infty\,.
\end{equation}
\smallskip

 $\bullet$ \textbf{Cases $\theta=1$ and $\theta\in(1,\infty)$.}
Since the system is mass-conservative,
\begin{align}\label{eq:conserva}
\eta_s(0) \;\leq\; \sum_{x=0}^N\eta_s(x) \;=\; \sum_{x=0}^N\eta_0(x)\,.
\end{align}
From \eqref{mu_N_ex},  by an application of the Law of Large Numbers, under $\mu_N$ we have that
\begin{align}\label{eq:LLN}
\frac{\alpha}{N^\theta}\sum_{x=0}^N\eta_0(x)\;\overset{N\to\infty}{\longrightarrow} \;
\begin{cases}
	\gamma(0)+\alpha\int_0^1 \gamma(u)du, &  \text{ if } \theta=1\\
	\gamma(0),& \text{ if } \theta\in(1,\infty)\\
\end{cases}
\end{align}
in $L^1$.
Putting together \eqref{eq:int_reescrita_ex} and \eqref{eq:conserva}  allows us to deduce that 
\begin{align*}
\bb E_{\mu_N}^{\theta, N}\bigg[\Big\vert \int_{\tau}^{\tau+\zeta} \msf{L}_N \<{\pi}_s^N, {H}\> \,ds\Big\vert\bigg]  
&\;\leq\; \bb E_{\mu_N}^{\theta, N}\bigg[ \int_{\tau}^{\tau+\zeta}\big\vert \msf{L}_N \<{\pi}_s^N, {H}\> \big\vert ds   \bigg] \\
&\;\leq\; \zeta \bigg(\bb E_{\mu_N}^{\theta, N}\Big[3+\frac{\alpha}{N^\theta}\sum_{x=0}^N\eta_0(x)\Big]\cdot  \Vert H'\Vert_\infty +  \Vert H''\Vert_\infty\bigg)\,,
\end{align*}
and the tightness  of the integral term  is a consequence of  the inequality above, \eqref{eq:LLN}, Aldous' criterion and Chebyshev's inequality. \medskip

 $\bullet$ \textbf{Case $\theta\in [0,1)$.} Here we cannot use \eqref{eq:LLN}, whose limit is infinite in this case, so a different approach is needed. Recall \eqref{mu_tilde_ex}. By the Cauchy-Schwarz inequality, attractiveness, and invariance of $\widetilde{\mu}_N$, we have that
\begin{align*}
\bb E_{\mu_N}^{\theta, N}\Big[ \Big(\int_r^t \frac{\alpha \eta_s(0)}{N^\theta}ds\Big)^2\Big] & \;\leq\;
\bb E_{\mu_N}^{\theta, N}\Big[ |t-r|\int_r^t \Big(\frac{\alpha \eta_s(0)}{N^\theta}\Big)^2 ds\Big]\\
& \;\leq\;
|t-r| \,\bb E_{\widetilde{\mu}_N}^{\theta, N}\Big[ \int_r^t \Big(\frac{\alpha \eta_s(0)}{N^\theta}\Big)^2 ds\Big] \;\leq\;
|t-r|^2 \Vert \gamma\Vert_\infty^2\,.
\end{align*}
Applying   Kolmogorov-Centsov's  tightness criterion (c.f.\ for instance\break \cite[Ex.~4.11, page 64]{KaratzasShreve1991}) we obtain tightness of the time integral of \eqref{eq_LN_1} in the space $C([0,T];\bb R)$, which implies tightness in the space 
$D([0,T];\bb R)$. Tightness of the time integrals of \eqref{eq_LN_2}, \eqref{eq_LN_3}, and \eqref{eq_LN_4} can be managed for $\theta\geq 1$ via Aldous' criterion.
\medskip

We have therefore deduced tightness of the integral term. 
Due to the fact that the sum of tight sequences of processes is a tight sequence, we complete the proof of tightness of $\{\<\pi^{N}_t ,H\>:t\in[0,T]\}_{N\geq 1}$.

\subsection{Characterization of limit points}\label{subsec:charac_SEP}
Due to the tightness proved in the previous section, the sequence $\bb Q_{\mu_N}^{\theta,N}$ is relatively compact. Denote by $\bb
Q_*^\theta$ a limit along a convergent subsequence, and without loss of generality, denote the convergent subsequence itself  by $\bb Q_{\mu_N}^{\theta,N}$. Our goal is to show that $\bb
Q_*^\theta$ is concentrated on solutions to the respective PDE according to the range of $\theta$.
Recall \eqref{empirical_ex}. 
Since there is at most one particle per site at the sites $x=1, \ldots, N$, it is easy to show that $\bb
Q_*^\theta$ is concentrated on trajectories 
\begin{equation}\label{eq:delta_dirac}
\pi_t(du) \;=\;
\xi(t)\,\delta_0(du)+ \rho(t,u)\, du\,,
\end{equation}
whose  density $\rho(t,\cdot)$ with respect to the Lebesgue measure is non-negative and bounded by
$1$, and $\xi(t)$ denotes the component with respect to the Delta of Dirac at zero. In all  cases below we will use the already proved fact \eqref{limprobmart} that the sequence of  martingales \eqref{M_ex} converges to zero for any $H\in C^2[0,1]$.\medskip

 $\bullet$ \textbf{Case $\theta\in[0,1)$.}
Here, we claim that $\bb Q_*^\theta$ is concentrated on trajectories $\pi_t(du) =\rho(t,u) du$ such that $\rho(t,\cdot)$ is a weak solution to \eqref{heat_Neumann}. In addition to $H\in C^2[0,1]$, assume that $H'(0)=H'(1)=0$.

Since $\theta\in[0,1) $, by the attractiveness property,  the sequence of processes $\{\frac{1}{N}\eta_t(0):t\in[0,T]\}_{N\geq 1}$ and $\{\frac{1}{N}\eta_0(0):t\in[0,T]\}_{N\geq 1}$
converge to zero in probability as $N\to \infty$, so the component of \eqref{eq:delta_dirac} in the Dirac delta measure is $\xi(t)\equiv 0$.  Since $\bb Q^{\theta,N}_{\mu_N}$ is assumed to converge weakly to $\bb Q_*^\theta$, then $\<\pi_t, H\>$ converges to $\<\rho_t, H\>$. By  \eqref{mu_N_ex}, we conclude that $\<\pi_0, H\>$ converges to $\<\gamma, H\>$.

Let us analyze the integral part of \eqref{M_ex}.
Since $\theta\in[0,1) $, from the attractiveness and the fact that $H'(0)=H'(1)=0$, the terms \eqref{eq_LN_1} and \eqref{eq_LN_2} and \eqref{eq_LN_4} converge to zero. Since $H\in C^2[0,1]$, the discrete Laplacian approximates the continuous Laplacian, so the integral in time of \eqref{eq_LN_3} converges to 
\begin{equation*}
\int_0^t \< \rho_s, H''\>\,ds\,.
\end{equation*}
Putting together all facts above, we have deduced that 
\begin{equation*}
\begin{split}
	\bb Q_{*}^\theta \Big[\, \pi_{\,\bigcdot}:\,
	&\big\<\rho_t, H\big\> - \big\<\gamma, H\big\>  -  \int_0^t \big\< \rho_s,   H'' \big\>\, ds
	=0,\,\forall t\in[0,T]\, \Big]=1
\end{split}
\end{equation*}
for any $H\in C^2[0,1]$ such that $H'(0)=H'(1)=0$. Taking a dense set of test functions and then intersecting a countable number of events of probability one, we arrive at
\begin{equation*}
\begin{split}
	\bb Q_{*}^\theta \Big[\, &\pi_{\,\bigcdot}:\,
	\big\<\rho_t, H\big\> - \big\<\gamma, H\big\>  -  \int_0^t \big\< \rho_s,   H'' \big\>\, ds
	=0,\\
	&\forall t\in[0,T]\,,\; \forall H\in C^2[0,1] \text{ such that } H'(0)=H'(1)=0\Big]\;=\;1\,.
\end{split}
\end{equation*}
\medskip

 $\bullet$ \textbf{Case $\theta=1$.}
Here, we claim that $\bb Q_*^\theta$ is concentrated on trajectories $\pi_t(du) =
\xi(t)\delta_0(du)+\rho(t,u) du$ such that $\rho(t,\cdot)$ is a weak solution to \eqref{non_linear_Dirichlet}.
Let $H\in C^2[0,1]$ such that $H(0)=H'(1)=0$. First, note that $H(0)=0$ implies that
\begin{equation*}
\< \pi_t, H\> \;=\; \sum_{x=0}^N \eta_t(x) H\big(\pfrac{x}{N}\big) \;=\; \sum_{x=1}^N \eta_t(x) H\big(\pfrac{x}{N}\big)\,,
\end{equation*}
hence the component $\xi(t)\delta_0(du)$ of \eqref{eq:delta_dirac} does not appear in the limit of $\<\pi_t, H\>$.
As before, the term \eqref{eq_LN_4} converges to zero because $H'(1)=0$ and  the time integral  of \eqref{eq_LN_3} converges to 
\begin{equation*}
\int_0^t \< \rho_s, H''\>\,ds\,.
\end{equation*}
The sum of  \eqref{eq_LN_1} and \eqref{eq_LN_2} is equal to  
\begin{align}\label{eq:eq5.16}
\frac{\alpha}{N^\theta}\eta_s(0)\big(1-\eta_s(1)\big)\cdot   H'(0) 
\end{align}
plus a negligible term of order $O(1/N)$. Denote $\msf{M}(N) = \sum_{x=0}^N\eta_0(x)$. Since $\theta=1$ and using the conservation of the number of particles, we can rewrite \eqref{eq:eq5.16} as 
\begin{align}
&\alpha\bigg[\frac{\msf{M}(N)}{N} - \frac{1}{N}\sum_{x=1}^N \eta_s(x) \bigg]\Big(1-\frac{1}{\eps N}\sum_{x=1}^{\eps N} \eta_s(x)\Big)\cdot   H'(0) \notag\\
&+\alpha\bigg[\frac{\msf{M}(N)}{N} - \frac{1}{N}\sum_{x=1}^N \eta_s(x) \bigg]\Big(-\eta_s(1)+\frac{1}{\eps N}\sum_{x=1}^{\eps N} \eta_s(x)\Big)\cdot   H'(0)\,. \label{2}
\end{align}
From \eqref{mu_N_ex}  and the Law of Large Numbers, we deduce that
\begin{equation*}
\frac{1}{N}\msf{M}(N) \;\longrightarrow\; \frac{1}{\alpha}\gamma(0)+ \int_0^1\gamma(u)\,du
\end{equation*}
 in probability as $N\to\infty$. Recall \eqref{eq:delta_dirac}.
Since $\bb Q^{\theta,N}_{\mu_N}$ is assumed to converge to~$\bb Q_*^\theta$, we have
\begin{equation*}
\frac{1}{N}\sum_{x=1}^N\eta_s(x) \;\longrightarrow\; \<\pi_s, 1\> - \xi(s) = \int_0^1\rho_s(u)\,du
\end{equation*}
and 
\begin{equation*}
\frac{1}{\eps N}\sum_{x=1}^{\eps N}\eta_s(x) = \<\pi_s, {\bf 1}_{(0,\eps)} \> \;\longrightarrow \;  
\frac{1}{\eps}\int_0^\eps\rho_s(u)\,du\,.
\end{equation*}
Recall  the statement of Theorem~\ref{thm:hydro_ex}  for $\theta=1$ about the continuous profile $\gamma:[0,1]\to [0,1]$ that has the additional assumption $0<\gamma(0)<1$. Due to this assumption, Proposition~\ref{prop:linear_entropy} about entropy growth holds, whose statement is a hypothesis in  Propositions~\ref{lem:replacement} and \ref{s05}.  
By the energy estimate given in Proposition~\ref{s05} we know that, under $\bb Q^\theta_*$ with probability one $\rho(t,u)\in \Sobolev$. So, a.e.\ in time the profile $\rho$ is absolutely continuous with respect to the Lebesgue measure, and in particular, it is continuous. Thus, 
\begin{equation*}
\frac{1}{\eps}\int_0^\eps\rho_t(u)\,du \;\longrightarrow\;  \rho_t(0)
\end{equation*}
for a.e.\ $t\in [0,T]$.  By the replacement lemma given in Proposition~\ref{lem:replacement}, the  time integral of the expression \eqref{2} converges to zero in $L^1$ as $N\to \infty$ and then $\eps \downarrow 0$. Putting all these facts together, the integral in time of the sum of \eqref{eq_LN_1} and \eqref{eq_LN_2} converges to
\begin{equation*}
\int_0^t \alpha \big(1-\rho_s(0)\big)\bigg(M- \int_0^1\rho_s(u)du\bigg)   H'(0)  \,ds\,,
\end{equation*}
where $M= \frac{1}{\alpha}\gamma(0)+ \int_0^1\gamma(u)du$. In summary, up to here we have proved that
\begin{equation*}
\begin{split}
	\bb Q_{*}^\theta \Big[\, &\pi_{\,\bigcdot}:\,
	\big\<\rho_t, H\big\> - \big\<\gamma, H\big\>  -  \int_0^t \big\< \rho_s,   H'' \big\>\, ds\\
	&-\int_0^t \alpha \big(1-\rho_s(0)\big)\bigg(M- \int_0^1\rho_s(u)du\bigg)   H'(0)  \,ds =0,\,\forall t\in[0,T]\Big]\;=\;1
\end{split}
\end{equation*}
for all $H\in C^2[0,1]$ such that $H(0)=H'(1)=0$. Intersecting a countable number of events of probability one then yields
\begin{equation*}
\begin{split}
	\bb Q_{*}^\theta \Big[\, &\pi_{\,\bigcdot}:\,
	\big\<\rho_t, H\big\> - \big\<\gamma, H\big\>  -  \int_0^t \big\< \rho_s,   H'' \big\>\, ds\\
	&-\int_0^t \alpha \big(1-\rho_s(0)\big)\bigg(M- \int_0^1\rho_s(u)du\bigg)   H'(0)  \,ds \;=\;0\,,\\
	&\forall t\in[0,T]\,,\; \forall H\in C^2[0,1] \text{ such that } H(0)=H'(1)=0\,\Big]\;=\;1\,.
\end{split}
\end{equation*}

\medskip
 $\bullet$ \textbf{Case $\theta\in(1,\infty)$.}
Here, we claim that $\bb Q_*^\theta$ is concentrated on trajectories $\pi_t(du) =
\xi(t)\delta_0(du)+\rho(t,u) du$ such that $\rho(t,\cdot)$ is a weak solution to \eqref{heat_dirichlet}.
The proof is analogous to the one given in the previous case noting that $\theta\in(1,\infty)$ makes the boundary term vanish in the limit.

\subsection{Uniqueness of weak solutions}\label{sub:uniqueness_SEP}
\begin{proof}[Proof of Proposition~\ref{prop:uniqueness_5.7}]
Let $\rho^1$ and $\rho^2$ be weak solutions of \eqref{heat_Neumann}. For $\xi=\rho^1-\rho^2$, it holds
\begin{equation*}
\<\xi_t, H\>  \;=\; \int_{0}^{t} \<\xi_s, H''\>\, ds\,,
\end{equation*}
for any $H\in C^2[0,1]$ such that $H'(0)=H'(1)=0$. Taking into account $\{\Phi_k\}_{k\geq 0}$,  the complete orthonormal basis of $L^2[0,1]$ given by
\begin{equation*}
\Phi_0(u) \;=\; 1\quad \text{ and } \quad \Phi_k(u) \;=\; \sqrt{2}\cos{(\pi k u)}\,, \quad\forall k\geq 1\,,
\end{equation*}
composed by eigenfunctions of the  Sturm-Liouville problem associated with the Laplacian operator under
Neumann boundary conditions
\begin{equation*}
\begin{cases}
	-f''(u) = \lambda f(u), & \text{ for } u\in (0,1),\\
	f'(0) = f'(1) = 0,
\end{cases}
\end{equation*}
the argument follows the same steps as in the proof of Proposition~\ref{prop:unique_4.3}. 
\end{proof}

The next lemma guarantees that a weak solution to  \eqref{non_linear_Dirichlet} satisfies almost surely in time the boundary condition of the corresponding strong solution.
\begin{lemma}\label{lemma:5.8}
Let $\rho$ be a weak solution to  \eqref{non_linear_Dirichlet}. Then,
\begin{equation*}
	\rho_t(0) \;=\; \alpha \big(1-\rho_t(0)\big) \Big(M - \int_0^1 \rho_t(u)\,du\Big)\,, 
\end{equation*}
for a.e.\ $t\in [0,T]$.
\end{lemma}

\begin{proof}
Let $G_n:\bb R\to \bb R$ be the function defined by $G_n(u) = u\mathbf{1}_{\{u\leq \frac{1}{n}\}} + \frac{1}{n}\mathbf{1}_{\{u> \frac{1}{n}\}}$. 
In Figure~\ref{fig:H} we illustrate $G_n$ and its derivative $G_n'$.

Let $\Psi:\bb R \to \bb R$ be an even $C^\infty$-approximation of identity with support on the interval $[-1, 1]$, see \cite[Appendix~C4, page 629]{E} for an example. Let $\Psi_\eps(u) := \frac{1}{\eps}\Psi(\frac{u}{\eps})$ be the standard mollifier, whose support is contained in $[-\eps,\eps]$. 

Define  the test function $H_n:[0,1]\to \bb R$ by 
\begin{equation}\label{eq:Hn}
	H_n(u) \;=\; \big(G_n * \Psi_{1/n^2}\big)(u)\,,
\end{equation}
the restriction of the convolution to the interval $[0,1]$.  It is straightforward to check that $H_n\in C^\infty[0,1]$, $H_n$ coincides with $G_n$ in the set $[0,1] \setminus [\frac{1}{n} -\frac{1}{n^2}, \frac{1}{n} +\frac{1}{n^2}]$, $H_n(0)=H_n'(1)=0$,  $0\leq H_n \leq \frac{1}{n}$, $0\leq H_n' \leq 1$ and $H_n'(u)=0$ for $u>\frac{1}{n}+\frac{1}{n^2}$. 
\begin{figure}[H]
	\begin{tikzpicture}[scale=1]
		\begin{scope}
			\draw[->] (-1,0) -- (5,0) node[below] {$u$};
			\draw[->] (0,-1) -- (0,3) node[left]{$G_n$};
			\draw[ultra thick] (-0.75, -0.75) -- (0.5, 0.5) -- (4.5,0.5);
			\draw[dashed] (4,0.5) -- (4,0) node[below]{$1$};
			\draw[dashed] (0.5,0.5) -- (0.5,0) node[below]{$\frac{1}{n}$}; 
			\draw[dashed] (0.5,0.5) -- (0,0.5) node[left]{$\frac{1}{n}$}; 
		\end{scope}
		
		\begin{scope}[xshift=7cm]
			\draw[->] (-1,0) -- (5,0) node[below] {$u$};
			\draw[->] (0,-1) -- (0,3) node[left]{$G_n'$};
			\draw[ultra thick] (-0.75, 2.5) -- (0.5, 2.5);
			\draw[ultra thick] (0.5, 0) -- (4.5, 0);
			\draw (4,0.1)--(4,-0.1) node[below]{$1$};
			\draw (0,2.5) node[anchor = north east]{$1$}; 
			\draw[dashed] (0.5,2.5) -- (0.5,0) node[below]{$\frac{1}{n}$}; 
		\end{scope}
		
	\end{tikzpicture}
	\caption{The function $G_n$ and its derivative $G_n'$. The function $H_n := G_n * \Psi_{1/n^2}$
		is a smoothed version of $G_n$ around the point $u=1/n$.}
	\label{fig:H}
\end{figure}

Let $\rho$ be a weak solution to  \eqref{non_linear_Dirichlet}. Hence, the integral equation \eqref{eq:cond_3} holds for any time $t\in[0,T]$ and for any $H\in C^2[0,1]$ such that $H(0)=H'(1)=0$. Since $\rho\in \Sobolev$, applying an integration by parts for Sobolev spaces (see \cite[Lemma~7.1]{fgn1} for instance) and using that $H'(1)=0$, we obtain 
\begin{align*}
		\big\<\rho_t, H\big\> - \big\<\gamma, H\big\>  = & -\int_0^t \big\< \p_u\rho_s,  H' \big\>\, ds - \int_0^t  \rho_s(0) H'(0)\, ds\\
		& + \int_0^t \alpha \big(1-\rho_s(0)\big)\bigg(M- \int_0^1\rho_s(u)du\bigg)   H'(0)  \,ds
\end{align*}
for any time $t\in[0,T]$ and for any $H\in C^2[0,1]$ such that $H(0)=H'(1)=0$.
Choosing the test function as the function $H_n$ defined in \eqref{eq:Hn}, taking the limit as $n\to \infty$ and noting that $H_n$ and $H_n'$ converge to zero in $L^2[0,1]$, we arrive at
\begin{equation*}
	\begin{split}
		\int_0^t  \rho_s(0)\, ds \;=\;\int_0^t \alpha \big(1-\rho_s(0)\big)\bigg(M- \int_0^1\rho_s(u)du\bigg)  \,ds
	\end{split}
\end{equation*}
for any time $t\in[0,T]$, concluding the proof of the lemma.
\end{proof}

\begin{proof}[Proof of Proposition~\ref{prop:unique_non_linear_Dirichlet}]
For ease of notation, denote
\begin{equation*}
m_s \;:=\; M - \int_0^1\rho_s(u)\,du\,,
\end{equation*}
which represents the mass in the reservoir. Since weak solutions of \eqref{non_linear_Dirichlet}  satisfy $0\leq \rho \leq 1$,
Lemma~\ref{lemma:5.8} ensures that $m_s\geq 0$ a.e.\ in time.
Lemma~\ref{lemma:5.8} also implies that
\begin{equation}\label{eq:boundary_rho}
\rho_s(0) \;=\; \frac{\alpha m_s}{1 + \alpha m_s} 
\end{equation}
a.e.\ in time.
Let $\rho^{1}$ and $\rho^{2}$ be two weak solutions of  \eqref{non_linear_Dirichlet} and  denote $\xi =\rho^{1} - \rho^{2} \in \Sobolev$, which satisfies the integral equation
\begin{equation}\label{eq:cond_x}
\begin{split}
	&\big\<\xi_t, H\big\>   =  \int_0^t \big\< \xi_s,  H'' \big\>\, ds\\
	& -\alpha H'(0)\int_0^t\bigg[
	\Big(M -  \int_0^1 \rho^1_s(u)du\Big) \xi_s(0) + \Big(1-\rho^2_s(0)\Big)\int_0^1\xi_s(u)du \bigg]ds
\end{split}
\end{equation}
for any time $t\in[0,T]$ and for any $H\in C^2[0,1]$ such that $H(0)=H'(1)=0$.
Denote also $m^1_s$ and $m^2_s$ the corresponding masses in the reservoir. 
Both $\rho^1$ and $\rho^2$ satisfy  \eqref{eq:boundary_rho}, thus:
\begin{equation}\label{eq:A}
\begin{split}
	\xi_s(0) & \;=\; \frac{\alpha m^1_s}{1 + \alpha m^1_s} - \frac{\alpha m^2_s}{1 + \alpha m^2_s} 
     \;=\; \frac{\alpha A_s}{(1 + \alpha m_s^1)(1 + \alpha m_s^2)}
\end{split}
\end{equation}
a.e.\ in time, where  $A_s := m_s^1 - m_s^2 = - \int_0^1\xi_s(u)du$.
Applying   \eqref{eq:A} and the equality $1 - \rho^2_s(0) = \frac{1}{1 + \alpha m_s^2}$, we can rewrite the expression inside the bracket of \eqref{eq:cond_x}, to get
\begin{align*}
m^1_s \xi_s(0) - (1 - \rho_s^2(0)) A_s  & \;=\;  \frac{\alpha   m_s^1 A_s}{(1 + \alpha m_s^1) (1 + \alpha m_s^2)} - \frac{A_s}{1 + \alpha m_s^2}\\
&\;=\;   \frac{- A_s}{(1 + \alpha m_s^1)(1 + \alpha m_s^2)} \;=\;  -\Gamma_s A_s\,,
\end{align*}
where $\Gamma_s := \frac{1}{(1 + \alpha m_s^1)(1 + \alpha m_s^2)}> 0$.
Therefore, equation \eqref{eq:cond_x} rewrites as
\begin{equation}\label{eq:simplifies}
\big\<\xi_t, H\big\>  \; = \; \int_0^t \big\< \xi_s,  H'' \big\>\, ds + \alpha H'(0) \int_0^t \Gamma_s A_s \,ds\,.
\end{equation}
Recall the eigenfunctions $\{\Psi_k\}_{k\geq 0}$ defined in \eqref{Psi_k}. By \eqref{eq:simplifies},  
\begin{equation*}
c_k'(t) \;=\; -\lambda_k c_k(t) + \alpha \Psi'_k(0) \Gamma_t A_t\,.
\end{equation*}
Considering the energy functional $\mc E(t) := \sum_{k\geq 0}\frac{c_k(t)^2}{2\lambda_k}$, we obtain that
\begin{equation*}
\mc E'(t) \;=\; \sum_{k\geq 0}\frac{c_k(t){c}_k'(t)}{\lambda_k} \;=\; - \sum_{k \geq 0} c_k^2 + \alpha \Gamma_t A_t \sum_{k \geq 0} \frac{\Psi'_k(0)}{\lambda_k} c_k\,.
\end{equation*}
From  \eqref{eq:one}, \eqref{eq:inverse_Psi} and \eqref{eq:a2}, we know that $\frac{\Psi'_k(0)}{\lambda_k} = \frac{2}{\Psi'_k(0)}$ and $\sum_{k \geq 0} \frac{2}{\Psi'_k(0)} c_k = \<\xi_t, 1\>= - A_t$. Since $\Gamma_t > 0$, we obtain
\begin{equation*}
\mc E'(t) \;=\; - \Vert \xi_t\Vert_2^2 - \alpha \Gamma_t A_t^2 \;\leq\; 0\,,
\end{equation*}
leading to the proof of uniqueness of weak solutions to  \eqref{non_linear_Dirichlet}.
\end{proof}

\begin{proof}[Proof of Proposition~\ref{prop:uniqueness_5.10}]
The proof of Proposition~\ref{prop:unique_4.3}  applies here \textit{ipsis litteris}.
\end{proof}

\newcommand{\constante}{0.9}

\section{Computational simulations}\label{sec:sec7}
We provide here some computational simulations of both processes considered in this paper. The code was written in Python with the assistance of the AI assistant Claude (Anthropic), and is available in \cite{simulationcode}. The authors designed the experiments and checked their output against the theoretical results obtained in this paper. The authors made no other use of AI anywhere in this work.

In all cases, the initial profile is given by
$\gamma(x)= \frac{1}{2}+x(1-x)$, and the scaling parameter is $N=64$.
The dots in the graphs correspond to an approximation of $\mathbb{E}[\eta_t(x)]$ via the Law of Large Numbers. Each process is run $25600$ times, and the initial configuration is redrawn for each run. The initial number of particles in the reservoir is of order $N^\theta/\alpha$ and is normalized by this value, in agreement with Remark~\ref{rmk:rmk3.3}.
The empirical mean profiles are shown at macroscopic times $t=0$, $t=0.1$, $t=0.2$, $t=0.5$, and $t=1$, and the colored bars on the left represent the normalized reservoir values at the same times. The legends list each time together with the corresponding reservoir value relative to its initial value.
\begin{figure}[H]
    \centering
    \includegraphics[width=\constante\linewidth]{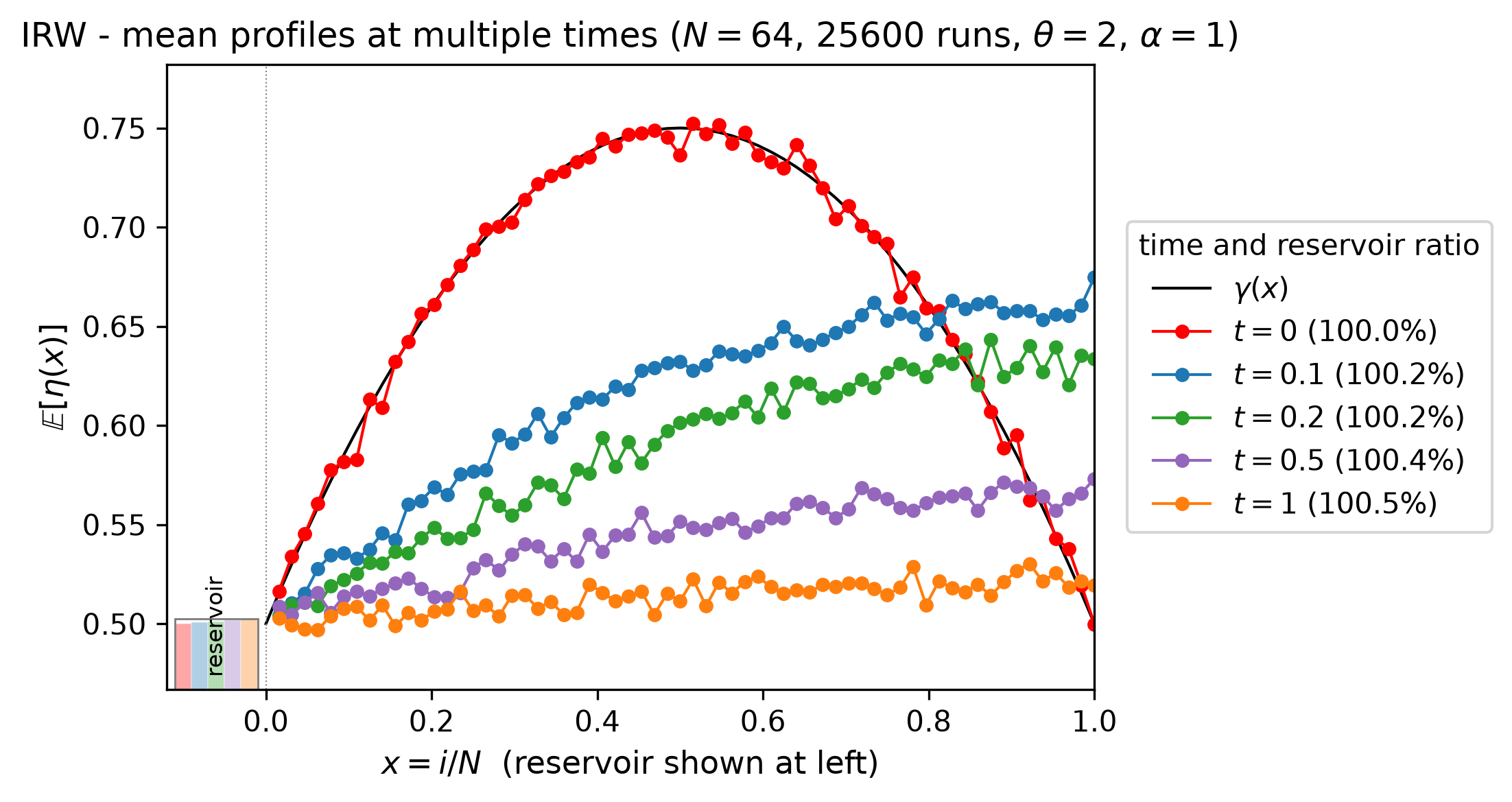}
    \caption{A simulation of IRW in contact with a slow finite reservoir in the \textit{supercritical} regime $\theta=2$. Note that the left boundary above agrees with the expected Dirichlet boundary condition as in \eqref{weak_solution_Dirichlet}. That is, the left-side limit of the profile is constant in time and equal to $\gamma(0)=0.5$, where $\gamma$ is the initial profile.}
    \label{fig:irw_theta_2}
\end{figure}
\begin{figure}[H]
    \centering
    \includegraphics[width=\constante\linewidth]{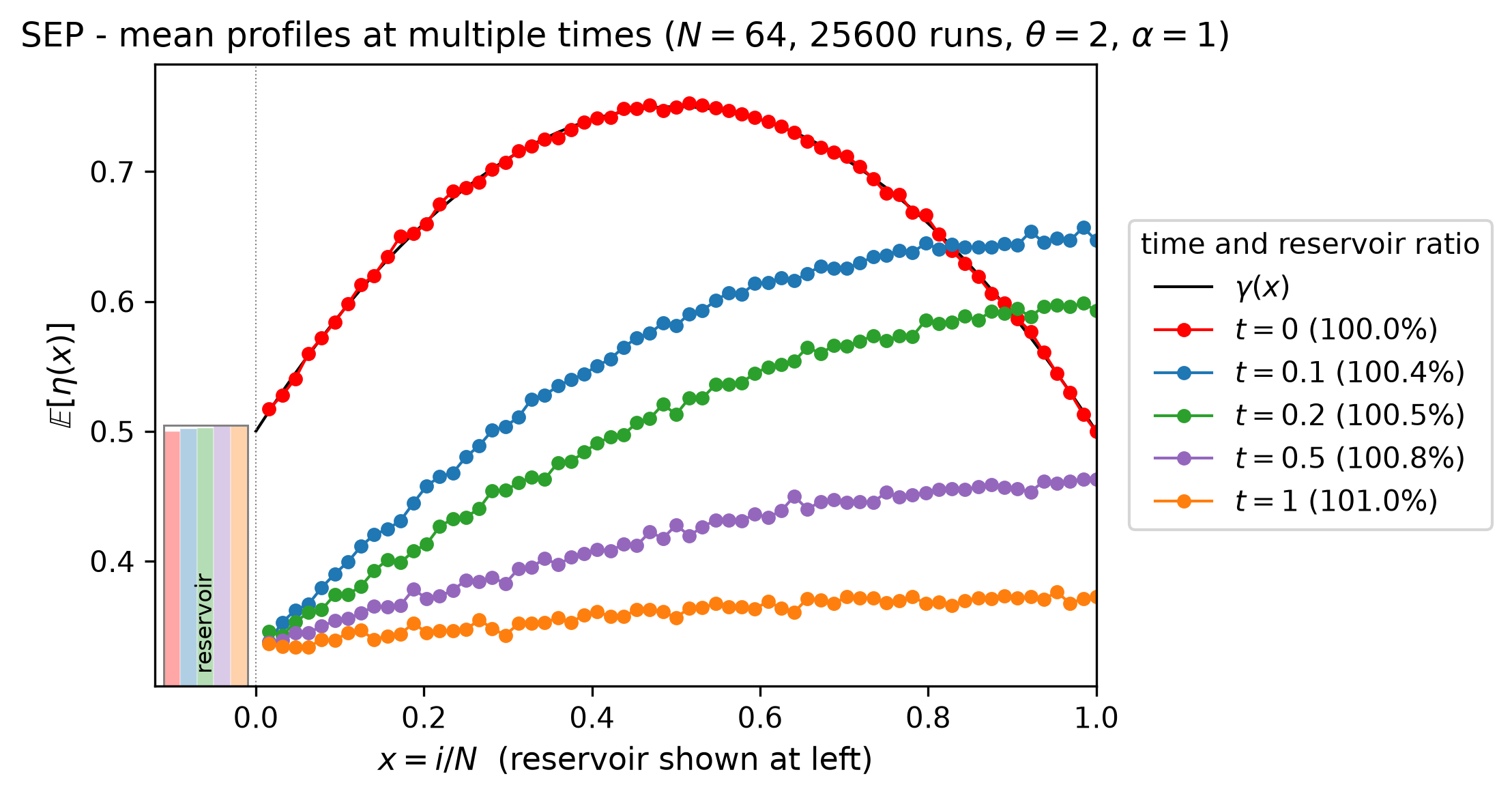}
    \caption{A simulation of SEP in contact with a slow finite reservoir in the \textit{supercritical} regime $\theta=2$. Note that picture agrees with the homogeneous Dirichlet boundary condition, as in \eqref{heat_dirichlet}. Moreover, picture seems to be smoother in comparison with the one in previous picture. This is due to the different nature of processes and different variances of the corresponding invariant measures (which influence the nature of the local equilibrium).}
    \label{fig:sep_theta_2}
\end{figure}
\begin{figure}[H]
    \centering
    \includegraphics[width=\constante\linewidth]{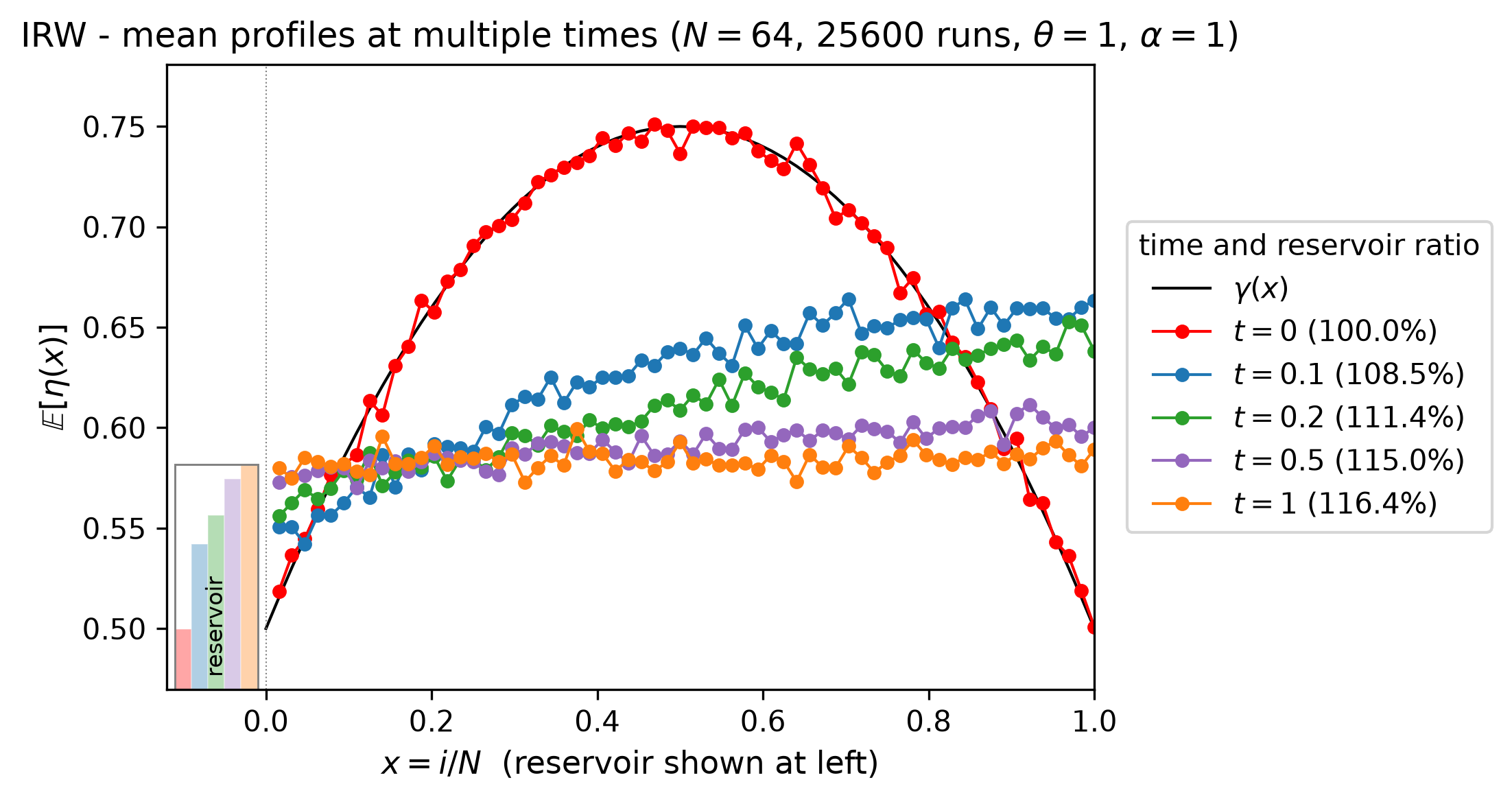}
    \caption{A simulation of IRW in contact with a slow finite reservoir in the \textit{critical} regime $\theta=1$. Note that the both the total (normalized) mass in the reservoir and the total mass in the bulk evolve on time, as expected from the PDE \eqref{weak_solution_non_local}.}
    \label{fig:irw_theta_1}
\end{figure}
\begin{figure}[H]
    \centering
    \includegraphics[width=\constante\linewidth]{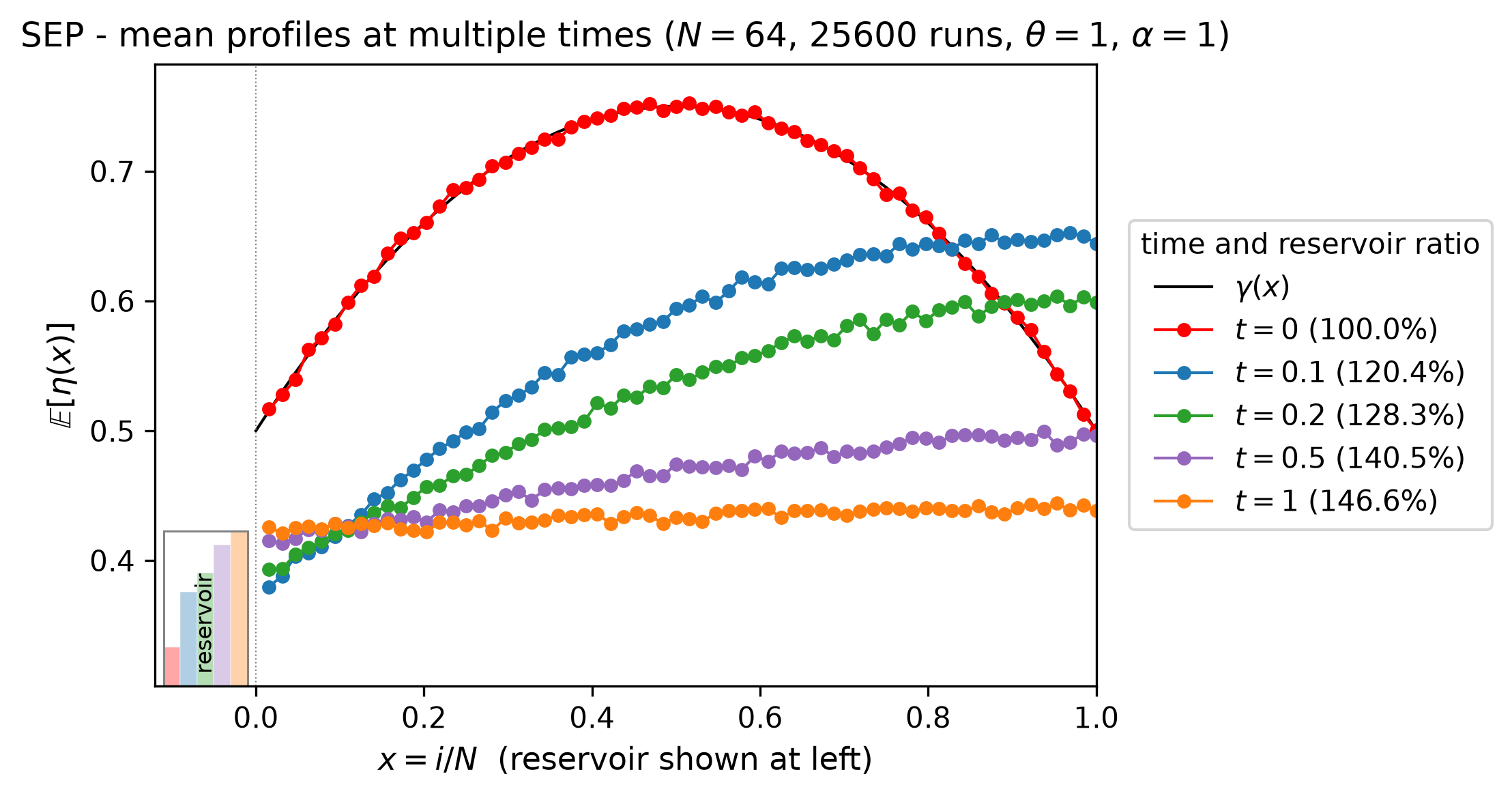}
    \caption{A simulation of SEP in contact with a slow finite reservoir in the \textit{critical} regime $\theta=1$. Note that both the total (normalized) mass in the reservoir and the total mass in the bulk evolve in time, in agreement with the PDE \eqref{non_linear_Dirichlet}.}
    \label{fig:sep_theta_1}
\end{figure}
\begin{figure}[H]
    \centering
    \includegraphics[width=\constante\linewidth]{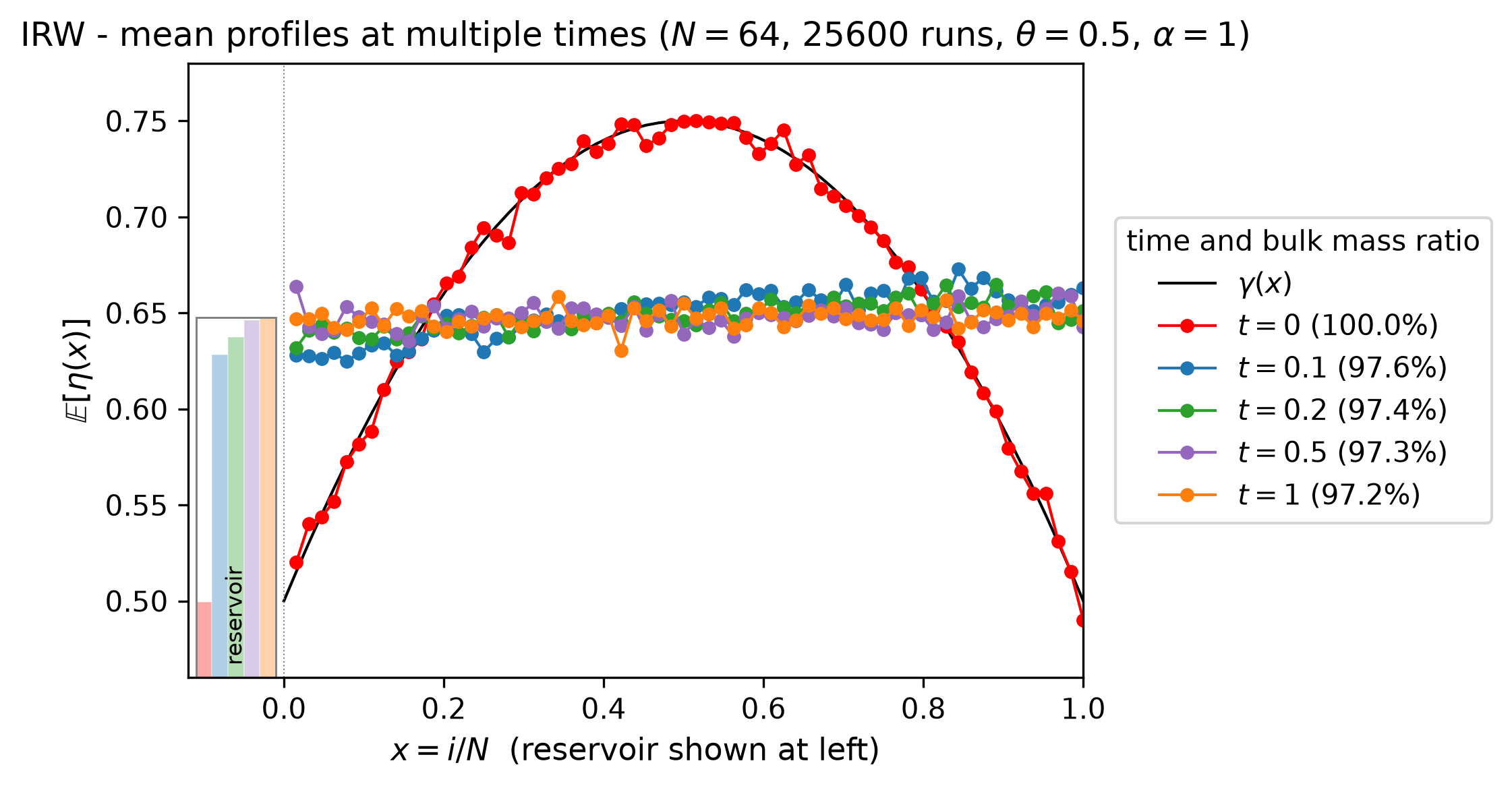}
    \caption{A simulation of IRW in contact with a slow finite reservoir in the \textit{subcritical} regime $\theta=0.5$. Note that the total mass variation in the bulk is not significant, in agreement with the expected Neumann boundary conditions \eqref{heat_Neumann}. On the other hand, the reservoir exhibits a large variation, which is natural in view of the fact that the initial mass in the bulk is of a larger order than the initial mass in the reservoir.}
    \label{fig:irw_theta_0p5}
\end{figure}
\begin{figure}[H]
    \centering
    \includegraphics[width=\constante\linewidth]{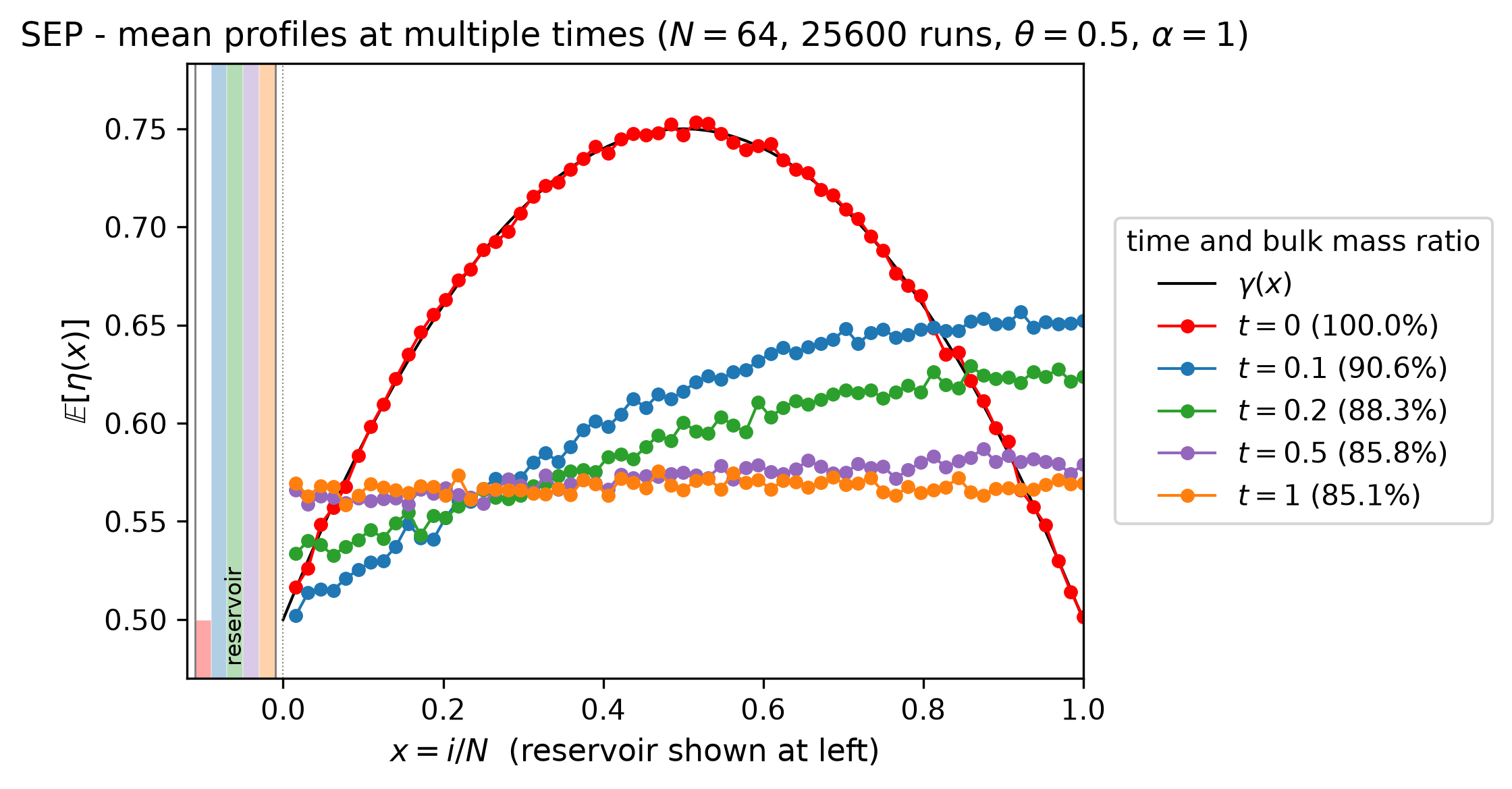}
    \caption{A simulation of SEP in contact with a slow finite reservoir in the \textit{subcritical} regime $\theta=0.5$. Since $\theta=0.5$, the total mass in the bulk is of a larger order than the initial total mass in the reservoir, which explains the abrupt variation of the mass in the reservoir.  On the other hand, here the total mass in the bulk exhibited significant variation over time, which we believe is due to 
$N=64$ not being sufficiently large.}
    \label{fig:sep_theta_0p5}
\end{figure}

\section*{Acknowledgements}
T.F.\ was supported by the National Council for Scientific and Technological Development - CNPq via  Universal Grants (Grant Numbers 406001/2021-9 and 401314/2025-1) and  Bolsa de Produtividade number 306554/2024-0; by FAPESB (EDITAL FAPESB Nº 012/2022 - UNIVERSAL - NºAPP0044/2023).

\bibliography{bibliografia}

\begin{thebibliography}{10}

\bibitem{Aldous}
D.~Aldous.
\newblock {Stopping Times and Tightness}.
\newblock {\em The Annals of Probability}, 6(2):335 -- 340, 1978.

\bibitem{BaldassoMenezesNeumann2017}
R.~Baldasso, O.~Menezes, A.~Neumann, and R.~Rangel.
\newblock Exclusion process with slow boundary.
\newblock {\em Journal of Statistical Physics}, 167(5):1112--1142, 2017.

\bibitem{bertini_landim_mourragui}
L.~Bertini, C.~Landim, and M.~Mourragui.
\newblock {Dynamical large deviations for the boundary driven weakly asymmetric
  exclusion process}.
\newblock {\em The Annals of Probability}, 37(6):2357 -- 2403, 2009.

\bibitem{Boccardo1996}
L.~Boccardo, A.~Dall'Aglio, and T.~Gallou{\"e}t.
\newblock Nonlinear parabolic equations with {W}entzell boundary conditions.
\newblock {\em Journal of Differential Equations}, 127(2):429--456, 1996.

\bibitem{BouleyLandim2024}
A.~Bouley and C.~Landim.
\newblock Dynamical large deviations for boundary driven gradient symmetric
  exclusion processes in mild contact with reservoirs.
\newblock {\em Journal of Statistical Physics}, 191(10):132, 2024.

\bibitem{Daners2008}
D.~Daners.
\newblock Robin and {W}entzell boundary conditions for elliptic equations.
\newblock {\em Journal of Functional Analysis}, 255(11):3163--3197, 2008.

\bibitem{DeMasi_Presutti}
A.~De~Masi and E.~Presutti.
\newblock {\em Mathematical methods for hydrodynamic limits}, volume 1501 of
  {\em Lecture Notes in Mathematics}.
\newblock Springer-Verlag, Berlin, 1991.

\bibitem{Durrett}
R.~Durrett.
\newblock {\em {Probability: Theory and Examples}}, volume~31 of {\em Cambridge
  Series in Statistical and Probabilistic Mathematics}.
\newblock Cambridge University Press, Cambridge, fourth edition, 2010.

\bibitem{efgnt}
D.~Erhard, T.~Franco, P.~Gon\c{c}alves, A.~Neumann, and M.~Tavares.
\newblock Non-equilibrium fluctuations for the ssep with a slow bond.
\newblock {\em Annales de l'Institut Henri Poincar\'e. Probabilit\'es et
  Statistiques}, 56(2):1099--1128, 2020.

\bibitem{EK}
S.~N. Ethier and T.~G. Kurtz.
\newblock {\em Markov processes}.
\newblock Wiley Series in Probability and Mathematical Statistics: Probability
  and Mathematical Statistics. John Wiley \& Sons, Inc., New York, 1986.
\newblock Characterization and convergence.

\bibitem{E}
L.~C. Evans.
\newblock {\em {Partial differential equations}}, volume~19 of {\em {Graduate
  Studies in Mathematics}}.
\newblock American Mathematical Society, Providence, RI, 1998.

\bibitem{Favini2002}
A.~Favini, G.~R. Goldstein, J.~A. Goldstein, and S.~Romanelli.
\newblock The heat equation with generalized {W}entzell boundary conditions.
\newblock {\em Journal of Evolution Equations}, 2(1):1--19, 2002.

\bibitem{Favini2004}
A.~Favini, G.~R. Goldstein, J.~A. Goldstein, and S.~Romanelli.
\newblock C\textsubscript{0}-semigroups generated by second order differential
  operators with {W}entzell boundary conditions.
\newblock {\em Proceedings of the American Mathematical Society},
  132(1):235--244, 2004.

\bibitem{Feller1952}
W.~Feller.
\newblock The parabolic differential equations and associated semi-groups of
  transformations.
\newblock {\em Annals of Mathematics. Second Series}, 55(3):468--519, 1952.

\bibitem{Feller1954}
W.~Feller.
\newblock Diffusion processes in one dimension.
\newblock {\em Annals of Mathematics}, 60(2):199--234, 1954.

\bibitem{Fontes_Isopi_Newman}
L.~R.~G. Fontes, M.~Isopi, and C.~M. Newman.
\newblock {Random walks with strongly inhomogeneous rates and singular
  diffusions: convergence, localization and aging in one dimension}.
\newblock {\em The Annals of Probability}, 30(2):579 -- 604, 2002.

\bibitem{simulationcode}
M.~Franco, T.~Franco, and P.~Gon\c{c}alves.
\newblock Simulation code for ``finite reservoirs lead to wentzell boundary
  conditions for independent random walks and exclusion processes''.
\newblock GitHub repository, 2026.

\bibitem{FrancoGoncalvesLandimNeumann2022}
T.~Franco, P.~Gon{\c{c}}alves, C.~Landim, and A.~Neumann.
\newblock Dynamical large deviations for the boundary driven symmetric
  exclusion process with {R}obin boundary conditions.
\newblock {\em ALEA, Latin American Journal of Probability and Mathematical
  Statistics}, 19:1497--1546, 2022.

\bibitem{fgn1}
T.~Franco, P.~Gon\c{c}alves, and A.~Neumann.
\newblock Hydrodynamical behavior of symmetric exclusion with slow bonds.
\newblock {\em Ann. Inst. H. Poincar\'{e} Probab. Statist.}, 49(2):402--427,
  2013.

\bibitem{fgn2}
T.~Franco, P.~Gon\c{c}alves, and A.~Neumann.
\newblock Phase transition in equilibrium fluctuations of symmetric slowed
  exclusion.
\newblock {\em Stoch. Proc. Appl.}, 123(12):4156--4185, 2013.

\bibitem{fgn3}
T.~Franco, P.~Gon\c{c}alves, and A.~Neumann.
\newblock {Phase transition of a heat equation with {R}obin's boundary
  conditions and exclusion process}.
\newblock {\em Trans. Amer. Math. Soc.}, 367:6131--6158, 2015.

\bibitem{fgschutz2015}
T.~Franco, P.~Gon\c{c}alves, and G.~M. Sch{\"u}tz.
\newblock Scaling limits for the exclusion process with a slow site.
\newblock {\em Stoch. Proc. Appl.}, 2015.

\bibitem{FN}
T.~Franco and A.~Neumann.
\newblock Large deviations for the exclusion process with a slow bond.
\newblock {\em Ann. Appl. Probab.}, 27(6):3547--3587, 2017.

\bibitem{FrancoTavares2019}
T.~Franco and M.~Tavares.
\newblock Hydrodynamic limit for the {SSEP} with a slow membrane.
\newblock {\em Journal of Statistical Physics}, 175(2):233--268, 2019.

\bibitem{Goldstein2006}
G.~R. Goldstein.
\newblock Derivation and physical interpretation of general boundary
  conditions.
\newblock {\em Advances in Differential Equations}, 11(4):419--456, 2006.

\bibitem{jaralandimteixeira}
M.~Jara, C.~Landim, and A.~Teixeira.
\newblock {Quenched scaling limits of trap models}.
\newblock {\em The Annals of Probability}, 39(1):176 -- 223, 2011.

\bibitem{KaratzasShreve1991}
I.~Karatzas and S.~E. Shreve.
\newblock {\em Brownian Motion and Stochastic Calculus}, volume 113 of {\em
  Graduate Texts in Mathematics}.
\newblock Springer, New York, 2 edition, 1991.

\bibitem{kl}
C.~Kipnis and C.~Landim.
\newblock {\em {Scaling limits of interacting particle systems}}, volume 320 of
  {\em {Grundlehren der mathematischen Wissenschaften}}.
\newblock Springer-Verlag Berlin Heidelberg, 1st edition, 1999.

\bibitem{Liang2004}
J.~Liang, R.~Nagel, and T.-J. Xiao.
\newblock {\em Nonautonomous heat equations with generalized {W}entzell
  boundary conditions}, pages 321--331.
\newblock Birkh{\"a}user Basel, Basel, 2004.

\bibitem{Liggett}
T.~M. Liggett.
\newblock {\em {Interacting particle systems}}, volume 276 of {\em {Grundlehren
  der Mathematischen Wissenschaften [Fundamental Principles of Mathematical
  Sciences]}}.
\newblock Springer-Verlag, New York, 1985.

\bibitem{Luo_Trudinger_1991}
Y.~Luo and N.~S. Trudinger.
\newblock Linear second order elliptic equations with {V}enttsel boundary
  conditions.
\newblock {\em Proceedings of the Royal Society of Edinburgh: Section A
  Mathematics}, 118(3--4):193–207, 1991.

\bibitem{DeMasiPresuttiTsagkarogiannisVares2011}
A.~De Masi, E.~Presutti, D.~Tsagkarogiannis, and M.~E. Vares.
\newblock Current reservoirs in the simple exclusion process.
\newblock {\em Journal of Statistical Physics}, 144(6):1151--1170, 2011.

\bibitem{Mourragui_mixed}
M.~Mourragui, E.~Saada, and S.~Velasco.
\newblock {Hydrodynamic and hydrostatic limit for a generalized contact process
  with mixed boundary conditions}.
\newblock {\em Electronic Journal of Probability}, 28(none):1 -- 44, 2023.

\bibitem{Titchmarsh}
E.~C. Titchmarsh.
\newblock {\em Eigenfunction Expansions Associated with Second-Order
  Differential Equations}, volume~I.
\newblock Oxford University Press, 2nd edition, 1962.

\bibitem{Vasquez}
J.~L. Vázquez and E.~Vitillaro.
\newblock Heat equation with dynamical boundary conditions of
  reactive–diffusive type.
\newblock {\em Journal of Differential Equations}, 250(4):2143--2161, 2011.

\bibitem{Warma2005Parabolic}
M.~Warma.
\newblock Regularity and well-posedness for parabolic equations with general
  {W}entzell boundary conditions.
\newblock {\em Journal of Evolution Equations}, 5(1):167--185, 2005.

\bibitem{Warma2005Trace}
M.~Warma.
\newblock Wentzell boundary conditions in $\mathbb{R}^n$ and trace theorems.
\newblock {\em Advances in Differential Equations}, 10(10):1189--1232, 2005.

\bibitem{Wentzell1956Semigroups}
A.~D. Wentzell.
\newblock Semigroups of operators corresponding to a generalized differential
  operator of second order.
\newblock {\em Doklady Akademii Nauk SSSR}, 111:269--272, 1956.
\newblock In Russian.

\bibitem{Venttsel1959}
A.~D. Wentzell.
\newblock On boundary conditions for multidimensional diffusion processes.
\newblock {\em Teor. Veroyatnost. i Primenen.}, 4(2):172--185, 1959.

\end{thebibliography}
\bibliographystyle{plain}

\end{document}